\pdfoutput=1
\documentclass{amsart}

\usepackage{amssymb,amsmath,stmaryrd,mathrsfs}
\def\definetac{\newif\iftac}    % Can't define a \newif inside another \if!
\ifx\tactrue\undefined
  \definetac
  \ifx\state\undefined\tacfalse\else\tactrue\fi
\fi
\iftac\else\usepackage{amsthm}\fi
\usepackage[all,2cell,pdf]{xy}
\UseAllTwocells
\usepackage{enumitem}
\usepackage{xcolor}
\definecolor{darkgreen}{rgb}{0,0.45,0} 
\usepackage[pagebackref,colorlinks,citecolor=darkgreen,linkcolor=darkgreen]{hyperref}
\usepackage{mathtools}          % for all sorts of things
\usepackage{tikz}
\usepackage{braket}             % for \Set, etc.

\usepackage{url}                % for citations to web sites
\usepackage{xspace}             % put spaces after a \command in text

\makeatletter
\let\ea\expandafter

\def\mdef#1#2{\ea\ea\ea\gdef\ea\ea\noexpand#1\ea{\ea\ensuremath\ea{#2}\xspace}}
\def\alwaysmath#1{\ea\ea\ea\global\ea\ea\ea\let\ea\ea\csname your@#1\endcsname\csname #1\endcsname
  \ea\def\csname #1\endcsname{\ensuremath{\csname your@#1\endcsname}\xspace}}

\DeclareRobustCommand\widecheck[1]{{\mathpalette\@widecheck{#1}}}
\def\@widecheck#1#2{%
    \setbox\z@\hbox{\m@th$#1#2$}%
    \setbox\tw@\hbox{\m@th$#1%
       \widehat{%
          \vrule\@width\z@\@height\ht\z@
          \vrule\@height\z@\@width\wd\z@}$}%
    \dp\tw@-\ht\z@
    \@tempdima\ht\z@ \advance\@tempdima2\ht\tw@ \divide\@tempdima\thr@@
    \setbox\tw@\hbox{%
       \raise\@tempdima\hbox{\scalebox{1}[-1]{\lower\@tempdima\box
\tw@}}}%
    {\ooalign{\box\tw@ \cr \box\z@}}}

\newcount\foreachcount

\def\foreachletter#1#2#3{\foreachcount=#1
  \ea\loop\ea\ea\ea#3\@alph\foreachcount
  \advance\foreachcount by 1
  \ifnum\foreachcount<#2\repeat}

\def\foreachLetter#1#2#3{\foreachcount=#1
  \ea\loop\ea\ea\ea#3\@Alph\foreachcount
  \advance\foreachcount by 1
  \ifnum\foreachcount<#2\repeat}

\def\definescr#1{\ea\gdef\csname s#1\endcsname{\ensuremath{\mathscr{#1}}\xspace}}
\foreachLetter{1}{27}{\definescr}
\def\definecal#1{\ea\gdef\csname c#1\endcsname{\ensuremath{\mathcal{#1}}\xspace}}
\foreachLetter{1}{27}{\definecal}
\def\definebold#1{\ea\gdef\csname b#1\endcsname{\ensuremath{\mathbf{#1}}\xspace}}
\foreachLetter{1}{27}{\definebold}
\def\definebb#1{\ea\gdef\csname l#1\endcsname{\ensuremath{\mathbb{#1}}\xspace}}
\foreachLetter{1}{27}{\definebb}
\def\definefrak#1{\ea\gdef\csname f#1\endcsname{\ensuremath{\mathfrak{#1}}\xspace}}
\foreachletter{1}{9}{\definefrak} % \fi is something else!
\foreachletter{10}{27}{\definefrak}
\def\definebar#1{\ea\gdef\csname #1bar\endcsname{\ensuremath{\overline{#1}}\xspace}}
\foreachLetter{1}{27}{\definebar}
\foreachletter{1}{8}{\definebar} % \hbar is something else!
\foreachletter{9}{15}{\definebar} % \obar is something else!
\foreachletter{16}{27}{\definebar}
\def\definetil#1{\ea\gdef\csname #1til\endcsname{\ensuremath{\widetilde{#1}}\xspace}}
\foreachLetter{1}{27}{\definetil}
\foreachletter{1}{27}{\definetil}
\def\definehat#1{\ea\gdef\csname #1hat\endcsname{\ensuremath{\widehat{#1}}\xspace}}
\foreachLetter{1}{27}{\definehat}
\foreachletter{1}{27}{\definehat}
\def\definechk#1{\ea\gdef\csname #1chk\endcsname{\ensuremath{\widecheck{#1}}\xspace}}
\foreachLetter{1}{27}{\definechk}
\foreachletter{1}{27}{\definechk}
\def\defineul#1{\ea\gdef\csname u#1\endcsname{\ensuremath{\underline{#1}}\xspace}}
\foreachLetter{1}{27}{\defineul}
\foreachletter{1}{27}{\defineul}

\def\autofmt@n#1\autofmt@end{\mathrm{#1}}
\def\autofmt@b#1\autofmt@end{\mathbf{#1}}
\def\autofmt@l#1#2\autofmt@end{\mathbb{#1}\mathsf{#2}}
\def\autofmt@c#1#2\autofmt@end{\mathcal{#1}\mathit{#2}}
\def\autofmt@s#1#2\autofmt@end{\mathscr{#1}\mathit{#2}}
\def\autofmt@f#1\autofmt@end{\mathsf{#1}}
\def\autofmt@u#1\autofmt@end{\underline{\smash{\mathsf{#1}}}}
\def\autofmt@U#1\autofmt@end{\underline{\underline{\smash{\mathsf{#1}}}}}
\def\autofmt@h#1\autofmt@end{\widehat{#1}}
\def\autofmt@r#1\autofmt@end{\overline{#1}}
\def\autofmt@t#1\autofmt@end{\widetilde{#1}}
\def\autofmt@k#1\autofmt@end{\check{#1}}

\def\auto@drop#1{}
\def\autodef#1{\ea\ea\ea\@autodef\ea\ea\ea#1\ea\auto@drop\string#1\autodef@end}
\def\@autodef#1#2#3\autodef@end{%
  \ea\def\ea#1\ea{\ea\ensuremath\ea{\csname autofmt@#2\endcsname#3\autofmt@end}\xspace}}
\def\autodefs@end{blarg!}
\def\autodefs#1{\@autodefs#1\autodefs@end}
\def\@autodefs#1{\ifx#1\autodefs@end%
  \def\autodefs@next{}%
  \else%
  \def\autodefs@next{\autodef#1\@autodefs}%
  \fi\autodefs@next}

\DeclareSymbolFont{bbold}{U}{bbold}{m}{n}
\DeclareSymbolFontAlphabet{\mathbbb}{bbold}

\newcommand{\bbone}{\ensuremath{\mathbbb{1}}\xspace}

\mdef\delbar{\overline{\partial}}

\mdef\hf{\textstyle\frac12 }
\mdef\thrd{\textstyle\frac13 }
\mdef\qtr{\textstyle\frac14 }

\newcommand{\op}{^{\mathrm{op}}}

\newcommand{\coop}{^{\mathrm{coop}}}

\SelectTips{cm}{}
\newdir{ >}{{}*!/-10pt/@{>}}    % extra spacing for tail arrows in XYpic
\newcommand{\pushoutcorner}[1][dr]{\save*!/#1+1.2pc/#1:(1,-1)@^{|-}\restore}
\newcommand{\pullbackcorner}[1][dr]{\save*!/#1-1.2pc/#1:(-1,1)@^{|-}\restore}

\mdef\Id{\mathrm{Id}}
\mdef\id{\mathrm{id}}
\alwaysmath{ell}
\alwaysmath{infty}
\alwaysmath{odot}
\def\frc#1/#2.{\frac{#1}{#2}}   % \frc x^2+1 / x^2-1 .
\mdef\ten{\mathrel{\otimes}}

\mdef\sqten{\mathrel{\boxtimes}}

\DeclareRobustCommand\widecheck[1]{{\mathpalette\@widecheck{#1}}}
\def\@widecheck#1#2{%
    \setbox\z@\hbox{\m@th$#1#2$}%
    \setbox\tw@\hbox{\m@th$#1%
       \widehat{%
          \vrule\@width\z@\@height\ht\z@
          \vrule\@height\z@\@width\wd\z@}$}%
    \dp\tw@-\ht\z@
    \@tempdima\ht\z@ \advance\@tempdima2\ht\tw@ \divide\@tempdima\thr@@
    \setbox\tw@\hbox{%
       \raise\@tempdima\hbox{\scalebox{1}[-1]{\lower\@tempdima\box
\tw@}}}%
    {\ooalign{\box\tw@ \cr \box\z@}}}

\DeclareMathOperator\colim{colim}

\DeclareMathOperator\Aut{Aut}
\DeclareMathOperator\End{End}
\DeclareMathOperator\Hom{Hom}

\mdef\we{\overset{\sim}{\longrightarrow}}
\mdef\leftwe{\overset{\sim}{\longleftarrow}}

\let\xto\xrightarrow

\def\rightarrowtailfill@{\arrowfill@{\Yright\joinrel\relbar}\relbar\rightarrow}
\newcommand\xrightarrowtail[2][]{\ext@arrow 0055{\rightarrowtailfill@}{#1}{#2}}

\def\twoheadrightarrowfill@{\arrowfill@{\relbar\joinrel\relbar}\relbar\twoheadrightarrow}
\newcommand\xtwoheadrightarrow[2][]{\ext@arrow 0055{\twoheadrightarrowfill@}{#1}{#2}}

\def\slashedarrowfill@#1#2#3#4#5{%
  $\m@th\thickmuskip0mu\medmuskip\thickmuskip\thinmuskip\thickmuskip
   \relax#5#1\mkern-7mu%
   \cleaders\hbox{$#5\mkern-2mu#2\mkern-2mu$}\hfill
   \mathclap{#3}\mathclap{#2}%
   \cleaders\hbox{$#5\mkern-2mu#2\mkern-2mu$}\hfill
   \mkern-7mu#4$%
}
\def\rightslashedarrowfill@{%
  \slashedarrowfill@\relbar\relbar\mapstochar\rightarrow}
\newcommand\xslashedrightarrow[2][]{%
  \ext@arrow 0055{\rightslashedarrowfill@}{#1}{#2}}
\mdef\hto{\xslashedrightarrow{}}
\mdef\htoo{\xslashedrightarrow{\quad}}

\def\toiso{\xto{\smash{\raisebox{-.5mm}{$\scriptstyle\sim$}}}}

\long\def\my@drawfill#1#2;{%
\@skipfalse
\fill[#1,draw=none] #2;
\@skiptrue
\draw[#1,fill=none] #2;
}
\newif\if@skip
\newcommand{\skipit}[1]{\if@skip\else#1\fi}
\newcommand{\drawfill}[1][]{\my@drawfill{#1}}

\newif\ifhyperref
\@ifpackageloaded{hyperref}{\hyperreftrue}{\hyperreffalse}
\iftac
  \let\your@state\state
  \def\state#1{\gdef\currthmtype{#1}\your@state{#1}}
  \let\your@staterm\staterm
  \def\staterm#1{\gdef\currthmtype{#1}\your@staterm{#1}}
  \let\defthm\newtheorem
  \def\currthmtype{}
  \ifhyperref
    \def\autoref#1{\ref*{label@name@#1}~\ref{#1}}
  \else
    \def\autoref#1{\ref{label@name@#1}~\ref{#1}}
  \fi
  \AtBeginDocument{%
    \let\old@label\label%
    \def\label#1{%
      {\let\your@currentlabel\@currentlabel%
        \edef\@currentlabel{\currthmtype}%
        \old@label{label@name@#1}}%
      \old@label{#1}}
  }
\else
  \ifhyperref
    \def\defthm#1#2{%
      \newtheorem{#1}{#2}[section]%
      \expandafter\def\csname #1autorefname\endcsname{#2}%
      \expandafter\let\csname c@#1\endcsname\c@thm}
  \else
    \def\defthm#1#2{\newtheorem{#1}[thm]{#2}}
    \ifx\SK@label\undefined\let\SK@label\label\fi
    \let\old@label\label
    \let\your@thm\@thm
    \def\@thm#1#2#3{\gdef\currthmtype{#3}\your@thm{#1}{#2}{#3}}
    \def\currthmtype{}
    \def\label#1{{\let\your@currentlabel\@currentlabel\def\@currentlabel%
        {\currthmtype~\your@currentlabel}%
        \SK@label{#1@}}\old@label{#1}}
    \def\autoref#1{\ref{#1@}}
  \fi
\fi

\newtheorem{thm}{Theorem}[section]

\defthm{cor}{Corollary}
\defthm{prop}{Proposition}
\defthm{lem}{Lemma}
\defthm{sch}{Scholium}
\defthm{assume}{Assumption}
\defthm{claim}{Claim}
\defthm{conj}{Conjecture}
\defthm{hyp}{Hypothesis}
\defthm{fact}{Fact}
\defthm{quest}{Question}
\iftac\theoremstyle{plain}\else\theoremstyle{definition}\fi
\defthm{defn}{Definition}
\defthm{notn}{Notation}
\iftac\theoremstyle{plain}\else\theoremstyle{remark}\fi
\defthm{rmk}{Remark}
\defthm{eg}{Example}
\defthm{egs}{Examples}
\defthm{ex}{Exercise}
\defthm{ceg}{Counterexample}
\defthm{con}{Construction}
\defthm{conv}{Convention}
\defthm{warn}{Warning}
\defthm{digr}{Digression}
\defthm{per}{Perspective}
\defthm{disc}{Discussion}
\defthm{term}{Terminology}
\defthm{heur}{Heuristics}

\def\thmqedhere{\expandafter\csname\csname @currenvir\endcsname @qed\endcsname}

\setitemize[1]{leftmargin=2em}
\setenumerate[1]{leftmargin=*}

\iftac
  \let\c@equation\c@subsection
\else
  \let\c@equation\c@thm
\fi
\numberwithin{equation}{section}

\@ifpackageloaded{mathtools}{\mathtoolsset{showonlyrefs,showmanualtags}}{}

\alwaysmath{alpha}
\alwaysmath{beta}
\alwaysmath{gamma}
\alwaysmath{Gamma}
\alwaysmath{delta}
\alwaysmath{Delta}
\alwaysmath{epsilon}
\mdef\ep{\varepsilon}
\alwaysmath{zeta}
\alwaysmath{eta}
\alwaysmath{theta}
\alwaysmath{Theta}
\alwaysmath{iota}
\alwaysmath{kappa}
\alwaysmath{lambda}
\alwaysmath{Lambda}
\alwaysmath{mu}
\alwaysmath{nu}
\alwaysmath{xi}
\alwaysmath{pi}
\alwaysmath{rho}
\alwaysmath{sigma}
\alwaysmath{Sigma}
\alwaysmath{tau}
\alwaysmath{upsilon}
\alwaysmath{Upsilon}
\alwaysmath{phi}
\alwaysmath{Pi}
\alwaysmath{Phi}
\mdef\ph{\varphi}
\alwaysmath{chi}
\alwaysmath{psi}
\alwaysmath{Psi}
\alwaysmath{omega}
\alwaysmath{Omega}

\makeatother

\tikzset{lab/.style={auto,font=\scriptsize}} % arrow labels
\usetikzlibrary{arrows}

\usepackage[all]{xy}

\usepackage{xcolor}% provides \colorlet
\definecolor{fxnote}{rgb}{1.0000,0.0000,0.0000}
\colorlet{fxnotebg}{yellow}

\newcommand{\tr}{\ensuremath{\operatorname{tr}}}
\newcommand{\tw}{\ensuremath{\operatorname{tw}}}
\newcommand{\D}{\sD}
\newcommand{\E}{\sE}

\newcommand{\V}{\sV}

\autodefs{\cDER\cCat\cGrp\cCAT\cDir\cMONCAT\bSet\cProf\bCat
\cPro\cMONDER\cBICAT\ncomp\nid\ncoker\niso\nIm\sProf\ncyl\fC\fZ\fT\nhocolim\nholim\cPsNat\ncoll\nCh\bsSet\cAlg\cIm\cI\cP}

\def\ho{\mathscr{H}\!\mathit{o}\xspace}

\let\oldboxtimes\boxtimes
\def\boxtimes{\mathrel{\oldboxtimes}}

\newcommand{\fib}{\mathsf{fib}}
\newcommand{\cof}{\mathsf{cof}}

\def\ccsub{_{\mathrm{cc}}}
\def\pdh(#1,#2){\llbracket #1,#2\rrbracket}
\def\ldh(#1,#2){\llbracket #1,#2\rrbracket\ccsub}
\def\pend(#1){\pdh(#1,#1)}
\def\lend(#1){\ldh(#1,#1)}

\def\DTl#1#2#3#4#5#6#7{%
  \xymatrix@C=3pc{{#1} \ar[r]^-{#2} &
    {#3} \ar[r]^-{#4} &
    {#5} \ar[r]^-{#6} &
    {#7}
  }}

\newsavebox{\tvabox}
\savebox\tvabox{\hspace{1mm}\begin{tikzpicture}[>=latex',baseline={(0,-.18)}]
  \draw[->] (0,.1) -- +(1,0);
  \node at (.5,0) {$\scriptscriptstyle\bot$};
  \draw[->] (1,-.1) -- +(-1,0);
  \draw[->] (1,-.2) -- +(-1,0);
\end{tikzpicture}\hspace{1mm}}

\newcommand{\exx}{\mathrm{ex}}

\newcommand{\Mod}{\mathrm{Mod}}

\newcommand{\dia}{\mathrm{dia}}

\DeclareMathOperator{\Ext}{Ext}

\def\cSp{\ensuremath{\mathcal{S}\!\mathit{p}}\xspace}

\newcommand{\tcof}{\mathsf{tcof}}
\newcommand{\tfib}{\mathsf{tfib}}
\newcommand{\cok}{\mathrm{cok}}

\newcommand{\cube}[1]{\square^{#1}}
\newcommand{\cotr}[2]{\cube{#1},\mathrm{cotr}_{#2}}
\renewcommand{\tr}[2]{\cube{#1},\mathrm{tr}_{#2}}
\newcommand{\cotrmor}[1]{\mathrm{cotr}_{#1}}
\newcommand{\chunk}[1]{\mathrm{chunk}(\cube{{#1}})}

\title{Abstract cubical homotopy theory}
\author{Falk Beckert and Moritz Groth}
\date{\today}

\thanks{%
The first named author was supported by grant DFG~SPP~1786 from the German Science Foundation (DFG)
}

\begin{document}

\begin{abstract}
Triangulations and higher triangulations axiomatize the calculus of derived cokernels when applied to strings of composable morphisms. While there are no cubical versions of (higher) triangulations, in this paper we use coherent diagrams to develop some aspects of a rich cubical calculus. Applied to the models in the background, this enhances the typical examples of triangulated and tensor-triangulated categories.

The main players are the cardinality filtration of $n$-cubes, the induced interpolation between cocartesian and strongly cocartesian $n$-cubes, and the yoga of iterated cone constructions. In the stable case, the representation theories of chunks of $n$-cubes are related by compatible strong stable equivalences and admit a global form of Serre duality. As sample applications, we use these Serre equivalences to express colimits in terms of limits and to relate the abstract representation theories of chunks by infinite chains of adjunctions.

On a more abstract side, along the way we establish a general decomposition result for colimits, which specializes to the classical Bousfield--Kan formulas. We also include a short discussion of abstract formulas and their compatibility with morphisms, leading to the idea of universal formulas in monoidal homotopy theories.
\end{abstract}

\maketitle

\tableofcontents

\section{Motivation via higher triangulations}
\label{sec:motiv}

Since their introduction in the 1960's \cite{verdier:thesis,verdier:derived,puppe:stabil}, \emph{triangulated categories} have proved extremely useful in various areas of mathematics. One of their nice features is their ubiquity. Triangulated categories arise naturally in algebraic geometry as derived categories of schemes (\cite{verdier:thesis,verdier:derived} or \cite{huybrechts:fourier}), in representation theory as derived categories of algebras (see \cite{happel:triangulated} or \cite{angeleri-happel-krause:handbook}), in modular representation theory as stable module categories \cite{benson-rickard-carlson:thick-stmod}, in homotopy theory as homotopy categories of spectra or related (stable model) categories (\cite{vogt:boardman} or \cite{hovey:model}), and in algebraic analysis (\cite{schapira:analysts}). A nice survey on this ubiquity of triangulated categories can be found in \cite{hjr:triangulated}, while an additional nice survey from the tensor-triangulated perspective is in \cite{balmer:TTG}.

For reasons that will become clear in a moment, we revisit the main reasoning behind the axioms of triangulated categories. For concreteness, let us consider the case of the derived category $D(\cA)$ of an abelian category $\cA$. The basic idea is to try to capture at the level of $D(\cA)$ some shadows of the calculus of \emph{derived cokernel constructions} of chain maps --- as encoded by distinguished triangles. A first key axiom asks for the fact that all morphisms in $D(\cA)$ can be extended to distinguished triangles. And a second key axiom asks for the weak functoriality of this passage to distinguished triangles (for simplicity, we ignore the remaining axioms). 

Correspondingly, a crucial step in one of the proofs of the \emph{existence} of the classical Verdier triangulation on $D(\cA)$ (and similarly in other examples \cite{franke:adams,maltsiniotis:seminar,groth:ptstab}) consists of the following. Starting with a morphism $f\colon C\to D$ in $D(\cA)$, we can find a chain map $F\colon X\to Y$ in $\nCh(\cA)$ such that its image under the localization functor $\gamma\colon\nCh(\cA)\to D(\cA)$ is isomorphic to $f$. This allows us to apply functorial iterated derived cokernel constructions to $F$ and then to pass back to $D(\cA)$ in order to obtain a distinguished triangle extending $f$.

To put this more abstractly, let $[1]$ denote the partially ordered set $(0<1)$ considered as a category, and let us use exponential notation for categories of diagrams. Presented this way, the existence of triangles relies in an essential way on the fact that the forgetful functor
\begin{equation}\label{eq:dia-[1]}
\dia_{[1]}\colon D(\cA^{[1]})\to D(\cA)^{[1]}
\end{equation}
which sends a chain map to its image under the localization functor is \emph{essentially surjective}. Similarly, the weak functoriality of distinguished triangles in $D(\cA)$ is a consequence of \eqref{eq:dia-[1]} also being \emph{full} (and this functor is hence an \emph{epivalence} in the sense of Keller \cite{keller:epivalent}).

These nice properties of \eqref{eq:dia-[1]} extend to more general shapes. For instance, if we denote by $[n]$ the partially ordered set $(0<1<\ldots<n)$, then the functor
\begin{equation}\label{eq:dia-[n]}
\dia_{[n]}\colon D(\cA^{[n]})\to D(\cA)^{[n]}
\end{equation}
which sends a string of composable chain maps to its image under the localization functor is an epivalence (it is full and essentially surjective). And in fact, these properties of \eqref{eq:dia-[n]} in the case of $n=2$ yield a clean proof not only of the existence of octahedron diagrams in $D(\cA)$ but also of their weak functorial dependence on the pair of composable morphisms in $D(\cA)$ \cite{franke:adams,maltsiniotis:seminar,groth:ptstab}.

Triangulated categories encode aspects of the calculus of derived cokernels of morphisms and of pairs of composable morphisms. Going beyond this and taking into account longer strings of composable morphisms, there is the related idea of considering higher versions of octahedron diagrams. This idea goes back at least to \cite{beilinson:perverse} and was axiomatized by Maltsiniotis \cite{maltsiniotis:higher} in a notion of \emph{higher triangulation}. If one follows the above reasoning more systematically, then one can use the epivalences \eqref{eq:dia-[n]} to construct \emph{canonical higher triangulations} on $D(\cA)$. In fact, this works more generally for the values of strong, stable derivators \cite{maltsiniotis:higher,gst:Dynkin-A,groth:revisit} such as homotopy derivators of exact categories \cite{gillespie:exact,stovicek:exact-model}, stable cofibration categories \cite{schwede:p-order,lenz:derivators}, stable model categories \cite{hovey:model}, or stable $\infty$-categories \cite{HTT,HA,lenz:derivators} (hence, in particular, in the specific examples mentioned at the beginning of this introduction).

Now, the main goal of this paper and its sequels is to generalize the calculus of morphisms and finite strings of composable morphisms in a different direction. We want to study the calculus of \emph{morphisms, squares, cubes, and higher dimensional $n$-cubes} in abstract homotopy theories (with a certain focus on pointed or stable homotopy theories). This is partially motivated as an attempt to systematically study compatibility properties satisfied by enhancements of the typical monoidal, triangulated categories (aka.~tensor-triangulated categories) arising in nature. In fact, forming iterated pointwise tensor products of morphisms leads very naturally to representations of $n$-cubes. However, the cubical calculus is also interesting in Goodwillie calculus \cite{goodwillie:II,munson-volic} as well as in not necessarily monoidal or stable homotopy theories, and the ultimate goal of this project is to contribute to the understanding of this calculus.

Following the examples of triangulations and higher triangulations, a first attempt could consist of trying to axiomatize such a calculus at the level of cubical diagrams in $D(\cA)$. Let us denote by $\cube{n}, n\geq 0,$ the $n$-cube, which is to say the $n$-fold power of $[1]$. The previous situations suggest that, whatever the precise axioms of ``cubical triangulations'' would be, one would be led to consider the functor
\begin{equation}\label{eq:dia-n-cube}
\dia_{\cube{n}}\colon D(\cA^{\cube{n}})\to D(\cA)^{\cube{n}}
\end{equation}
which sends representations of $n$-cubes with values in chain complexes to the underlying representations in the derived category. The problem now is that, even in the case of vector spaces over a field and already for $n=2$, the functor \eqref{eq:dia-n-cube} is \emph{not} an epivalence. Hence, unlike the previous cases, the pattern of lifting diagrams against \eqref{eq:dia-n-cube} in order to then apply certain functorial constructions to them does not work. (See \S\ref{sec:coherent} for a more detailed discussion of this failure and of an attempt to partially sail around this problem.) As an upshot, in order to develop the intended calculus, we proceed differently and work with representations of $n$-cubes in chain complexes, considered as objects in derived categories $D(\cA^{\cube{n}})$ for $n\geq 0$. And to allow for additional flexibility for constructions, we work with the entire derivator of the abelian category
\[
\D_\cA\colon B\mapsto D(\cA^B),
\]
which encodes the derived categories of various $B$-shaped diagram categories. 

More generally, the main theme is to develop the calculus of $n$-cubes in derivators, pointed derivators, and mostly stable derivators (such as homotopy derivators of exact categories, stable cofibration categories, stable model categories, or stable $\infty$-categories). Consequently, this paper and the sequels belong to abstract representation theory \cite{gst:basic,gst:tree,gst:Dynkin-A,gst:acyclic}, the formal study of stability \cite{groth:ptstab,gps:mayer,gs:stable}, and to the formal study of the interaction of stability and monoidality with its applications to fairly general additivity of traces results \cite{gps:additivity,ps:linearity,ps:linearity-fp,gallauer:traces}. The search for good compatibility axioms to be imposed on monoidality and stability has already quite some history; it goes back at least to \cite{margolis:spectra} and was pushed significantly further in \cite{hps:axiomatic,may:additivity,keller-neeman:D4}. 

We now turn to a more detailed description of the content of this paper. 

\section{Introduction}
\label{sec:intro}

The key organizational tool for abstract cubical homotopy theory is the \emph{cardinality filtration} of $n$-cubes. The $n$-cube $\cube{n}$ can be realized as the power set of $\{1,\ldots,n\}$, and correspondingly it admits a filtration by the full subcategories $\cube{n}_{0\leq l}$ spanned by the subsets of cardinality at most $l$. The differences between the layers of this filtration are measured, using suggestive notation, by the \emph{chunks} $\cube{n}_{k\leq l}$, and the natural inclusions between these chunks are conveniently organized by means of \autoref{fig:intro}. In this paper we aim for an understanding of stable representations of these chunks $\cube{n}_{k\leq l}$ for arbitrary but fixed $n$, while in \cite{bg:parasimplicial} we investigate how the representation theories of the various chunks for the various $n$ assemble together. 
\begin{figure}
\begin{displaymath}
\xymatrix{
\emptyset\ar[r]&\cube{n}_{=0}\ar[r]&\cube{n}_{0\leq 1}\ar[r]&\ldots\ar[r]&\cube{n}_{0\leq n-1}\ar[r]&\cube{n}\\
&\emptyset\ar[r]\ar[u]&\cube{n}_{=1}\ar[r]\ar[u]&\ldots\ar[r]\ar[u]&\cube{n}_{1\leq n-1}\ar[r]\ar[u]&\cube{n}_{1\leq n}\ar[u]\\
&&\ldots\ar[u]\ar[r]&\ldots\ar[u]\ar[r]&\ldots\ar[u]\ar[r]&\ldots\ar[u]\\
&&&\ldots\ar[u]\ar[r]&\cube{n}_{=n-1}\ar[r]\ar[u]&\cube{n}_{n-1\leq n}\ar[u]\\
&&&&\emptyset\ar[r]\ar[u]&\cube{n}_{=n}\ar[u]\\
&&&&&\emptyset\ar[u]
}
\end{displaymath}
\caption{The cardinality filtration of the $n$-cube $\cube{n}$}
\label{fig:intro}
\end{figure}
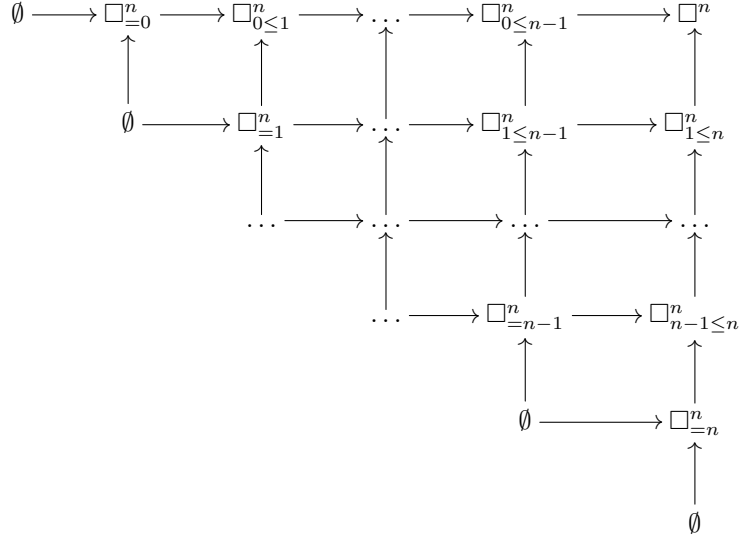

In the first part of this paper (consisting of \S\S\ref{sec:coherent}-\ref{sec:total}) we prepare the discussion of more general chunks by first turning to representations of the posets $\cube{n}_{0\leq k}$ and $\cube{n}_{k\leq n}$ in derivators \D (see the first row and the final column in \autoref{fig:intro}). Let us begin by the first row. We introduce \emph{$k$-cotruncated $n$-cubes} as the essential image of the fully faithful left Kan extension morphism $\D^{\cube{n}_{0\leq k}}\to\D^{\cube{n}}$. These classes of $n$-cubes interpolate between constant $n$-cubes (for $k=0$), strongly cocartesian $n$-cubes (for $k=1$), cocartesian $n$-cubes (for $k=n-1$), and generic $n$-cubes (for $k=n$). The class of $k$-cotruncated $n$-cubes can be characterized as those $n$-cubes for which all $(k+1)$-dimensional subcubes are cocartesian (\autoref{thm:cotr-filtration}). To study this interpolation more carefully, we pass to \emph{pointed} derivators and develop the yoga of iterated partial cones. Partial cones commute with each other (a variant of the 3-by-3-lemma), and $n$-fold cones are canonically isomorphic to total cofibers (\autoref{thm:total-cof}). As a consequence we conclude that all $(k+1)$-fold cones vanish on $k$-cotruncated $n$-cubes. 

The \emph{duality principle} from derivator theory immediately yields the picture dual to the one we just sketched. At the level of shapes in \autoref{fig:intro} this corresponds to the passage from the first row to the final column. We define \emph{$k$-truncated $n$-cubes} as the essential image of the right Kan extension morphism $\D^{\cube{n}_{k\leq n}}\to\D^{\cube{n}}$. The reason why we stress this duality principle is that for \emph{stable} derivators these two pictures fit together particularly nicely (see \cite{gps:mayer,gs:stable,gst:basic} for related facts). In the stable case, an $n$-cube $X\in\D^{\cube{n}}$ is $k$-cotruncated iff all $(k+1)$-fold cones of $X$ vanish iff all $(k+1)$-fold fibers of $X$ vanish iff $X$ is $(n-k)$-truncated. As a consequence we conclude that the posets $\cube{n}_{0\leq k}$ and $\cube{n}_{n-k\leq n}$ have the same abstract stable representation theories (\autoref{thm:sse-chunks-special-case}), i.e., they are \emph{strongly stably equivalent} in the sense of \cite{gst:basic}.

In the second part of this paper (consisting of \S\S\ref{sec:sse-cotruncated}-\ref{sec:global}) we focus on the stable picture, and this part relies in an essential way on the various calculational tools for iterated cones collected in \S\ref{sec:total}. One goal is to expand on the strong stable equivalences $\cube{n}_{0\leq k}\stackrel{s}{\sim}\cube{n}_{n-k\leq n}$. It turns out that the entire mirror symmetry of \autoref{fig:intro} can be realized at the level of abstract stable representations. As an incarnation of this \emph{chunk symmetry}, we show that for all $0\leq k\leq l\leq n$ the chunks $\cube{n}_{k\leq l}$ and $\cube{n}_{n-l\leq n-k}$ are strongly stably equivalent (\autoref{thm:sse-chunks}). In order to make precise in which sense these strong stable equivalences 
\begin{equation}\label{eq:intro-sse-chunk}
\Phi_{k,l}=\Phi_{k,l}^{(n)}\colon\D^{\cube{n}_{k\leq l}}\toiso\D^{\cube{n}_{n-l\leq n-k}}
\end{equation}
are compatible for the various chunks, it is convenient to organize the chunks themselves in a category $\chunk{n}$. This is the twisted morphism category of $[n]$, it agrees with the shape of \autoref{fig:intro}, and the chunk symmetry is simply induced by the self-duality of $[n]$. For a fixed stable derivator \D, the assignments
\[
R^{(n)}(\D)\colon\cube{n}_{k\leq l}\mapsto\D^{\cube{n}_{k\leq l}}\qquad\text{and}\qquad
(Z^{(n)})^\vee(\D)\colon\cube{n}_{k\leq l}\mapsto\D^{\cube{n}_{n-l\leq n-k}}
\]
encoding the abstract representation theories of the various chunks can be extended to pseudo-functors $\chunk{n}\to\cDER_{\mathrm{St},\exx}$ with values in the $2$-category of stable derivators and exact morphisms. (The notation for these pseudo-functors is motivated by their construction; see \S\ref{sec:global}.) We show in \autoref{thm:sse-chunk-global} that the strong stable equivalences \eqref{eq:intro-sse-chunk} can be assembled to a pseudo-natural equivalence
\begin{equation}\label{eq:intro-sse-chunk-pseudo}
\Phi=\Phi^{(n)}\colon R^{(n)}(\D)\toiso (Z^{(n)})^\vee(\D)\colon\chunk{n}\to\cDER_{\mathrm{St},\exx},
\end{equation}
and we also obtain variants of these equivalences.

An additional interesting symmetry comes from abstract or spectral \emph{Serre duality}. Stable derivators of representations of chunks admit Serre equivalences
\begin{equation}\label{eq:intro-Serre}
S_{k,l}=S_{k,l}^{(n)}\colon\D^{\cube{n}_{k\leq l}}\toiso\D^{\cube{n}_{k\leq l}}.
\end{equation}
Furthermore, similar to the case of the strong stable equivalences, we establish in \autoref{thm:Serre} that these Serre equivalences \eqref{eq:intro-Serre} assemble to suitable pseudo-natural equivalences
\[
S=S^{(n)}\colon R^{(n)}(\D)\toiso L^{(n)}(\D)\colon \chunk{n}\to\cDER_{\mathrm{St},\exx}.
\]
More specifically, this means that for every morphism $i\colon\cube{n}_{k',l'}\to\cube{n}_{k,l}$ in \autoref{fig:intro} the left and right Kan extension morphisms $i_!,i_\ast\colon\D^{\cube{n}_{k',l'}}\to\D^{\cube{n}_{k,l}}$ correspond to each other under conjugation by the Serre equivalences
\begin{equation}\label{eq:intro-Serre-conj}
S_{k,l}\circ i_\ast \cong i_!\circ S_{k',l'}.
\end{equation}

These equivalences \eqref{eq:intro-Serre} are examples of the more general concept of \emph{global Serre duality}, a notion we introduce formally and study more systematically in \cite{bg:global}. Here we content ourselves by indicating two directions of application, both of which deserve additional exploration. First, in our specific situation, the Serre equivalences \eqref{eq:intro-Serre} induce canonical isomorphisms
\begin{equation}\label{eq:intro-colim-lim}
\colim_{\cube{n}_{k\leq l}}\toiso\mathrm{lim}_{\cube{n}_{k\leq l}}\circ S_{k,l}^{(n)}\colon\D^{\cube{n}_{k\leq l}}\to\D,
\end{equation}
thereby again making precise that ``in stable land the distinction between (homotopy finite) colimits and limits is blurred''. We recall that by \cite{gs:stable} the defining feature of stability is that homotopy finite colimits are \emph{weighted limits}. In view of this result, the canonical isomorphisms \eqref{eq:intro-colim-lim} can be interpreted by saying that the Serre equivalences transform the \emph{weighted limit} back into an \emph{actual limit}, thereby making the ``blurring more direct''. The natural isomorphisms \eqref{eq:intro-Serre-conj} provide an additional incarnation of this principle for homotopy finite Kan extensions.

As a second application we note that for every morphism $i\colon\cube{n}_{k',l'}\to\cube{n}_{k,l}$ in $\chunk{n}$ the restriction morphism
\[
i^\ast\colon\D^{\cube{n}_{k\leq l}}\to\D^{\cube{n}_{k'\leq l'}}
\]
generates an infinite chain of adjunctions. Some chunks enjoy an abstract fractionally Calabi--Yau property, and in those cases the formulas for these iterated adjoints admit a particularly simple form. 

In this paper we also establish two results on derivators which are of independent interest and which we want to stress in this introduction. These results are certainly well-known to experts, but we could not find citable references containing them and hence include the results here. 

First, in \S\ref{sec:formulas} we show that the cubical calculus is compatible with exact morphisms of derivators. In fact, this turns out to be a special case of a more general (seemingly technical, but actually quite convenient) result concerning the compatibility of morphisms of derivators with canonical mates (\autoref{prop:mates-vs-morphisms}). Roughly speaking, this result says that morphisms of derivators enjoy a ``lax-oplax compatibility with formulas'', and this immediately specializes to the ``pseudo-compatibility'' of the cubical calculus with exact morphisms. More generally, applied to actions of monoidal derivators, this implies that ``formulas'' propagate from a monoidal derivator \V to \V-modules, thereby making precise the universality of ``formulas'' in monoidal homotopy theories. For example, by the universality of spectra \cite{heller:htpythies,heller:stable,franke:adams,cisinski:derived-kan,tabuada:universal-invariants,cisinski-tabuada:non-connective,cisinski-tabuada:non-commutative}, all ``exact formulas'' (such as \eqref{eq:intro-Serre-conj} and \eqref{eq:intro-colim-lim}) propagate from spectra to stable derivators.

Second, in \S\ref{sec:decompose} we stress the philosophy that a combination of homotopy finality arguments with the smoothness/properness of Grothendieck (op)fibrations leads to fairly general decomposition results for (co)limits. This abstract philosophy is formulated as a decomposition theorem (\autoref{thm:decompose}). We illustrate this theorem by showing that the classical Bousfield--Kan formulas for homotopy (co)limits \cite{bousfield-kan:htpy-limits} hold in derivators (\autoref{thm:BK-formulas}). As a further illustration we obtain a decomposition result related to exhaustive filtrations of categories (\autoref{thm:exhaustive-filt}). As a sample application of this latter result, and this was our original motivation for including these decomposition techniques in this paper, we obtain in \S\ref{sec:punctured} a fairly systematic proof of the inductive formula for colimits of punctured $n$-cubes.

In this paper we freely use the basic theory of derivators (see \cite{grothendieck:derivators,franke:adams,keller:epivalent} and \cite{cisinski:direct,cisinski:derived-kan,maltsiniotis:intro,maltsiniotis:k-theory}) and follow the conventions and notation from \cite{gst:tree,groth:revisit}. For a more detailed introduction to derivators we refer the reader to \cite{groth:ptstab}. For convenience, we conclude the introduction by the following list of conventions, notation, and key facts which will be used throughout this paper.

\begin{itemize}
\item We denote by $\cCat$ the $2$-category of small categories and by $\cCAT$ the $2$-category of not necessarily small categories.
\item We follow the convention of Heller \cite{heller:htpythies} and Franke \cite{franke:adams} that derivators are suitable $2$-functors $\D\colon\cCat\op\to\cCAT$ with domain $\cCat\op$, thereby modeling derivators on \emph{diagrams} in homotopy theories. There is an equivalent approach based on \emph{presheaves} and in that case the domain is the $2$-category $\cCat\coop$ (as used, for example, by Grothendieck \cite{grothendieck:derivators} and Cisinski \cite{cisinski:direct,cisinski:derived-kan}). Both conventions yield isomorphic $2$-categories of derivators.
\item We denote by $\bbone$ the terminal category, and we tacitly use the canonical isomorphisms $A\cong A^\bbone$ for $A\in\cCat$. The unique functors to terminal categories are denoted by $\pi=\pi_A\colon A\to\bbone$.
\item Given a derivator \D and $A\in\cCat$, we write $\D^A\colon B\mapsto\D(A\times B)$ for the derivator of coherent $A$-shaped diagrams in \D. For every functor $u\colon A\to B$ in $\cCat$, we denote by
\[
u^\ast\colon\D^B\to\D^A,\quad u_!\colon\D^A\to\D^B,\quad\text{and}\quad u_\ast\colon\D^A\to\D^B
\]
the corresponding precomposition, left Kan extension, and right Kan extension morphism. Kan extensions along functors $\pi=\pi_A\colon A\to\bbone$ are denoted by $\colim_A\colon\D^A\to\D$ and $\lim_A\colon\D^A\to\D$, respectively, and referred to as colimits and limits. Precomposition along $a\colon\bbone\to A$ is referred to as evaluation at $a$, and its effect on morphisms is denoted by $(f\colon X\to Y)\mapsto(f_a\colon X_a\to Y_a)$. 
\item In this paper we systematically use parametrized Kan extensions and restrictions, resulting in the use of morphisms of derivators instead of underlying functors. In particular, we write $X\in\D$ in case there is a category $A\in\cCat$ such that $X\in\D(A)$. The motivation is that we want to formulate the results in the language of abstract homotopy theories in contrast to homotopy categories. An interpretation in terms of actual categories can be obtained by evaluation at $\bbone$.  
\item Kan extensions along fully faithful functors are fully faithful and hence restrict to equivalences. It is convenient to distinguish these notationally. More generally, given a fully faithful morphism $F\colon\D\to\E$ of derivators we denote by $F^\simeq\colon\D\toiso\mathrm{Im}(F)$ the induced equivalence onto the essential image $\mathrm{Im}(F)\subseteq\E$.
\item Given a pointed derivator \D, $A\in\cCat$, and a subcategory $B\subseteq A$, we denote by $\D^{A,B}\subseteq\D^A$ the prederivator spanned by all $X\in\D^A$ such that $X|_B=0$. 
\item We denote by $\cDER$ the $2$-category of derivators, morphisms of derivators (pseudo-natural transformations) and natural transformations (modifications), and by $\cDER_{\mathrm{St},\exx}\subseteq\cDER$ the $2$-full subcategory spanned by stable derivators, exact morphisms and all natural transformations.
\item Finally, we denote by $y_{\cC}$ the prederivator represented by a category $\cC$, by $\D_\cA$ the derivator of an abelian category (defined on homotopy finite categories only), by $\ho_\cM$ the homotopy derivator of a Quillen model category $\cM$, and by $\ho_\cD$ the homotopy derivator of a bicomplete $\infty$-category $\cD$.
\end{itemize}

\section{Coherent versus incoherent $n$-cubes}
\label{sec:coherent}

In this section we stress the difference between coherent and incoherent $n$-cubes. Even in the case of a field and already in dimension two, the corresponding forgetful functor does not reflect isomorphism types. This illustrates that, in contrast to triangulations and higher triangulations, it is not possible to capture the cubical calculus at the level of incoherent diagrams.

\begin{notn}\label{notn:cubes}
Let $n\in\lN_{\geq 0}$ be a natural number. We denote by
\[
\cube{n}=\mathcal{P}(\lbrace 1,\ldots,n\rbrace)
\]
the power set of $\lbrace 1,\ldots,n\rbrace$ and refer to this poset as the \textbf{$n$-cube}. The $n$-cube is, of course, isomorphic to the $n$-th power $[1]^n$ of the poset $[1]=(0<1)$. Hence, elements of $\cube{n}$ can be described as subsets or as tuples. The correspondence is realized by the characteristic function and will be used tacitly without additional notation. Occasionally, it will be useful to emphasize notationally the set $\lbrace 1,\ldots,n\rbrace$ of which we form the power set, and in that case we also write
\begin{equation}\label{eq:cube-set}
\cube{\lbrace 1,\ldots,n\rbrace}=\cube{n}
\end{equation}
for the $n$-cube. 
\end{notn}

\begin{defn}
Given a derivator \D and $n\in\lN_{\geq 0}$, we refer to $\D^{\cube{n}}$ as the \textbf{derivator of (coherent) $n$-cubes} in \D.
\end{defn}

\begin{egs}\label{egs:shift}
The passage from categories, abelian categories, Quillen model categories, and $\infty$-categories to derivators is compatible with the formation of exponentials. Hence, given a bicomplete category $\cC$, a Grothendieck abelian category $\cA$, a Quillen model category $\cM$, or a bicomplete $\infty$-category $\cD$, for $n\geq 0$ there are equivalences of derivators
\[
y_{\cC}^{\cube{n}}\simeq y_{\cC^{\cube{n}}},\quad \D_\cA^{\cube{n}}\simeq\D_{\cA^{\cube{n}}},\quad \ho_{\cM}^{\cube{n}}\simeq\ho_{\cM^{\cube{n}}}, \quad\text{and}\quad \ho_{\cD}^{\cube{n}}\simeq\ho_{\cD^{\cube{n}}}.
\]
\end{egs}

Using derivators of more restricted types this extends to more general input data such as exact categories in the sense of Quillen \cite{quillen:k-theory} or cofibration categories (see \cite{keller:exact,cisinski:derivable}).

\begin{rmk}
For every prederivator \D and $B\in\cCat$ there is an underlying diagram functor
\[
\dia_B\colon\D(B)\to\D(\bbone)^B
\]
defined by $\dia_B(X)_b=X_b, b\in B$ \cite[\S1]{groth:ptstab}. These functors assemble to a strict underlying diagram morphism to the prederivator represented by the underlying category $\D(\bbone)$,
\begin{equation}\label{eq:dia-mor}
\dia\colon\D\to y_{\D(\bbone)}.
\end{equation}
The prederivator \D is (equivalent to) a represented prederivator if and only if the underlying diagram morphism \eqref{eq:dia-mor} is an equivalence.  
\end{rmk}

Thus, in the examples not arising from ordinary category theory (hence, the motivational examples for the theory of derivators), the underlying diagram morphisms fail to be equivalences. In order to illustrate this by specific examples related to $n$-cubes we include the following proposition. Therein, we denote by $\lZ$ the integers considered as a discrete category, so that $(-)^\lZ$ is the passage to the category of $\lZ$-graded objects.

\begin{prop}\label{prop:dia}
Let $\cA$ be a semisimple abelian category and let $B\in\cCat$. The underlying diagram functor
\[
\dia_B\colon D(\cA^B)\to D(\cA)^B
\]
is equivalent to the graded homology functor
\[
H_\ast\colon D(\cA^B)\to (\cA^B)^\lZ.
\]
\end{prop}
\begin{proof}
Unraveling definitions it is immediate that the square
\begin{equation}\label{eq:dia-I}
\vcenter{
\xymatrix{
D(\cA^B)\ar[r]^-{\dia_B}\ar[d]_-{H_\ast}&D(\cA)^B\ar[d]^-{H_\ast}\\
(\cA^B)^\lZ\ar[r]_-\simeq&(\cA^\lZ)^B
}
}
\end{equation}
is commutative, where the bottom horizontal map is the categorical exponential law. Since $\cA$ is semisimple, the homology functor $H_\ast\colon D(\cA)\to\cA^\lZ$ is an equivalence of categories (\cite[\S III.2]{gelfand-manin:homological}) and hence so is the vertical morphism on the right.
\end{proof}

\begin{cor}\label{cor:dia}
Let $R$ be a semisimple ring, let $B$ be a category with finitely many objects, and let $RB$ be the category algebra. The underlying diagram functor
\[
\dia_B\colon D(\Mod(R)^B)\to D(R)^B
\]
is equivalent to the graded homology functor
\[
H_\ast\colon D(RB)\to \Mod(RB)^\lZ.
\]
\end{cor}
\begin{proof}
The equivalence of categories $\Mod(RB)\toiso\Mod(R)^B$ allows us to extend diagram \eqref{eq:dia-I} to
\begin{equation}\label{eq:dia-II}
\vcenter{
\xymatrix{
D(RB)\ar[r]^-\simeq\ar[d]_-{H_\ast}\ar@{}[rd]|{\cong}&D(\Mod(R)^B)\ar[r]^-{\dia_B}\ar[d]_-{H_\ast}&D(\Mod(R))^B\ar[d]^-{H_\ast}_\simeq\\
\Mod(RB)^\lZ\ar[r]_-\simeq&(\Mod(R)^B)^\lZ\ar[r]_-\simeq&(\Mod(R)^\lZ)^B.
}
}
\end{equation}
The square on the left commutes up to a natural isomorphism by additivity of homology, and this diagram hence concludes the proof.
\end{proof}

\begin{rmk}
Since underlying diagram functors generalize graded homology functors, these functors tend to fail to be equivalences. But for certain shapes they are reasonably well-behaved: Derivators arising in nature (such as the ones in \autoref{egs:shift}) are \emph{strong}. Thus, in those cases, for every \emph{free} category $F$ and every $A\in\cCat$ the partial underlying diagram functor
\begin{equation}\label{eq:partial-dia}
\dia_{F,A}\colon\D(F\times A)\to\D(A)^F
\end{equation}
(which makes diagrams incoherent in the $F$-direction but remembers the coherence in the $A$-direction) is essentially surjective and full. In the stable case, this allows us to lift incoherent morphisms and chains of composable such morphisms to coherent diagrams and then to establish the \emph{existence} of distinguished triangles and higher triangles \cite{franke:adams,maltsiniotis:seminar,groth:ptstab,groth:revisit,gst:Dynkin-A}.
\end{rmk}

\begin{rmk}\label{rmk:strong-triang}
We note that, even for a semisimple abelian category $\cA$, the underlying diagram morphism
\[
\dia\colon\D_\cA\to y_{D(\cA)}
\]
is \emph{not} an equivalence of derivators. An indirect way to see this is as follows. If the morphism were an equivalence, then $D(\cA^B),B\in\cCat,$ would be triangulated and abelian, and hence semisimple. But at the level of algebras, this exponentiation specializes to the passage from a field to path algebras of quivers, group algebras and incidence algebras of posets. And clearly these algebras fail to be semisimple.
\end{rmk}

We also want to give a specific counter-example.

\begin{eg}
Let us consider a field $k$ and the poset $[1]=(0<1)$. In this case the underlying diagram functor is equivalent to $H_\ast\colon D(k[1])\to(\Mod(k[1]))^\lZ$ (\autoref{cor:dia}). While this functor is full and essentially surjective (\autoref{rmk:strong-triang}), it is not faithful. In fact, it kills
\begin{equation}\label{eq:evil-ext}
\Hom_{D(k[1])}((k\to 0),\Sigma(0\to k))\cong\Ext^1_{k[1]}((k\to 0),(0\to k))\cong k
\end{equation}
since in $\Mod(k[1])^\lZ$ we have $\Hom_{\Mod(k[1])^\lZ}((k\to 0),\Sigma(0\to k))\cong 0$. 
\end{eg}

\begin{rmk}
By means of Auslander--Reiten theory one can show that in a certain sense precisely the group \eqref{eq:evil-ext} is responsable for the lack of faithfulness of $H_\ast$ \cite[\S\S4-5]{happel:triangulated}. Moreover, this group is also the reason why the left derived cokernel
\[
L\mathrm{coker}\cong C\colon D(k[1])\to D(k)
\]
does not factor through the morphism category $D(k)^{[1]}$. (Recall that every abelian category $\cA$ admits a left derived cokernel functor given by the classical cone construction $L\mathrm{coker}\cong C\colon D(\cA^{[1]})\to D(\cA)$.)
\end{rmk}

\begin{rmk}
For derivators arising in nature (such as homotopy derivators of exact categories, abelian categories, model categories, or $\infty$-categories), the partial underlying diagram functors \eqref{eq:partial-dia} fail to be essentially surjective and full for more general shapes. In related situations obstruction theories studying the existence of lifts have been developed (see \cite{dwyer-kan:obstruction,dwyer-kan:realizing,dks:htpy-comm-real} and \cite{blanc-markl:htpy-ops,blanc-johnson-turner:htpy-ops}).
\end{rmk}

We again illustrate this by a specific example.

\begin{eg}\label{eg:field}
Let $\D_k$ be the derivator of a field and let $k\square$ be the path algebra of the commutative square $\square$. The functor $\dia_{\square,\bbone}\colon\D_k(\square)\to\D_k(\bbone)^\square$ is by \autoref{cor:dia} equivalent to the graded homology functor
\[
H_\ast\colon D(k\square)\to \Mod(k\square)^\lZ
\]
which is not an epivalence. In fact, otherwise it would preserve isomorphism types in the sense that $H_\ast(X)\cong H_\ast(Y)$ implies that there is an isomorphism $X\cong Y$, and this is impossible by the following example. Let us consider the following complexes which are concentrated in homological degree $0$ and $1$,
\[
\xymatrix{
0\ar[r]\ar[d]&0\ar[d]\ar@{}[dr]|{\to}&k\ar[r]\ar[d]&0\ar[d]&&0\ar[r]\ar[d]&0\ar[d]\ar@{}[dr]|{\to}&k\ar[r]\ar[d]&0\ar[d]\\
0\ar[r]&k&0\ar[r]&0&&k\ar[r]&k&k\ar[r]&0.
}
\]
Here, all structure maps of representations and all components of differentials are identities as soon as possible and they vanish otherwise. Both complexes have isomorphic homology which is concentrated in degree zero where it agrees with the indecomposable injective representation corresponding to the vertex $(0,0)\in\square$. However, these chain complexes are not quasi-isomorphic.
\end{eg}

\begin{rmk}
Note that this example is not in contradiction to the fact that $\dia_{[1]}$ is always an epivalence in strong derivators. For every derivator \D the functor $\dia_{\square}$ factors as 
\[
\dia_{\square}\colon\D(\square)\to\D([1])^{[1]}\to\D(\bbone)^\square.
\]
Here, the first functor is a diagram functor for the derivator $\D^{[1]}$ while the second one is the exponential
\[
(\dia_{[1]})^{[1]}\colon\D([1])^{[1]}\to\D(\bbone)^\square.
\]
Strong derivators are by definition closed under exponentials, and for such derivators the first functor is hence an epivalence. However, even for strong derivators, the second functor is not necessarily an epivalence, since epivalences are not closed under exponentials.
\end{rmk}

\begin{rmk}
As we have just seen, in general, incoherent cubes can not be lifted to coherent cubes (at least not uniquely up to isomorphism). However, in the case of strong, stable, monoidal derivators we can lift those cubes which arise as iterated pointwise products of morphisms, and this suggests that certain aspects of the cubical calculus can be encoded at the level of incoherent diagrams. An excellent illustration of this is provided by \cite{may:additivity}. In \emph{loc.~cit.}, May proposes very well-chosen axioms for monoidal, triangulated categories which capture an additivity result for Euler characteristics. (An extension to more general traces, however, is impossible \cite{ferrand:nonadd}.)
\end{rmk}

\section{Cocartesian and strongly cocartesian $n$-cubes}
\label{sec:strongly-cocart}

The main purpose of this section is to establish some basic notation related to coherent $n$-cubes in derivators and to collect first facts on (strongly) cocartesian $n$-cubes.

\begin{notn}\label{notn:source-punctured}
For every $n\in\lN_{\geq 0}$ and $0\leq k\leq n$ we denote by
\[
\iota_{\leq k}\colon\cube{n}_{\leq k}\to\cube{n}
\]
the inclusion of the full subposet spanned by all subsets with at most $k$ elements. We agree on the convention that $\iota_{\leq -1}\colon\cube{n}_{\leq -1}\to\cube{n}$ is the empty functor. We also use obvious variants of this notation such as
\[
\iota_{<k}\colon\cube{n}_{<k}\to\cube{n},\quad\iota_{=k}\colon\cube{n}_{= k}\to\cube{n},\quad\iota_{\geq k}\colon\cube{n}_{\geq k}\to\cube{n},\quad\mbox{and}\quad\iota_{>k}\colon\cube{n}_{>k}\to\cube{n}.
\]
As a special case there are the inclusions
\begin{equation}\label{eq:source-punctured}
\iota_{\leq 1}\colon\cube{n}_{\leq 1}\to\cube{n}\qquad\mbox{and}\qquad\iota_{\leq n-1}\colon\cube{n}_{\leq n-1}\to\cube{n}
\end{equation}
of the \textbf{source of valence $n$} and of the \textbf{punctured $n$-cube}, respectively (see \autoref{fig:cube-inclusions}).
\end{notn}

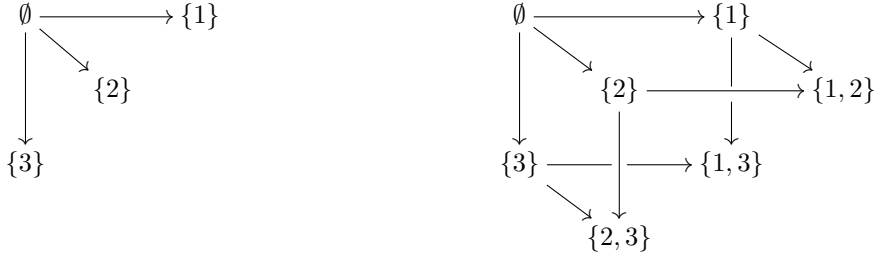
\begin{figure}
\xymatrix@-1pc{
&\emptyset\ar[rr]\ar[dr]\ar[dd]&& \lbrace 1\rbrace &&&&
\emptyset\ar[rr]\ar[dr]\ar[dd] && \lbrace 1\rbrace \ar[dr] \ar'[d][dd]\\
&& \lbrace 2 \rbrace && &&&
& \lbrace 2 \rbrace \ar[rr] \ar[dd] && \lbrace 1,2 \rbrace \\
&\lbrace 3 \rbrace && &&&&
\lbrace 3 \rbrace \ar[dr] \ar'[r][rr] && \lbrace 1,3 \rbrace\\
&& && {\phantom{\lbrace 1,2,3\rbrace}}&&&
& \lbrace 2,3 \rbrace && 
}
\caption{The source of valence three and the punctured $3$-cube.}
\label{fig:cube-inclusions}
\end{figure}

Coherent $n$-cubes in the essential image of the left Kan extensions along the functors in \eqref{eq:source-punctured} deserve particular names.

\begin{defn}\label{defn:cocart}
Let \D be a derivator, $n\geq 2$, and $X\in\D^{\cube{n}}$.
\begin{enumerate}
\item The $n$-cube $X$ is \textbf{cocartesian} if it lies in the essential image of $(\iota_{\leq n-1})_!$.
\item The $n$-cube $X$ is \textbf{ strongly cocartesian} if it lies in the essential image of $(\iota_{\leq 1})_!$.
\end{enumerate}
\end{defn}

There are dual notions of \textbf{(strongly) cartesian $n$-cubes}, and the duality principle allows us to focus on the notions in \autoref{defn:cocart}.

\begin{rmk}\label{rmk:cocart}
Let \D be a derivator, $n\geq 2,$ and $X\in\D^{\cube{n}}$.
\begin{enumerate}
\item A square is cocartesian if and only if it is strongly cocartesian.
\item Since $\iota_{\leq 1}\colon\cube{n}_{\leq 1}\to\cube{n}$ factors through $\iota_{\leq n-1}\colon\cube{n}_{\leq n-1}\to\cube{n}$, every strongly cocartesian $n$-cube is cocartesian. But from dimension three on there is no converse to this implication (\autoref{thm:cocart}).
\item Since left Kan extensions along fully faithful functors are fully faithful, we note that $X$ is cocartesian resp.~strongly cocartesian if and only if the adjunction counit
\[
\varepsilon\colon (\iota_{\leq n-1})_!\iota_{\leq n-1}^\ast X\to X\qquad\mbox{resp.}\qquad
\varepsilon\colon (\iota_{\leq 1})_!\iota_{\leq 1}^\ast X\to X
\]
is invertible. In the first case this amounts to saying that $X$ is cocartesian if and only if the canonical mate
\[
\colim_{\cube{n}_{\leq n-1}}\iota_{\leq n-1}^\ast X\to X_{\lbrace 1,\ldots,n\rbrace}
\]
is an isomorphism. (In \S\ref{sec:punctured} we extend to derivators a well-known inductive formula for (homotopy) colimits of punctured $n$-cubes.)
\end{enumerate}
\end{rmk}

\begin{notn}\label{notn:simpl}
Let $n\in\lN_{\geq 0}$ be a natural number. We denote by
\[
[n]=(0<1<\ldots<n)
\]
the finite ordinal. Moreover, for $0\leq j\leq k\leq n$ there is a unique monotone map
\[
i_{jk}\colon[1]\to[n]
\]
such that $0\mapsto j$ and $1\mapsto k$.
\end{notn}

We recall the following basic facts related to these notions (\cite[\S8]{gst:basic}).

\begin{thm}\label{thm:cocart}
Let \D be a derivator and let $n\geq 2$.
\begin{enumerate}
\item Let $X\in\D^{\cube{n}\times[1]}$ be such that the $n$-cube $X_0\in\D^{\cube{n}}$ is cocartesian. Then $X$ is cocartesian if and only if $X_1$ is cocartesian.
\item An $n$-cube $X\in\D^{\cube{n}}$ is strongly cocartesian if and only if all $k$-dimensional subcubes, $2\leq k\leq n,$ are cocartesian if and only if all subsquares are cocartesian.
\item Let $X\in\D^{\cube{n}\times[2]}$ be such that $i_{01}^\ast X$ is (strongly) cocartesian. Then $i_{12}^\ast X$ is (strongly) cocartesian if and only if $i_{02}^\ast X$ is (strongly) cocartesian.
\end{enumerate}
\end{thm}
\begin{proof}
The first statement is \cite[Thm.~8.11]{gst:basic} and the second one is a combination of \cite[Thm.~8.4]{gst:basic} and \cite[Cor.~8.12]{gst:basic}. The remaining statement is \cite[Prop.~8.15]{gst:basic}.
\end{proof}

Obviously, the classes of (strongly) cocartesian $n$-cubes should be invariant under the swapping of coordinates. A formal proof of this relies on \autoref{lem:he-swap-coor}, which we include in a generality which is also suitable for later applications.

\begin{notn}\label{notn:cotrunc}
Let $n\geq 2$ be a natural number.
\begin{enumerate}
\item For natural numbers $0\leq k\leq l\leq n,$ we denote by 
\begin{equation}\label{eq:cotrunc}
\iota_{l,k}\colon\cube{n}_{\leq k}\to\cube{n}_{\leq l}\qquad\text{and}\qquad\iota_{k,l}\colon\cube{n}_{\geq l}\to\cube{n}_{\geq k}
\end{equation}
the obvious inclusions. (The notation is only overloaded in the case of $l=k$ and then both functors are identities anyhow.)
\item For natural numbers $1\leq i,j\leq n,$ we denote by
\begin{equation}\label{eq:swap}
\sigma_{i,j}\colon\cube{n}\toiso\cube{n}
\end{equation}
the symmetry constraint which swaps the $i$-th and the $j$-th coordinates. We also use the same symbol for suitable restrictions, for instance to the subposets in \autoref{notn:source-punctured} such as
\[
\sigma_{i,j}\colon\cube{n}_{\leq k}\toiso\cube{n}_{\leq k}\qquad\mbox{for}\qquad 0\leq k\leq n.
\]
\end{enumerate}
\end{notn}

\begin{lem}\label{lem:he-swap-coor}
For natural numbers $n\geq 2,$ $0\leq k\leq l\leq n,$ and $1\leq i,j\leq n$ the following commutative squares are homotopy exact,
\[
\xymatrix{
\cube{n}_{\leq k}\ar[r]^-{\sigma_{i,j}}\ar[d]_-{\iota_{l,k}}\drtwocell\omit{\id}&\cube{n}_{\leq k}\ar[d]^-{\iota_{l,k}}&&
\cube{n}_{\geq l}\ar[r]^-{\sigma_{i,j}}\ar[d]_-{\iota_{k,l}}&\cube{n}_{\geq l}\ar[d]^-{\iota_{k,l}}\\
\cube{n}_{\leq l}\ar[r]_-{\sigma_{i,j}}&\cube{n}_{\leq l},&&
\cube{n}_{\geq k}\ar[r]_-{\sigma_{i,j}}&\cube{n}_{\geq k}.\ultwocell\omit{\id}
}
\]
\end{lem}
\begin{proof}
We give a proof for the square on the left, the other case is similar. In order to show that the canonical mate $(\iota_{l,k})_!\sigma_{i,j}^\ast\to\sigma_{i,j}^\ast(\iota_{l,k})_!$ is an isomorphism, we consider its component at an arbitrary object $x\in\cube{n}_{\leq l}$. In the pasting on the left in the diagram
\[
\xymatrix{
(\iota_{l,k}/x)\ar[r]\ar[d]_-\pi\drtwocell\omit{}&\cube{n}_{\leq k}\ar[r]^-{\sigma_{i,j}}\ar[d]_-{\iota_{l,k}}\drtwocell\omit{\id}&\cube{n}_{\leq k}\ar[d]^-{\iota_{l,k}}\ar@{}[rrd]|{=}&&
(\iota_{l,k}/x)\ar[r]^-{\sigma_{i,j}}\ar[d]_-\pi\drtwocell\omit{\id}&(\iota_{l,k}/\sigma_{i,j}x)\ar[r]\ar[d]_-\pi\drtwocell\omit{}&\cube{n}_{\leq k}\ar[d]^-{\iota_{l,k}}\\
\bbone\ar[r]_-x&\cube{n}_{\leq l}\ar[r]_-{\sigma_{i,j}}&\cube{n}_{\leq l}&&
\bbone\ar[r]_-=&\bbone\ar[r]_-{\sigma_{i,j}x}&\cube{n}_{\leq l},
}
\]
the square on the left is a slice square and it hence remains to show that this pasting is homotopy exact. Since $\iota_{l,k}$ commutes with $\sigma_{i,j}$, there is an induced functor 
\[
\sigma_{i,j}\colon(\iota_{l,k}/x)\to(\iota_{l,k}/\sigma_{i,j}x)\colon (y, f\colon\iota_{l,k} y\to x)\mapsto (\sigma_{i,j} y,\sigma_{i,j}f\colon\iota_{l,k}\sigma_{i,j}y\to\sigma_{i,j}x),
\]
which clearly is an isomorphism and hence homotopy final. In the above pasting on the right the square on the right is a slice square, and we deduce that that pasting is homotopy exact. Since both pastings agree, this concludes the proof.
\end{proof}

Thus, in every derivator there are canonical isomorphism
\begin{equation}\label{eq:sign-rep}
(\iota_{l,k})_!\sigma_{i,j}^\ast\toiso\sigma_{i,j}^\ast(\iota_{l,k})_!\quad\mbox{and}\quad
\sigma_{i,j}^\ast(\iota_{k,l})_\ast\toiso(\iota_{k,l})_\ast\sigma_{i,j}^\ast.
\end{equation}

\begin{cor}\label{cor:swap-cocart}
Let \D be a derivator, let $n\geq 2$, and let $1\leq i,j\leq n$. The restriction morphism $\sigma_{i,j}^\ast\colon\D^{\cube{n}}\to\D^{\cube{n}}$ preserves and reflects cocartesian and cartesian $n$-cubes.
\end{cor}
\begin{proof}
Let $\iota=\iota_{\leq n-1}$ and let $X\in\D^{\cube{n}}$. By \autoref{rmk:cocart}, the $n$-cubes $X$ and $\sigma_{i,j}^\ast X$ are cocartesian if and only if the corresponding component of the adjunction counit $\varepsilon\colon \iota_!\iota^\ast\to\id$ is invertible. To relate theses two characterization, we consider the following pastings
\[
\xymatrix{
\cube{n}_{\leq n-1}\ar[r]^-\iota\ar[d]_-\iota\drtwocell\omit{\id}&\cube{n}\ar[r]^-{\sigma_{i,j}}\ar[d]_-=\drtwocell\omit{\id}&\cube{n}\ar[d]^-=\ar@{}[rrd]|{=}&&
\cube{n}_{\leq n-1}\ar[r]^-{\sigma_{i,j}}\ar[d]_-\iota\drtwocell\omit{\id}&\cube{n}_{\leq n-1}\ar[r]^-\iota\ar[d]^-\iota\drtwocell\omit{\id}&\cube{n}\ar[d]^-=\\
\cube{n}\ar[r]_-=&\cube{n}\ar[r]_-{\sigma_{i,j}}&\cube{n}&&
\cube{n}\ar[r]_-{\sigma_{i,j}}&\cube{n}\ar[r]_-=&\cube{n}
}
\]
which trivially agree. The functoriality of canonical mates with respect to pasting implies that
\[
\xymatrix{
\iota_!\sigma_{i,j}^\ast \iota^\ast\ar[r]^-\sim\ar[d]_-=&\sigma_{i,j}^\ast\iota_!\iota^\ast\ar[d]^-{\sigma_{i,j}^\ast\varepsilon}\\
\iota_!\iota^\ast\sigma_{i,j}^\ast\ar[r]_-{\varepsilon\sigma_{i,j}^\ast}&\sigma_{i,j}^\ast
}
\]
commutes. The horizontal morphism is invertible by \autoref{lem:he-swap-coor}, and we conclude the proof by observing that $\sigma_{i,j}^\ast$ reflects isomorphisms. The case of cartesian $n$-cubes is dual.
\end{proof}

\begin{cor}\label{cor:swap-cocart-II}
Let \D be a derivator, let $n\geq 2$, and let $1\leq i,j\leq n$. The restriction morphism $\sigma_{i,j}^\ast\colon\D^{\cube{n}}\to\D^{\cube{n}}$ preserves and reflects strongly cocartesian and strongly cartesian $n$-cubes.
\end{cor}
\begin{proof}
Invoking \autoref{lem:he-swap-coor} in the case of $k=1$ and $l=n$ this time, the proof is completely parallel to the one of \autoref{cor:swap-cocart}. Alternatively, the result follows from \autoref{cor:swap-cocart} and \autoref{thm:cocart}.
\end{proof}

\begin{rmk}
Since canonical mates are compatible with respect to pasting, these results extend from $\sigma_{i,j}$ to more general symmetries induced by permutations.
\end{rmk}

\section{Decompositions of colimits}
\label{sec:decompose}

In this section we establish a general decomposition result for colimits in derivators (\autoref{thm:decompose}). As a special case, we deduce that the classical Bousfield--Kan formulas for homotopy colimits \cite{bousfield-kan:htpy-limits} hold in every derivator (\autoref{thm:BK-formulas}). As an additional application we obtain a decomposition result for suitable exhaustive filtrations of small categories (\autoref{thm:exhaustive-filt}). In \S\ref{sec:punctured} we specialize the latter result to more specific situations related to $n$-cubes.

\begin{notn}
Let $C\in\cCat$ and $F\colon C\to\cCat$. We denote by $\int_C F=\int_{c\in C} Fc\in\cCat$ the Grothendieck construction with objects the pairs $(c\in C,x\in Fc)$. A morphism $(f,g)\colon (c,x)\to(c',x')$ is a pair consisting of a morphism $f\colon c\to c'$ in $C$ and a morphism $g\colon (Ff)(x)\to x'$ in $Fc'$. We denote by 
\[
q=q_F\colon\int_C F\to C\colon (c,x)\mapsto c
\]
the canonical projection functor, which is a Grothendieck opfibration.

For every $c\in C$ there is a canonical inclusion functor $i_c\colon Fc\to\int_C F\colon x\mapsto (c,x)$ and which sends a morphism $g\colon x\to x'$ to $(\id_c,g)$. We note that $i_c$ induces an isomorphism of categories $Fc\toiso q^{-1}(c)$, as is witnessed by the pullback square
\begin{equation}\label{eq:opfibration}
\vcenter{
\xymatrix{
Fc\ar[r]^-{i_c}\ar[d]_-\pi\pullbackcorner&\int_CF\ar[d]^-q\\
\bbone\ar[r]_-c&C.
}
}
\end{equation}
\end{notn}

One nice feature of Grothendieck opfibrations is that left Kan extensions along them are given by `integration over fibers' in the following precise sense.

\begin{lem}\label{lem:opfibration}
For every Grothendieck opfibration $q\colon E\to B$ in $\cCat$ and every functor $f\colon B'\to B$ the pullback square
\[
\xymatrix{
E'\ar[r]\ar[d]\pullbackcorner\drtwocell\omit{\id}& E\ar[d]^-q\\
B'\ar[r]_-f&B
}
\]
is homotopy exact.
\end{lem}
\begin{proof}
This is \cite[Prop.~1.24]{groth:ptstab}.
\end{proof}

\begin{cor}\label{cor:opfibration}
Let \D be a derivator and let $q\colon E\to B$ be a Grothendieck opfibration in $\cCat$. For every $X\in\D^E$
\begin{enumerate}
\item there is a canonical isomorphism $\colim_Bq_!(X)\toiso\colim_E X$ and
\item for $b\in B$ there are canonical isomorphisms $\colim_{q^{-1}(b)}i_b^\ast X\toiso q_!(X)_b$, where $i_b\colon q^{-1}(b)\to E$ is the inclusion functor.
\end{enumerate}
\end{cor}
\begin{proof}
The first statement is immediate from uniqueness of left adjoints applied to $\pi_E=\pi_B\circ q\colon E\to\bbone$. We obtain in every derivator \D a canonical isomorphism
\[
\colim_{q^{-1}(b)}i_b^\ast\toiso b^\ast q_!
\]
by specializing \autoref{lem:opfibration} to $f=b\colon\bbone\to B$ (hence essentially to diagrams of the form \eqref{eq:opfibration}).
\end{proof}

The Fubini theorem $\colim_A\colim_B\toiso\colim_{A\times B}$ can be considered as a special case of this result (the case of the `untwisted product').

\begin{rmk}\label{rmk:opfibration}
\autoref{cor:opfibration} makes more precise in which sense left Kan extension along a Grothendieck opfibration $q\colon E\to B$ is calculated by `integration over fibers'. One can calculate $\colim_E$ by first calculating the colimits over $q^{-1}(b),b\in B,$ and then forming a colimit over these. Given a derivator \D and $X\in\D^E$ we write $X\!\!\mid_{q^{-1}(b)}$ for $i_b^\ast X$, and, for simplicity, we refer to the conclusion of \autoref{cor:opfibration} by saying that there is a canonical isomorphism
\[
\colim_{b\in B}\colim_{q^{-1}(b)}\big(X\!\!\mid_{q^{-1}(b)}\!\!\big)\toiso\colim_E X.
\]
This is a bit vague in view of the necessary distinction between coherent and incoherent diagrams; but whenever we write such a formula, this implicitly means that there is coherent lift of the (partially) incoherent diagram
\[
b\mapsto\colim_{q^{-1}(b)}\big(X\!\!\mid_{q^{-1}(b)}\!\!\big).
\]
\end{rmk}

A combination of these integrations over fibers together with the concept of homotopy finality provides a rather rich supply of more general decomposition results for colimits, and we hence make the following definition.

\begin{defn}\label{defn:decompose}
A \textbf{left decomposition} of a small category $A$ is a triple $(C,F,u)$ consisting of
\begin{enumerate}
\item a small category $C$,
\item a functor $F\colon C\to\cCat$,
\item and a homotopy final functor $u\colon\int_C F\to A$.
\end{enumerate}
\end{defn}

\begin{notn}\label{notn:cocone}
The \textbf{cocone} $A^\rhd\in\cCat$ of a small category $A\in\cCat$ is obtained from $A$ by freely adjoining a new terminal object $\infty$. It is the join $A^\rhd=A*\bbone$ of $A$ and the terminal category $\bbone$, and there are hence fully faithful inclusions $i_A\colon A\to A^\rhd$ and $\infty\colon\bbone\to A^\rhd$ which are jointly surjective on objects. 

Correspondingly, given a derivator \D, associated to $i_A$ there is the fully faithful left Kan extension morphism $(i_A)_!\colon\D^A\to\D^{A^\rhd}$. The essential image of $(i_A)_!$ consists by definition of the colimiting cocones. A diagram $X\in\D^{A^\rhd}$ is a colimiting cocone if and only if the counit $\varepsilon\colon (i_A)_!i_A^\ast X\toiso X$ is invertible if and only if the canonical mate $\colim_X(X\!\!\mid_A)\toiso X_\infty$ is an isomorphism (see \cite[\S2]{groth:revisit} for more details).
\end{notn}

\begin{thm}[\textbf{Decomposition theorem}]\label{thm:decompose}
Let $A\in\cCat$, let $(C,F,u)$ be a left decomposition of $A$, and let \D be a derivator. There is a morphism of derivators
\[
d\colon\D^A\to\D^{C^\rhd}
\]
such that for every $X\in\D^A$
\begin{enumerate}
\item the cocone $d(X)\in\D^{C^\rhd}$ is colimiting, 
\item there is a canonical isomorphism $d(X)_\infty\toiso\colim_AX,$ and
\item there are canonical isomorphisms $\colim_{F(c)}(X\!\!\mid_{F(c)})\toiso d(X)_c, c\in C$, where $X\!\!\mid_{F(c)}$ is shorthand notation for $(i_c^\ast u^\ast)X$.
\end{enumerate}
\end{thm}
\begin{proof}
The homotopy finality of $u\colon\int_C F\to A$ is equivalent to the homotopy exactness of the square
\begin{equation}\label{eq:decomp-0}
\vcenter{
\xymatrix{
\int_C F\ar[r]^-u\ar[d]_-q&A\ar[dd]^-\pi\\
C\ar[d]_-\pi&\\
\bbone\ar[r]_-\id&\bbone.
}
}
\end{equation}
Thus, for $X\in\D^A$ and $Z=q_!u^\ast (X)\in\D^C$ there is a canonical isomorphism
\begin{equation}\label{eq:decomp-I}
\colim_C Z\toiso \colim_AX.
\end{equation}
Moreover, since $q$ is a Grothendieck opfibration, for every $c\in C$ the homotopy exactness of \eqref{eq:opfibration} yields canonical isomorphisms
\begin{equation}\label{eq:decomp-II}
\colim_{F(c)}(X\!\mid_{F(c)})\toiso Z_c,
\end{equation}
where $X\!\!\mid_{F(c)}$ is shorthand notation for $i_c^\ast u^\ast X$. To conclude the construction of $d\colon\D^A\to\D^{C^\rhd}$ we combine the above with the left Kan extension along the inclusion functor $i_C\colon C\to C^\rhd,$ i.e., we define the decomposition morphism as
\[
d=(i_C)_!q_!u^\ast\colon\D^A\to\D^{\int_CF}\to\D^C\to\D^{C^\rhd}.
\]
The homotopy exactness of the slice square
\[
\xymatrix{
C\ar[r]^-\id\ar[d]_-\pi\drtwocell\omit&C\ar[d]^-{i_C}\\
\bbone\ar[r]_-\infty&C^\rhd
}
\]
and the homotopy exactness of \eqref{eq:decomp-0} yield canonical isomorphisms
\[
\infty^\ast d=\infty^\ast(i_C)_!q_!u^\ast\cong \colim_Cq_!u^\ast\cong\colim_A.
\]
This settles the second property and the first one is obvious by definition of colimiting cocones (see \autoref{notn:cocone}). Since $i_C\colon C\to C^\rhd$ is fully faithful, the bottom square in 
\[
\xymatrix{
F(c)\ar[r]^-{i_c}\ar[d]_-\pi\pullbackcorner\drtwocell\omit{\id}&\int_C F\ar[d]^-q\\
\bbone\ar[r]_-c\ar[d]_-\id\drtwocell\omit{\id}&C\ar[d]^-{i_C}\\
\bbone\ar[r]_-c&C^\rhd
}
\]
is homotopy exact. This yields canonical isomorphism
\[
c^\ast d=c^\ast(i_C)_!q_!u^\ast\cong c^\ast q_!u^\ast\cong \colim_{F(c)} i_c^\ast u^\ast,
\]
thereby concluding the proof.
\end{proof}

\begin{rmk}
\begin{enumerate}
\item Formulas \eqref{eq:decomp-I} and \eqref{eq:decomp-II} jointly indicate in which sense a left decomposition of $A$ allows us to decompose the calculation of $\colim_A$ into intermediate steps. While this is made precise by the statement of \autoref{thm:decompose}, following \autoref{rmk:opfibration} we refer to this by saying that for $X\in\D^A$ there are canonical isomorphisms 
\begin{equation}\label{eq:vague}
\colim_{c\in C}\colim_{Fc}(X\!\!\mid_{Fc})\toiso\colim_AX.
\end{equation}
\item There is a dual notion of \textbf{right decompositions}, based on homotopy initial functors defined on Grothendieck constructions associated to presheaves of small categories, and there is a dual decomposition theorem.
\item Both versions of the decomposition theorem can be generalized to decompositions which are based on \emph{smooth} and \emph{proper} functors instead (see \cite{maltsiniotis:grothendieck,maltsiniotis:asphericity}).
\end{enumerate}
\end{rmk}

\begin{rmk}\label{rmk:htpy-final}
Let us recall that a functor $u\colon A\to B$ in $\cCat$ is homotopy final if and only if for every $b\in B$ the comma category $(b/u)$ has a weakly contractible nerve (see \cite[\S3]{gps:mayer} which relies on \cite{heller:htpythies,cisinski:presheaves}). Thus, in the situation of \autoref{thm:decompose}, we obtain a decomposition result for colimits whenever we can show that the slice categories $(a/q),a \in A,$ have weakly contractible nerves. This is, in particular, the case whenever $(a/q)$ can be connected to the terminal category $\bbone$ by a string of adjoint functors (possibly pointing in different directions). 
\end{rmk}

As a first application we note that the classical Bousfield--Kan formulas \cite{bousfield-kan:htpy-limits} are \emph{consequences} of having a derivator. In fact, we obtain this by specializing the two-sided simplicial bar resolutions for coends (\cite[Appendix~A]{gps:additivity}) to the case of \emph{dummy variables} in the sense of \cite{maclane}. We begin by recalling coends in derivators and the corresponding bar construction.

\begin{con}\label{con:two-sided-bar}
Let $A\in\cCat$ and let $\tw(A)$ be the twisted morphism category of $A$ with objects the morphisms $f\colon a\to b$ in $A$ and morphisms $f\to f'$ the two-sided factorizations of $f'$ through $f$,
\[
\xymatrix{
a\ar[r]^f&b\ar[d]\\
a'\ar[u]\ar[r]_-{f'}&b'.
}
\]
For every derivator \D the coend morphism is defined by
\[
\int^A=\colim_{\tw(A)\op}\circ(t\op,s\op)\colon\D^{A\op\times A}\to\D^{\tw(A)\op}\to\D,
\]
where $t$ and $s$ are the evaluation maps at targets and sources, respectively \cite[\S5]{gps:additivity}. We denote by $(\Delta/A)$ the category of simplices of $A$ (the category of elements of the nerve $NA$) with objects the strings $([n],a_0\to\ldots\to a_n)$ of composable morphisms in $A$. Forming the compositions $([n],a_0\to\ldots\to a_n)\mapsto (a_0\to a_n)$ defines a functor
\[
c\colon(\Delta/A)\op\to\tw(A)\op,
\]
which is known to be homotopy final (\cite[Appendix~A]{gps:additivity}). Finally, the functor
\[
p\colon(\Delta/A)\op\to\Delta\op\colon([n],a_0\to\ldots\to a_n)\mapsto[n]
\]
is a Grothendieck opfibration with discrete fibers, and the two-sided simplicial bar construction for coends in derivators is obtained by contemplating the homotopy exact square
\[
\xymatrix{
(\Delta/A)\op\ar[d]_-p\ar[r]^-c&\tw(A)\op\ar[dd]^-\pi\\
\Delta\op\ar[d]_-\pi&\\
\bbone\ar[r]_-\id&\bbone
}
\] 
\end{con}

\begin{prop}\label{prop:two-sided-bar}
Let \D be a derivator and $A\in\cCat$. There is a simplicial resolution morphism
\[
s\colon\D^{A\op\times A}\to\D^{\Delta\op}
\]
such that for every $X\in\D^{A\op\times A}$
\begin{enumerate}
\item there is a canonical isomorphism $\colim_{\Delta\op}sX\toiso\int^AX$ and
\item for every $[n]\in\Delta\op$ there are canonical isomorphisms
\[
(sX)_n\toiso\coprod_{(a_0\to\ldots \to a_n)\in (NA)_n}X(a_n,a_0).
\]
\end{enumerate}
\end{prop}
\begin{proof}
This follows from the claims made in \autoref{con:two-sided-bar} and which are proved in \cite[Appendix~A]{gps:additivity}.
\end{proof}

\begin{lem}\label{lem:source-htpy-final}
For every $A\in\cCat$ the functor $s\op\colon\tw(A)\op\to A$ is homotopy final.
\end{lem}
\begin{proof}
This follows from the proof of \cite[Lem.~5.4]{gps:additivity}
\end{proof}

We recall that we denote by $\pi_A\colon A\times B\to B$ the functor which projects away the category $A$.

\begin{cor}\label{cor:dummy}
For every derivator \D and $A\in\cCat$ there is a canonical isomorphism
$\int^A\circ\pi_{A\op}^\ast\cong\colim_A\colon\D^A\to\D,$
\[
\xymatrix{
\D^A\ar[r]^-{\pi_{A\op}^\ast}\ar[dr]_-{\colim_A}&\D^{A\op\times A}\ar[d]^-{\int^A}\\
&\D
}
\]
\end{cor}
\begin{proof}
By \autoref{lem:source-htpy-final} the square
\[
\xymatrix{
\tw(A)\op\ar[r]^-{(t\op,s\op)}\ar[d]_-\pi\drtwocell\omit{\id}&A\op\times A\ar[r]^-{\pi_{A\op}}&A\ar[d]^-\pi\\
\bbone\ar[rr]_-\id&&\bbone
}
\]
is homotopy exact and the result hence follows from the definition of coends in derivators (see \autoref{con:two-sided-bar}).
\end{proof}

This is a derivator version of the categorical result that coends with dummy variables are colimits. The following result is now an immediate consequence of earlier results but because of its importance we formulate it as a theorem.

\begin{thm}[\textbf{Bousfield--Kan formulas}]\label{thm:BK-formulas}
Let \D be a derivator and $A\in\cCat$. There is a simplicial resolution morphism
\[
s\colon\D^A\to\D^{\Delta\op}
\]
such that for every $X\in\D^A$
\begin{enumerate}
\item there is a canonical isomorphism $\colim_{\Delta\op}sX\toiso\colim_AX$ and
\item for every $[n]\in\Delta\op$ there are canonical isomorphisms
\[
(sX)_n\toiso\coprod_{(a_0\to\ldots \to a_n)\in (NA)_n}X_{a_0}.
\]
\end{enumerate}
\end{thm}
\begin{proof}
This follows from \autoref{prop:two-sided-bar} and \autoref{cor:dummy}.
\end{proof}

Following the conventions for \eqref{eq:vague}, there is hence a canonical isomorphism
\[
\colim_{[n]\in\Delta\op}\big(\coprod_{a_0\to\ldots \to a_n}X_{a_0}\big)\toiso\colim_AX.
\]
In low dimensions this simplicial diagram looks like
\[
\xymatrix{
\cdots\quad\ar@{<-}@<1.5mm>[r] \ar@{<-}@<-1.5mm>[r] \ar[r] \ar@<3mm>[r] \ar@<-3mm>[r]&\displaystyle\coprod_{a_0\to a_1}X_{a_0}\ar@<1.5mm>[r] \ar@<-1.5mm>[r] \ar@{<-}[r] &
    \displaystyle\coprod_a^{\vphantom{a_n}} X_a,
}
\]
as we expect of a Bousfield--Kan formula \cite{bousfield-kan:htpy-limits}.

\begin{rmk}
Presenting the arguments in a slightly different way, we see that \autoref{thm:BK-formulas} and \autoref{prop:two-sided-bar} are special cases of the decomposition theorem (\autoref{thm:decompose}).
\end{rmk}

\begin{defn}
Let $F\colon\D\to\E$ be a morphism of derivators.
\begin{enumerate}
\item The colimit morphism $\colim_{\Delta\op}\colon\D^{\Delta\op}\to\D$ is the \textbf{geometric realization} morphism of \D.
\item The morphism $F$ \textbf{preserves geometric realizations} if $F$ preserves colimits of shape $\Delta\op$.
\end{enumerate}
\end{defn}

\begin{cor}
Let $F\colon\D\to\E$ be a morphism of derivators.
\begin{enumerate}
\item The morphism $F$ preserves all colimits if and only if $F$ preserves coproducts and geometric realizations.
\item A replete subcategory $\cC\subseteq\D(\bbone)$ is closed under all colimits if and only if it is closed under coproducts and geometric realizations.
\end{enumerate}
\end{cor}
\begin{proof}
Let $F\colon\D\to\E$ be a morphism of derivators which preserves coproducts and geometric realizations, and let $A\in\cCat$. We begin by noting that
\[
s\op\circ c\colon(\Delta/A)\op\to\tw(A)\op\to A\colon (a_0\to\ldots\to a_n)\mapsto a_0
\]
is homotopy final as composition of the homotopy final functors $s\op$ (\autoref{lem:source-htpy-final}) and $c$ (\autoref{con:two-sided-bar}). By \cite[Lem.~5.8]{groth:revisit} it remains to show that $F$ preserves colimits of shape $(\Delta/A)\op$. By assumption we only have to verify that $F$ preserves left Kan extensions along $p\colon(\Delta/A)\op\to\Delta\op$. By \cite[Prop.~2.3]{groth:ptstab} this is the case as soon as $F$ preserves colimits of shape $(p/[n]),[n]\in\Delta\op$. Since $p$ is a discrete opfibration, the functors $p^{-1}([n])\to(p/[n])$ are homotopy final with domains given by discrete categories, and the claim follows from our assumption on $F$ and an additional application of \cite[Lem.~5.8]{groth:revisit}. The second statement is similar.
\end{proof}

We now specialize the decomposition theorem (\autoref{thm:decompose}) to suitable filtrations of small categories. 

\begin{defn}
For every $A\in\cCat$ we denote by $\mathrm{Sub}(A)\in\cCat$ the following small category.
\begin{enumerate}
\item Objects are subcategories $A'\to A$ of $A$.
\item Given two subcategories $A'\to A$ and $A''\to A$, a morphism from $A'$ to $A''$ is a functor $A'\to A''$ such that
\[
\xymatrix{
A'\ar[dr]\ar[rr]&&A''\ar[dl]\\
&A&
}
\]
commutes.
\end{enumerate}
The category $\mathrm{Sub}(A)$ is the \textbf{category of subcategories} of $A$. 
\end{defn}

\begin{defn}\label{defn:filtration}
A \textbf{filtration} of a small category $A$ is a pair $(C,F)$ consisting of
\begin{enumerate}
\item a small category $C$ and 
\item a functor $F\colon C\to\mathrm{Sub}(A).$
\end{enumerate}
A \textbf{filtration by full subcategories} is a filtration $(C,F)$ such that all $Fc, c\in C,$ are full subcategories of $C$. A filtration $(C,F)$ is \textbf{exhaustive} if $\bigcup_{c\in C} Fc=A$.
\end{defn}

\begin{rmk}\label{rmk:filtration}
Let $A\in\cCat$ and let $(F,C)$ be a filtration of $A$.
\begin{enumerate}
\item Morphisms in $\mathrm{Sub}(A)$ are necessarily faithful functors which are injective on objects. Hence, this applies to the functors $Ff\colon Fc\to Fc'$ in a filtration $F\colon C\to\mathrm{Sub}(A)$, thereby justifying the above terminology.
\item Note that whenever $(C,F)$ is a filtration by full subcategories, then all structure maps $Ff\colon Fc\to Fc'$ are fully faithful and injective on objects.
\end{enumerate}
\end{rmk}

\begin{con}\label{con:filt-decompose}
Let $A$ be a small category and let $(C,F)$ be a filtration of $A$. The inclusion functors $i_c=i_{Fc}\colon Fc\to A,c\in C,$ define a natural transformation to the constant diagram $\kappa_A\colon C\to\cCat\colon c\mapsto A$. The functoriality of the Grothendieck construction yields a functor $\int_C i\colon\int_C F\to \int _C \kappa_A=C\times A$, and by postcomposition with the projection $\pi\colon C\times A\to A$ we obtain
\begin{equation}\label{eq:filt-decompose}
u=u_{C,F}\colon \int_C F \to \int_C\kappa_A=C\times A\to A.
\end{equation}
\end{con}

\begin{defn}
A filtration $(C,F)$ of a small category $A$ is a \textbf{left filtration} if the resulting functor \eqref{eq:filt-decompose} is homotopy final, i.e., if $(C,F,u_{C,F})$ is a left decomposition of~$A$.
\end{defn}

\begin{thm}[\textbf{Decomposition by left filtrations}]\label{thm:filtration}
Let $A\in\cCat$, let $(C,F)$ be a left filtration of $A$, and let \D be a derivator. There is a decomposition morphism
\[
d\colon\D^A\to\D^C
\]
such that for every $X\in\D^A$
\begin{enumerate}
\item there is a canonical isomorphism $\colim_C dX\toiso \colim_AX$ and
\item for every $c\in C$ there is a canonical isomorphism
\[
(dX)_c\toiso\colim_{Fc}i_c^\ast X,
\]
where $i_c\colon Fc\to A$ is the inclusion functor.
\end{enumerate}
\end{thm}
\begin{proof}
Since \eqref{eq:filt-decompose} is by assumption homotopy final, $(C,F,u_{C,F})$ is a left decomposition of $A$ and the result follows from \autoref{thm:decompose}.
\end{proof}

We now show that an important class of left filtrations can be constructed from functors to posets.

\begin{con}\label{con:exhaustive-filt}
Let $A\in\cCat,$ let $P$ be a small poset, and let $f\colon A\to P$ be a functor. For every $p\in P$ we consider the full subcategories 
\[
\iota_{\leq p}\colon A_{\leq p}=\{x\in A\mid f(x)\leq p\}\to A.
\]
For $p\leq p'$ in $P$ there is a fully faithful inclusion functor $\iota_{p',p}\colon A_{\leq p}\to A_{\leq p'}$, and this assignment defines a functor
\begin{equation}\label{eq:filt-posets}
F=F_f\colon P\to\mathrm{Sub}(A)\colon p\mapsto A_{\leq p}.
\end{equation}
$F$ is an exhaustive filtration by full subcategories.
\end{con}

\begin{thm}[\textbf{Decomposition by exhaustive filtrations}]\label{thm:exhaustive-filt}
Let $A\in\cCat$, let $P\in\cCat$ be a poset, and let $f\colon A\to P$ be a functor. The filtration \eqref{eq:filt-posets} is a left filtration of $A$, the \textbf{left filtration associated to $f$}. In particular, for every derivator \D there is a decomposition morphism $d\colon\D^A\to\D^P$ such that for every $X\in\D^A$
\begin{enumerate}
\item there is a canonical isomorphism $\colim_P dX\toiso \colim_AX$ and
\item for every $p\in P$ there is a canonical isomorphism
\[
(dX)_p\toiso\colim_{A_{\leq p}}\iota_{\leq p}^\ast X.
\]
\end{enumerate}
\end{thm}
\begin{proof}
For every functor $f\colon A\to P$ we contemplate the following diagram
\begin{equation}\label{eq:}
\vcenter{
\xymatrix{
A_{\leq p}\ar[r]\ar[d]\pullbackcorner&(f/P)\ar[r]^-u\ar[d]_-q\drtwocell\omit{}&A\ar[d]^-f\\
\bbone\ar[r]_-p&P\ar[r]_-\id\ar[d]&P\ar[d]\\
&\bbone\ar[r]_-=&\bbone
}
}
\end{equation}
in $\cCat.$ The top square on the right is homotopy exact as a comma square (\cite[Prop.~1.26]{groth:ptstab}) and the functor $q$ is a Grothendieck opfibration. Since $P$ is a poset, the slice category $(f/p)$ can be identified with $A_{\leq p}$. Hence, for $F_f$ as in \eqref{eq:filt-posets} we obtain the identification $\int_PF_f=(f/P)$. The bottom square is homotopy exact as a constant square (alternatively, equivalences are homotopy final), and hence so is the vertical pasting of the two squares. But this is equivalent to the homotopy finality of $u$, which is to say that $(P,F_f)$ is a left filtration. The rest follows from \autoref{thm:filtration}.
\end{proof}

The decompositions of colimits (\autoref{thm:decompose}) and hence the formulas obtained in this section (such as \autoref{thm:BK-formulas}, \autoref{thm:filtration} or \autoref{thm:exhaustive-filt}) enjoy a certain compatibility with respect to morphisms of derivators, and we come back to this idea in \S\ref{sec:formulas}.

\section{Punctured $n$-cubes}
\label{sec:punctured}

In order to manipulate cocartesian $n$-cubes, by \autoref{rmk:cocart} it is convenient to have a better understanding of colimits of punctured $n$-cubes. In this short section we extend to derivators a well-known inductive formula for such colimits and illutrate it by some toy examples.

\begin{con}\label{con:prod-cocone}
Let $A,B\in\cCat$, and let $i_A\colon A\to A^\rhd,i_B\colon B\to B^\rhd$ be the fully faithful inclusion functors. The \emph{punctured product of the cocones} $A^\rhd$ and $B^\rhd$ is the full subcategory $(A^\rhd\times B^\rhd)-\{(\infty,\infty)\}$ obtained from $A^\rhd\times B^\rhd$ by removing the terminal object $(\infty,\infty)$. There are fully faithful inclusion functors
\[
\xymatrix{
A\times B\ar[r]^-{i_A}\ar[d]_-{i_B}&A^\rhd\times B\ar[d]^-{i_B}\\
A\times B^\rhd\ar[r]_-{i_A}&(A^\rhd\times B^\rhd)-\{(\infty,\infty)\}, 
}
\]
and this defines a diagram
\begin{equation}\label{eq:filt-prod-cocone}
F_{A,B}\colon\ulcorner\to\mathrm{Sub}((A^\rhd\times B^\rhd)-\{(\infty,\infty)\}).
\end{equation}
\end{con}

\begin{prop}[\textbf{Colimits of punctured products of cocones}]\label{prop:prod-cocone}
Let \D be a derivator and $A,B\in\cCat$. There is a morphism of derivators
\[
d\colon\D^{(A^\rhd\times B^\rhd)-\{(\infty,\infty)\}}\to\D^\square
\]
such that for every $X\in\D^{(A^\rhd\times B^\rhd)-\{(\infty,\infty)\}}$ the square $d(X)\in\D^\square$ is cocartesian with underlying diagram
\[
\xymatrix{
\colim_{A\times B} (X\!\!\mid_{A\times B})\ar[r]\ar[d]&\colim_B\infty_A^\ast X\ar[d]\\
\colim_A\infty_B^\ast X\ar[r]&\colim_{(A^\rhd\times B^\rhd)-\{(\infty,\infty)\}} X\pushoutcorner.
}
\]
\end{prop}
\begin{proof}
There is a unique functor $f_{A,B}\colon (A^\rhd\times B^\rhd)-\{(\infty,\infty)\}\to\ulcorner$ determined by
\[
(a,b)\mapsto (0,0),\quad (\infty,b)\mapsto(1,0),\quad\text{and}\quad (a,\infty)\mapsto(0,1)
\]
for all $a\in A$ and $b\in B$. Sticking to the notation of \autoref{con:exhaustive-filt}, we have $F_{A,B}=F_{f_{A,B}}$ and the result follows immediately from \autoref{thm:exhaustive-filt} and the homotopy finality of final objects.
\end{proof}

The formula for colimits of punctured $n$-cubes follows immediately from a formula for colimits of punctured cylinders of cones. Let us establish the relevant notation. 

\begin{con}\label{con:gen-punctured}
Given a small category $A\in\cCat$, the \emph{punctured cylinder on the cocone} is the punctured product of the cocones $A^\rhd$ and $\bbone^\rhd=[1]$, i.e., the category 
\[
(A^\rhd\times[1])-\{(\infty,1)\}
\]
obtained from the cylinder $A^\rhd\times[1]$ by removing the final object $(\infty,1)$. There are fully faithful inclusion functors
\[
\xymatrix{
A\ar[r]^-{i_A}\ar[d]_-{i_0}&A^\rhd\ar[d]^-{i_0}\\
A\times [1]\ar[r]_-{i_A}&(A^\rhd\times [1])-\{(\infty,1)\}, 
}
\]
In particular, for $k=0,1$ there are fully faithful inclusion functors
\[
j_k\colon A\to (A^\rhd\times [1])-\{(\infty,1)\}\colon a\mapsto (a,k).
\]
\end{con}

\begin{cor}[\textbf{Colimits of punctured cylinders of cocones}]\label{cor:cyl-cocone}
Let \D be a derivator and $A,B\in\cCat$. There is a morphism of derivators
\[
d\colon\D^{(A^\rhd\times [1])-\{(\infty,1)\}}\to\D^\square
\]
such that for every $X\in\D^{(A^\rhd\times [1])-\{(\infty,1)\}}$ the square $d(X)\in\D^\square$ is cocartesian with underlying diagram
\[
\xymatrix{
\colim_A j_0^\ast X\ar[d]\ar[r]&X_{\infty,0}\ar[d]\\
\colim_A j_1^\ast X\ar[r]&\colim X.\pushoutcorner
}
\]
\end{cor}
\begin{proof}
This is a special case of \autoref{prop:prod-cocone}.
\end{proof}

We now specialize further to the case that $A=\cube{n-1}_{\leq n-2}$ is a punctured $(n-1)$-cube. Its cocone is $\cube{n-1}$ and the punctured cylinder on $\cube{n-1}$ is the punctured $n$-cube $\cube{n}_{\leq n-1}$. In the following corollary we denote the object $(1,\ldots,1,0)\in\cube{n}_{\leq n-1}$ by $(\infty,0)$.

\begin{cor}[\textbf{Colimits of punctured $n$-cubes}]\label{cor:pun-cube}
For every derivator \D and $n\geq 2$ there is a morphism of derivators $d\colon\D^{\cube{n}_{\leq n-1}}\to\D^\square$ such that for every $X\in\D^{\cube{n}_{\leq n-1}}$ the square $d(X)\in\D^\square$ is cocartesian with underlying diagram
\[
\xymatrix{
\colim_{\cube{n-1}_{\leq n-2}} j_0^\ast X\ar[d]\ar[r]&X_{\infty,0}\ar[d]\\
\colim_{\cube{n-1}_{\leq n-2}} j_1^\ast X\ar[r]&\colim_{\cube{n}_{\leq n-1}} X.\pushoutcorner
}
\]
\end{cor}
\begin{proof}
This is a special case of \autoref{cor:cyl-cocone}.
\end{proof}

\begin{rmk}
\begin{enumerate}
\item \autoref{cor:pun-cube} yields an inductive procedure to calculate colimits of higher dimensional punctured $n$-cubes in terms of colimits of lower dimensional ones. If we specialize  to the degenerate case $n=2$, then the corollary reconfirms that we can compute a pushout as a pushout. 
\item Despite having been presented this way, the final coordinate in the punctured $n$-cube, of course, does not play a particular role (as one checks by invoking the restriction along the symmetries $\sigma_{i,j}$ from \autoref{notn:cotrunc}). 
\end{enumerate}
\end{rmk}

We illustrate this inductive recipe by the following two examples. 

\begin{eg}\label{eg:higher-susp}
For every pointed derivator \D and $n\geq 1$ there is a natural isomorphism
\[
\Sigma^{n-1}\cong \colim_{\cube{n}_{\leq n-1}}\circ\emptyset_\ast\colon\D\to\D.
\]
Moreover, the canonical isomorphisms \eqref{eq:sign-rep} induce the action of the symmetric group $\Sigma_{n-1}$ given by the signatures of the permutations,
\[
\Sigma_{n-1}\to\Aut(\Sigma^{n-1}x)\colon \sigma\mapsto\mathrm{sign}(\sigma)\id,\qquad x\in\D.
\]
\end{eg}
\begin{proof}
The inclusion $\emptyset=\iota_{\leq 0}\colon\bbone\to\cube{n}_{\leq n-1}, n\geq 2,$ of the initial object is a sieve. Hence, for every pointed derivator \D, the morphism $\emptyset_\ast\colon\D\to\D^{\cube{n}_{\leq n-1}}$ is right extension by zero objects (\cite[Prop.~3.6]{groth:ptstab}). For $n=2$ the claim follows by definition of suspension in pointed derivators. For $n\geq 3$ and $x\in\D$, we obtain by \autoref{cor:pun-cube} a cocartesian square
\[
\xymatrix{
\colim_{\cube{n-1}_{\leq n-2}}\emptyset_\ast x\ar[r]\ar[d]&0\ar[d]\\
0\ar[r]&\colim_{\cube{n}_{\leq n-1}}\emptyset_\ast x.\pushoutcorner
}
\]
By induction this suspension square induces the intended natural isomorphism. The symmetry $\sigma_{i,j}$ coming from the transposition $(ij)$ acts by the sign $-1$ (as it follows from \cite[Prop.~4.12]{groth:ptstab} by passing to parametrized Kan extensions), and the general case follows from the functoriality of canonical mates with respect to pasting.
\end{proof}

\begin{eg}
Let \D be derivator and let $n\geq 2$. Since $\cube{n}_{=1}$ is discrete on $n$ objects, there is an equivalence of derivators $\D^{\cube{n}_{=1}}\toiso\D\times\ldots\times\D$, and $(\iota_{=1})_!$ induces a morphism
\[
\D\times\ldots\times\D\simeq\D^{\cube{n}_{=1}}\to\D^{\cube{n}}.
\]
This morphism forms strongly cocartesian coproduct $n$-cubes (see \cite[\S4]{gst:tree}), and the inductive formula \autoref{cor:pun-cube} reduces to the calculation of finite coproducts. The canonical mates \eqref{eq:sign-rep} induce the symmetry constraints of the cocartesian monoidal structure.
\end{eg}

As a closely related example we also collect the following one dealing with iterated pushouts of sources of higher valence.

\begin{eg}\label{eg:colim-of-sources}
Let $\D$ be a derivator, let $n\geq 2$, and let $X\in\D^{\cube{n}_{\leq 1}}$ be a representation of the source of valence $n$ with underlying diagram looking like
\[
\xymatrix{
&&X_0\ar[rrd]\ar[rd]\ar[ld]\ar[lld]\ar@{}[d]|{\cdots}&&\\
X_1&X_2&\ldots&X_{n-1}&X_n.
}
\]
We denote the colimit of the coherent source $X$ (or, equivalently, the colimit of the punctured $n$-cube $(\iota_{n-1,1})_!X$) by
\[
\colim_{\cube{n}_{\leq 1}} X=X_1\cup_{X_0}X_2\cup_{X_0}\ldots\cup_{X_0}X_n.
\]
This notation is justified by the existence of canonical isomorphisms
\begin{align*}
\colim_{\cube{n}_{\leq 1}} X&=X_1\cup_{X_0}X_2\cup_{X_0}\ldots\cup_{X_0}X_n\\
&\cong X_1\cup_{X_0} \big(X_2\cup_{X_0}\ldots\cup_{X_0}X_n\big).
\end{align*}
In fact, this follows from \autoref{thm:cocart} applied to the strongly cocartesian $n$-cube $(\iota_{n,1})_!X$ together with the pasting property of cocartesian squares. For instance, in the case of $n=3$ this isomorphism is induced from `the diagonal square' in \eqref{eq:source-val-three} (diagonal in the $(2-3)$-direction).

\begin{equation}\label{eq:source-val-three}
\vcenter{
\xymatrix@-0.5pc{
X_0\ar[rr]\ar[dr]\ar[dd] && X_1 \ar[dr] \ar@{-->}'[d][dd]\\
& X_2 \ar[rr] \ar[dd] && X_1\cup_{X_0} X_2\ar[dd] \\
 X_3 \ar[dr] \ar@{-->}'[r][rr] && X_1\cup_{X_0} X_3\ar@{-->}[dr]\\
& X_2\cup_{X_0} X_3\ar[rr] &&\colim X 
}
}
\end{equation}

An alternative way to calculate this colimit, again using the pasting property, is as
\begin{align*}
\colim_{\cube{n}_{\leq 1}} X&\cong X_1\cup_{X_0} \big(X_2\cup_{X_0}\ldots\cup_{X_0}X_n\big)\\
&\cong \big(X_1\cup_{X_0} X_2)\cup_{X_2}\big(X_2\cup_{X_0}\ldots\cup_{X_0}X_n\big)\\
&\cong \big(X_1\cup_{X_0} X_2)\cup_{X_2}\big(X_2\cup_{X_0}X_3\big)\cup_{X_3}\ldots\cup_{X_{n-1}}
\big(X_{n-1}\cup_{X_0}X_n\big).
\end{align*}
Here, the last isomorphism follows from induction from the second one. For instance, again in the case of $n=3$, the second isomorphism is induced from the cocartesian square in the front of \eqref{eq:source-val-three}.

Moreover, the canonical mates $(\iota_{n,1})_!\sigma^\ast\toiso\sigma^\ast(\iota_{n,1})_!$ \eqref{eq:sign-rep} are invertible for every permutation $\sigma\in\Sigma_n$. In particular, by evaluation at $\infty\in\cube{n}$, this yields canonical isomorphisms
\[
\colim_{\cube{n}_{\leq 1}}\sigma^\ast X\toiso\colim_{\cube{n}_{\leq 1}}X,
\]
which, in the above notation, can be read as isomorphisms
\[
X_{\sigma(1)}\cup_{X_0}X_{\sigma(2)}\cup_{X_0}\ldots\cup_{X_0}X_{\sigma(n)}\toiso X_1\cup_{X_0}X_2\cup_{X_0}\ldots\cup_{X_0}X_n.
\] 
\end{eg}

\section{Cotruncated $n$-cubes}
\label{sec:cotruncated}

In this section we revisit the cardinality filtration of the $n$-cube $\cube{n}$ (see \autoref{fig:intro}). More specifically, we focus on left Kan extensions along functors in the first row,
\[
\xymatrix{
\emptyset\ar[r]&\cube{n}_{=0}\ar[r]&\cube{n}_{\leq 1}\ar[r]&\ldots\ar[r]&\cube{n}_{\leq n-1}\ar[r]&\cube{n}.
}
\]
By means of these left Kan extensions, at the level of coherent $n$-cubes in derivators we obtain an interpolation between cocartesian and strongly cocartesian $n$-cubes.

We keep using the notation as in \autoref{notn:source-punctured} and \autoref{notn:cotrunc}. In particular, for indices $n\geq 0$ and $-1\leq k\leq l\leq m\leq n$ we clearly have
\[
\iota_{m,l}\circ\iota_{l,k}=\iota_{m,k}\colon\cube{n}_{\leq k}\to\cube{n}_{\leq m}.
\]

\begin{defn}
Let \D be a derivator, $n\geq 0$, and $-1\leq k\leq l\leq n$. A coherent diagram $X\in\D^{\cube{n}_{\leq l}}$ is \textbf{$k$-cotruncated} if it lies in the essential image of
\[
(\iota_{l,k})_!\colon\D^{\cube{n}_{\leq k}}\to\D^{\cube{n}_{\leq l}}.
\]
\end{defn}

\begin{rmk}
Let \D be a derivator, $n\geq 0$, and $-1\leq l\leq k\leq n$.
\begin{enumerate}
\item Similarly, $X\in\D^{\cube{n}_{\geq l}}$ is \textbf{$k$-truncated} if it lies in the essential image of the corresponding right Kan extension. 
\item For $n\geq 2$, an $n$-cube is cocartesian if and only if it is $(n-1)$-cotruncated, and it is strongly cocartesian if and only if it is $1$-cotruncated.
\end{enumerate} 
\end{rmk}

In the case of proper $n$-cubes we use the following notation.

\begin{notn}
Let \D be a derivator, let $n\geq 0,$ and let $-1\leq k\leq n$. We denote the essential image of $(\iota_{\leq k})_!\colon\D^{\cube{n}_{\leq k}}\to\D^{\cube{n}}$ by
\[
\D^{\cotr{n}{k}}=\mathrm{essim}((\iota_{\leq k})_!\colon\D^{\cube{n}_{\leq k}}\to\D^{\cube{n}}).
\]
The left Kan extension yields an equivalence $(\iota_{\leq k})_!^\simeq\colon\D^{\cube{n}_{\leq k}}\toiso\D^{\cotr{n}{k}}$, showing that there is a \emph{derivator} $\D^{\cotr{n}{k}}$ of $k$-cotruncated $n$-cubes (and not merely a prederivator). Dually, we write
\[
\D^{\tr{n}{k}}=\mathrm{essim}((\iota_{\geq k})_\ast)\colon\D^{\cube{n}_{\geq k}}\to\D^{\cube{n}})
\]
for the derivator of $k$-truncated $n$-cubes.
\end{notn}

\begin{defn}\label{defn:cotr}
Let \D be a derivator, let $n\geq 0$, and $0\leq k\leq n$. The morphism of derivator
\[
\cotrmor{k}=(\iota_{\leq k})_!\iota_{\leq k}^\ast\colon\D^{\cube{n}}\to\D^{\cube{n}}
\]
is $k$-th \textbf{cotruncation morphism}.
\end{defn}

\begin{lem}\label{lem:cotrmor}
Let \D be a derivator, let $n\geq 0$, and $0\leq k\leq n$.
\begin{enumerate}
\item The morphisms $\cotrmor{k}$ are idempotent comonads on $\D^{\cube{n}}$ (the comultiplications $\cotrmor{k}\toiso\cotrmor{k}\circ\cotrmor{k}$ are invertible) with essential image $\D^{\cotr{n}{k}}$.
\item The counit $\varepsilon\colon\cotrmor{n}\toiso\id$ is an isomorphism. 
\item The morphism $\cotrmor{0}\colon\D^{\cube{n}}\to\D^{\cube{n}}$ sends $X$ to the constant $n$-cube on $X_\emptyset$.
\item The counit $\varepsilon\colon(\iota_{k,k-1})_!\iota_{k,k-1}^\ast\to\id$ induces a natural transformation 
\begin{equation}\label{eq:cotrmor}
\cotrmor{k-1}\to\cotrmor{k},\qquad 1\leq k\leq n.
\end{equation}
\end{enumerate}
\end{lem}
\begin{proof}
For every $k$ there is an adjunction
\[
((\iota_{\leq k})_!,\iota_{\leq k}^\ast)\colon\D^{\cube{n}_{\leq k}}\rightleftarrows\D^{\cube{n}}
\]
with a fully faithful left adjoint. Hence the unit $\eta$ is an isomorphism and so is the comultiplication of the resulting comonad. The second statement is obvious and the third one is immediate from the pointwise formula for left Kan extensions (axiom (Der4)). For the final statement, we note that $\iota_{\leq k-1}$ factors as
\[
\iota_{\leq k-1}=\iota_{\leq k}\circ\iota_{k,k-1}\colon\cube{n}_{\leq k-1}\to\cube{n}_{\leq k}\to\cube{n},
\]
and we obtain $\cotrmor{k-1}\cong (\iota_{\leq k})_!(\iota_{k,k-1})_!\iota_{k,k-1}^\ast\iota_{\leq k}^\ast\to 
(\iota_{\leq k})_!\iota_{\leq k}^\ast=\cotrmor{k}$.
\end{proof}

The main goal of this section is to establish \autoref{thm:cotr-filtration}, thereby generalizing \autoref{thm:cocart} to $k$-cotruncated $n$-cubes. As a preparation we recall the following generalization of Franke's detection result for (co)cartesian squares (\cite[Prop.~1.4.5]{franke:adams}). Despite its technical character, it allows us in many situations to detect colimiting cocones in larger diagrams.

\begin{lem}\label{lem:detection-general}
Let \D be a derivator, let $C\in\cCat$, and let $u\colon A\to B$ and $v\colon C^\rhd\to B$ be functors. Suppose that there is a full subcategory $B'\subseteq B$ such that
\begin{enumerate}
\item $u(A)\subseteq B'$,
\item $v(C)\subseteq B'$ and $v(\infty)\notin B'$, and
\item the functor $C\to (B'/v(\infty))$ induced by $v$ is a right adjoint.
\end{enumerate}
Then for any $X\in\D(A)$ the cocone $v^\ast u_!(X)\in\D(C^\rhd)$ is colimiting.
\end{lem}
\begin{proof}
This is \cite[Lemma~4.5]{gps:mayer}.
\end{proof}

Often it is useful to analyze Kan extensions along fully faithful functors by a factorization into intermediate steps, each of them adding one object at a time. For those intermediate steps the following refinement of the detection result is convenient (as will be illustrated in the proof of \autoref{thm:cotr-filtration}).

\begin{prop}\label{prop:add-colimiting-cocone}
Let \D be a derivator, let $C\in\cCat,$ and let $u\colon A\to B$ be a fully faithful functor between small categories such that the following two conditions are satisfied.
\begin{enumerate}
\item The complement $B-u(A)$ consists of precisely one object $b_0$. 
\item There is a small category $C$ and a homotopy exact square
\begin{equation}\label{eq:add-colimiting-cocone}
\vcenter{
\xymatrix{
C\ar[d]_-{i_C}\ar[r]^-j\drtwocell\omit{\id}&A\ar[d]^-u\\
C^\rhd\ar[r]_-k&B
}
}
\end{equation} 
such that $k$ satisfies $k(\infty)=b_0$.
\end{enumerate}
The morphism $u_!\colon\D^A\to\D^B$ is fully faithful and induces an equivalence onto the full subderivator $\D^{B,\exx}\subseteq\D^B$ spanned by all $X\in\D^B$ such that $k^\ast(X)\in\D^{C^\rhd}$ is a colimiting cocone. The equivalence $\D^A\toiso\D^{B,\exx}$ is pseudo-natural with respect to morphisms of derivators which preserve colimits of shape $C$.
\end{prop}
\begin{proof}
Since $u\colon A\to B$ is fully faithful, the same is true for $u_!\colon\D^A\to\D^B$, and $X\in\D^B$ lie in the essential image of $u_!$ if and only if $\varepsilon\colon u_!u^\ast(X)\to X$ is an isomorphism. By \cite[Lem.~1.21]{groth:ptstab} this is the case if and only if the component $\varepsilon_{b_0}$ is invertible. To re-express this differently, let us consider the pasting on the left in  
\[
\xymatrix{
C\ar[r]^-j\ar[d]_-{i_C}\drtwocell\omit{\id}&A\ar[r]^-u\ar[d]_-u\drtwocell\omit{\id}&B\ar[d]^-\id \ar@{}[rrd]|{=}&& 
C\ar[r]^-{i_C}\ar[d]_-{i_C}\drtwocell\omit{\id}&C^\rhd\ar[r]^-k\ar[d]_-\id\drtwocell\omit{\id}&B\ar[d]^-\id\\
C^\rhd\ar[r]_-k&B\ar[r]_-\id&B && 
C^\rhd\ar[r]_-\id&C^\rhd\ar[r]_-k&B.
}
\]
The functoriality of mates with respect to pasting and the homotopy exactness of the square to the very left imply that $X$ lies in the essential image of $u_!$ if and only if the canonical mate associated to the pasting on the left is an isomorphism on $X$. Since the above two pastings agree, the functoriality of mates implies that this is the case if and only  if $k^\ast X$ lies in the essential image of $(i_C)_!\colon\D^C\to\D^{C^\rhd}$ which is to say that $k^\ast X$ is a colimiting cocone. It follows from \cite[Lemma~3.7 and Prop.~3.9]{groth:revisit} that a morphism of derivators preserves left Kan extensions along $u$ if and only if it preserves colimits of shape $C$, thereby establishing the intended pseudo-naturality.
\end{proof}

\begin{rmk}\label{rmk:add-colimiting-cocone}
In the situation of \autoref{prop:add-colimiting-cocone}, the square \eqref{eq:add-colimiting-cocone} is homotopy exact as soon as the induced functor $r\colon C\to(u/b_0)$ is homotopy final (such as a right adjoint). In fact, this follows from the functoriality of mates with respect to pastings and \cite[Lemma~2.12]{groth:revisit} applied to the following situation
\[
\xymatrix{
C\ar[r]^-\id\ar[d]_-\pi\drtwocell\omit&C\ar[d]_-{i_C}\ar[r]^-j\drtwocell\omit{\id}&A\ar[d]^-u\ar@{}[rrd]|{=}&&
C\ar[r]^-r\ar[d]_-\pi\drtwocell\omit{\id}&(u/b_0)\ar[r]^-p\ar[d]_-\pi\drtwocell\omit&A\ar[d]^-u\\
\bbone\ar[r]_-\infty&C^\rhd\ar[r]_-k&B&&
\bbone\ar[r]_-\id&\bbone\ar[r]_-{b_0}&B.
}
\]
\end{rmk}

With this preparation we can now attack the proof of \autoref{thm:cotr-filtration}. The proof is an adaptation of the proof of \cite[Thm.~8.4]{gst:basic}, and the following is the key step.

\begin{notn}
For $n\geq 0,$ $-1\leq k\leq n,$ and $x\in\cube{n}_{\leq k}$ we denote by $d(x)$ the cardinality of the subset $x\subseteq\{1,\ldots,n\}.$ Moreover, for $x\subseteq y\subseteq\{1,\ldots,n\}$ let $[x,y]\subseteq\cube{n}_{\leq k}$
be the closed interval between $x,y$ in $\cube{n}_{\leq k}$, i.e., we set
\[
[x,y]=\{w\in\cube{n}_{\leq k}\mid x\leq w\leq y\}\subseteq\cube{n}_{\leq k},
\]
and we denote the above inclusion by $j_{[x,y]}\colon[x,y]\to\cube{n}_{\leq k}$. If for $x\leq y\in\cube{n}_{\leq k}$ we write $d(x,y)=d(y)-d(x)$ for the cardinality of the complement $y-x$, then there is a preferred isomorphism
\[
\cube{d(x,y)}\cong[x,y].
\]
In fact, invoking our notational convention \eqref{eq:cube-set}, the inclusion $y-x\subseteq\{1,\ldots,n\}$ induces a monomorphism of posets
\[
\cube{d(x,y)}=\cube{y-x}\to\cube{\{1,\ldots,n\}}=\cube{n},
\]
and its image factorization gives the preferred isomorphism. By abuse of notation we also denote this monomorphism of posets by $j_{[x,y]}\colon\cube{d(x,y)}\to\cube{n}$.
\end{notn}

\begin{con}\label{con:filtration}
Let $n\geq 2$ and let $1\leq k< l\leq n$ be natural numbers. We consider $x\leq y\in\cube{n}_{\leq l}$ with $k<d(x,y)\leq l$ (thereby avoiding that $x$ and $y$ both lie in $\cube{n}_{\leq k}$). Associated to this, there are the following full subposets of~$\cube{n}_{\leq l}$.
\begin{enumerate}
\item The poset $A_1\subseteq\cube{n}_{\leq l}$ is obtained from $\cube{n}_{\leq k}$ by adding the objects of the $d(x)$-cube $[\emptyset,x]$.
\item The poset $A_2$ contains $A_1$ and also the punctured $d(x,y)$-cube $[x,y]-\{y\}$.
\item Finally, $A_3$ is obtained from $A_2$ by adding the object $y$.
\end{enumerate}
These full subcategories come with inclusion functors which allow us to factor the inclusion $\iota_{l,k}\colon\cube{n}_{\leq k}\to\cube{n}_{\leq l}$ as 
\begin{equation}\label{eq:factor-inc}
\iota_{l,k}\colon\cube{n}_{\leq k}\stackrel{j_1}{\to}A_1\stackrel{j_2}{\to}A_2\stackrel{j_3}{\to}A_3\stackrel{j_4}{\to}\cube{n}_{\leq l}.
\end{equation}
(We note that $j_1$ and $j_2$ are potentially identities but that this is impossible for $j_3$.) In particular, for every derivator \D there is a left Kan extension morphism
\begin{equation}\label{eq:key}
(j_3)_!\colon\D^{A_2}\to\D^{A_3},
\end{equation}
and associated to $j_3$ there is the following commutative square
\begin{equation}\label{eq:key-he}
\vcenter{
\xymatrix{
[x,y]-\{y\}\ar[r]\ar[d]&A_2\ar[d]^-{j_3}\\
[x,y]\ar[r]_-{j_{[x,y]}}&A_3.
}
}
\end{equation}
\end{con}

\begin{lem}\label{lem:filtration}
Let $n\geq 2$, $1\leq k< l\leq n,$ and $x\leq y\in\cube{n}_{\leq l}$ with $k<d(x,y)\leq l.$ For every derivator \D and $j_3\colon A_2\to A_3$ as in \eqref{eq:factor-inc}, the morphism $(j_3)_!\colon\D^{A_2}\to\D^{A_3}$ is fully faithful and the essential image consists precisely of those $X\in\D^{A_3}$ such that $j_{[x,y]}^\ast(X)$ is cocartesian.
\end{lem}
\begin{proof}
By \autoref{prop:add-colimiting-cocone} it suffices to show that \eqref{eq:key-he} is homotopy exact, and this can be reduced to showing that the functor $[x,y]-\{y\}\to(j_3/y)$ is a right adjoint (\autoref{rmk:add-colimiting-cocone}). We note that $(j_3/y)$ can be identified with $[\emptyset,y]\cap A_2$, and it remains to show that $i\colon [x,y]-\{y\}\to [\emptyset,y]\cap A_2$ is a right adjoint. We claim that 
\[
l\colon[\emptyset,y]\cap A_2\to[x,y]-\{y\}\colon z\mapsto x\cup z
\]
is well-defined and left adjoint to $i$. Let us recall from \autoref{con:filtration} that the category $[\emptyset,y]\cap A_2$ contains three types of objects, namely those in $[\emptyset,y]\cap\cube{n}_{\leq k}$, those in $[\emptyset,x],$ and those in $[x,y]-\{y\}$. To show that $l$ is well-defined amounts to showing that $x\cup z$ is different from $y$ in all three cases. The last two cases being trivial, let us consider $z\in[\emptyset,y]\cap\cube{n}_{\leq k}$ and show that $y-x$ is not contained in $x\cup z$. In fact, $(x\cup z)\cap(y-x)\subseteq z\cap(y-x)\subseteq z$ implies $|(x\cup z)\cap(y-x)|\leq k$ while $|y-x|\geq k+1$ showing that $x\cup z$ does not contain $y-x$ and thereby $x\cup z\neq y$. Since $\id\leq il$ and $li=\id$, we indeed obtain the desired adjunction.
\end{proof}

\begin{thm}\label{thm:cotr-filtration}
Let \D be a derivator, $n\geq 2$, and $1\leq k\leq l\leq n$. The following are equivalent for a restricted $n$-cube $X\in\D^{\cube{n}_{\leq l}}$.
\begin{enumerate}
\item The restricted $n$-cube $X$ is $k$-cotruncated.
\item All $m$-subcubes of $X$, $k<m\leq l$, are cocartesian.
\item All $(k+1)$-subcubes of $X$ are cocartesian.
\end{enumerate}
\end{thm}
\begin{proof}
We begin with the equivalence of the first two statements. Let $X\in\D^{\cube{n}_{\leq l}}$ be $k$-cotruncated, i.e., $X$ lies in the essential image of $(\iota_{l,k})_!\colon\D^{\cube{n}_{\leq k}}\to\D^{\cube{n}_{\leq l}}$, and let $[x,y]\subseteq\cube{n}_{\leq l}$ be a subcube of dimension $m=d(x,y) \ge k+1$. In order to show that $j_{[x,y]}^\ast X\in\D^{\cube{m}}$ is cocartesian we consider the factorization of $\iota_{l,k}$ as in \eqref{eq:factor-inc}. Consequently, the left Kan extension morphism $(\iota_{l,k})_!$ factors as
\[
(\iota_{l,k})_!\colon\D^{\cube{n}_{\leq k}}\stackrel{(j_1)_!}{\to}\D^{A_1}\stackrel{(j_2)_!}{\to}\D^{A_2}\stackrel{(j_3)_!}{\to}\D^{A_3}\stackrel{(j_4)_!}{\to}\D^{\cube{n}_{\leq l}}.
\]
The fully faithfulness of $j_4$ and our assumption on $X$ imply that $\varepsilon\colon(j_4)_!j_4^\ast X\to X$ is invertible. And it hence remains to show that for $Y=j_4^\ast X$ the $m$-cube $j_{[x,y]}^\ast Y$ is cocartesian. But since $Y$ lies in the essential image of $(j_3)_!$ this is immediate from \autoref{lem:filtration}.

For the converse direction, a finite induction implies that it is enough to consider the case that $l=k+1$. Hence, let us consider $X\in\D^{\cube{n}_{\leq k+1}}$ such that all $(k+1)$-subcubes of $X$ are cocartesian, and our aim is to show that $X$ lies in the essential image of $(\iota_{k+1,k})_!$. Since $\iota_{k+1,k}$ is fully faithful, this is the case if and only if the components of the adjunction counit $\varepsilon_y\colon\big((\iota_{k+1,k})_!\iota_{k+1,k}^\ast X\big)_y\to X_y$ are invertible for all $y\in\cube{n}_{=k+1}$. To reformulate this for a fixed such $y$, we consider the following pasting
\begin{equation}
\vcenter{
\xymatrix{
[\emptyset,y]-\{y\}\ar[r]\ar[d]_-\pi\drtwocell\omit{}&\cube{n}_{\leq k}\ar[r]^-{\iota_{k+1,k}}\ar[d]_-{\iota_{k+1,k}}\drtwocell\omit{\id}&\cube{n}_{\leq k+1}\ar[d]^-=\\
\bbone\ar[r]_y&\cube{n}_{\leq k+1}\ar[r]_-=&\cube{n}_{\leq k+1}
}
}
\end{equation}
in which the square on the left is a slice square. The functoriality of canonical mates with pasting and the homotopy exactness of slice squares imply that $\varepsilon_y$ is invertible if and only if the canonical mate of the above pasting is invertible at $X\in\D^{\cube{n}_{\leq k+1}}$. This pasting agrees with the pasting
\begin{equation}
\vcenter{\xymatrix{
[\emptyset,y]-\{y\}\ar[r]^-=\ar[d]_-\pi\drtwocell\omit{}&[\emptyset,y]-\{y\}\ar[r]\ar[d]\drtwocell\omit{\id}&[\emptyset,y]\ar[d]_-=\ar[r]^-{j_{[\emptyset,y]}}\drtwocell\omit{\id}&\cube{n}_{\leq k+1}\ar[d]^-=\\
\bbone\ar[r]_-y&[\emptyset,y]\ar[r]_-=&[\emptyset,y]\ar[r]_-{j_{[\emptyset,y]}}&\cube{n}_{\leq k+1}
}}
\end{equation}
in which the square to the left is again a slice square. We conclude this part by observing that the canonical mate of this pasting is invertible on $X\in\D^{\cube{n}_{\leq k+1}}$ if and only if the $(k+1)$-cube $j_{[\emptyset,y]}^\ast X$ is cocartesian.

Finally, the equivalence of $(ii)$ and $(iii)$ is immediate from the first statement of \autoref{thm:cocart}.
\end{proof}

\begin{cor}\label{cor:cotr-filtration}
Let \D be a derivator, $n\geq 2$, and $1\leq k\leq n$. The following are equivalent for an $n$-cube $X\in\D^{\cube{n}}$.
\begin{enumerate}
\item The $n$-cube $X$ is $k$-cotruncated.
\item All $m$-subcubes of $X,$ $k<m\leq n,$ are cocartesian.
\item All $(k+1)$-subcubes of $X$ are cocartesian.
\end{enumerate}
\end{cor}
\begin{proof}
This is a special case of \autoref{thm:cotr-filtration}.
\end{proof}

\begin{conv}\label{conv:filtration}
Let \D be a derivator and $n\geq 2$. By \autoref{defn:cocart}, an $n$-cube $X\in\D^{\cube{n}}$ is referred to as being cocartesian if it lies in the essential image of the left Kan extension morphism
\[
(\iota_{\leq n-1})_!\colon\D^{\cube{n}_{\leq n-1}}\to\D^{\cube{n}}.
\]
We extend this terminology to the degenerate cases of $n=1$ and $n=0$. Thus, we say that a morphism $X\in\D^{[1]}$ (or a $1$-cube) is cocartesian if it is an isomorphism and that an object $X\in\D$ (a $0$-cube) is cocartesian if it is initial.
\end{conv}

The point of this convention is that we now obtain the following more complete description of the filtration of derivators of coherent $n$-cubes.

\begin{rmk}\label{rmk:filtration}
Let \D be a derivator and $n\geq 0$. The derivator $\D^{\cube{n}}$ is filtered by the essential images of the left Kan extension morphisms
\[
(\iota_{\leq k})_!\colon\D^{\cube{n}_{\leq k}}\to\D^{\cube{n}},\qquad -1\leq k\leq n.
\]
In the case of $k=n$ we obtain all $n$-cubes, for $k=n-1$ the cocartesian $n$-cubes, for $k=1$ the strongly cocartesian ones, for $k=0$ the constant $n$-cubes (all structure maps are invertible by \autoref{lem:cotrmor}), and for $k=-1$ the initial $n$-cubes. By means of \autoref{conv:filtration}, this can be reformulated by saying that for every fixed $-1\leq k\leq n$ the respective essential image consists precisely of those $n$-cubes such that all $(k+1)$-subcubes are cocartesian.
\end{rmk}

In \S\ref{sec:total} we discuss a relation between these derivators and the vanishing of iterated cone constructions in the case of pointed and stable derivators. We conclude this section by a short discussion of closure properties.

\begin{cor}\label{cor:swap-filt}
Let \D be a derivator, $n\geq 2$, $1\leq i,j\leq n,$ and $1\leq k<l\leq n$. A restricted coherent $n$-cube $X\in\D^{\cube{n}_{\leq l}}$ is $k$-cotruncated if and only if this is the case for $\sigma_{i,j}^\ast X$. In particular, in the case of $l=n$, we obtain that $X\in\D^{\cotr{n}{k}}$ if and only if $\sigma_{i,j}^\ast X\in\D^{\cotr{n}{k}}.$
\end{cor}
\begin{proof}
This is immediate from \autoref{cor:swap-cocart} and \autoref{thm:cotr-filtration}.
\end{proof}

\begin{prop}\label{prop:cancel-filt}
Let \D be a derivator, $n\geq 2$, $1\leq k\leq n$, and let $X\in\D^{\cube{n-1}\times[2]}$ with $\iota_{01}^\ast X\in\D^{\cotr{n}{k}}$. Then $\iota_{12}^\ast X\in\D^{\cotr{n}{k}}$ if and only if $\iota_{02}^\ast X\in\D^{\cotr{n}{k}}$.
\end{prop}
\begin{proof}
As of the position of a $(k+1)$-cube in $\cube{n}=\cube{n-1}\times[1]$ there are the three cases that it lies in $\cube{n-1}\times\{0\},$ in $\cube{n-1}\times\{1\}$, or is the form $\cube{k}\times[1]$ for a suitable $k$-subcube of $\cube{n-1}$. Let us assume that also $\iota_{12}^\ast X\in\D^{\cotr{n}{k}}$ and let us in turn consider the three types of $(k+1)$-cubes in $\iota_{02}^\ast X$. In the first case it lies in $\iota_{01}^\ast X$, in the second case in $\iota_{12}^\ast X$, and in the remaining case it is the composition of a $(k+1)$-cube in each of them. By \autoref{cor:cotr-filtration} all these four $(k+1)$-cubes are cocartesian, hence so is the composition of the latter two (\autoref{thm:cocart}). Invoking \autoref{cor:cotr-filtration} again we deduce that $\iota_{02}^\ast X$ lies in $\D^{\cotr{n}{k}}$. The converse implication is similar.
\end{proof}

\section{Iterated cones and total cofibers}
\label{sec:total}

By \autoref{cor:cotr-filtration} an $n$-cube is cotruncated if and only if suitable subcubes are cocartesian. In this and the following section we study these obstructions more systematically in pointed and stable derivators. Total cofibers in pointed derivators are obstructions against cocartesianness and hence against cotruncatedness, and these obstructions are complete in the stable case. We show that total cofibers are canonically isomorphic to $n$-fold cone constructions (\autoref{thm:total-cof}), hence in the stable case total cofibers are $n$-fold suspensions of total fibers. This leads to a symmetric characterization of cotruncated $n$-cubes in stable derivators (see \S\ref{sec:sse-cotruncated}).

To begin with, we recall the following special case of the construction of canonical comparison maps between colimiting cocones and generic cocones \cite[\S2]{groth:revisit}.

\begin{con}
For every $n\geq 1$, the $n$-cube $\cube{n}$ is the cocone of the punctured $n$-cube $\cube{n}_{\leq n-1}$ with canonical inclusion
\[
\iota_{\leq n-1}\colon\cube{n}_{\leq n-1}\to\cube{n}=(\cube{n}_{\leq n-1})^\rhd.
\]
The cocone $(\cube{n})^\rhd=\cube{n}\ast\bbone$ is obtained from $\cube{n}$ by adjoining a new terminal object $\infty+1$, and this category corepresents morphisms of cocones on punctured cubes. Note that in this category there is a unique morphism $\infty\to\infty+1$ from the former terminal object to the new terminal object. The category $(\cube{n})^\rhd$ comes with fully faithful source and target inclusion functors
\begin{equation}\label{eq:s-t}
s\colon\cube{n}\to (\cube{n})^\rhd\qquad\text{and}\qquad t\colon\cube{n}\to (\cube{n})^\rhd.
\end{equation}
Both functors are the identity on the punctured cube and they are respectively determined by the additional relations $s(\infty)=\infty$ and $t(\infty)=\infty+1$. Given a prederivator \D and $X\in\D^{(\cube{n})^\rhd}$, we refer to $s^\ast X$ and $t^\ast X$ as the \textbf{source $n$-cube} and \textbf{target $n$-cube}, respectively.
\end{con}

Left Kan extension along $s$ simply adds isomorphisms $s_!(X)_\infty\toiso s_!(X)_{\infty+1}$, while left Kan extension along $t$ is described in the following proposition.

\begin{prop}\label{prop:cocart-cocone}
Let \D be a derivator, let $n\geq 1$, and let $s,t\colon\cube{n}\to (\cube{n})^\rhd$ be as in \eqref{eq:s-t}.
\begin{enumerate}
\item The morphism $t_!\colon\D^{\cube{n}}\to\D^{(\cube{n})^\rhd}$ is fully faithful and $Y\in\D^{(\cube{n})^\rhd}$ lies in the essential image of~$t_!$ if and only if the source cube $s^\ast Y$ is cocartesian.
\item An $n$-cube $X\in\D^{\cube{n}}$ is cocartesian if and only if the coherent morphism
\begin{equation}
t_!(X)_\infty\to t_!(X)_{\infty+1}
\end{equation}
is an isomorphism.
\end{enumerate}
\end{prop}
\begin{proof}
This is a special case of \cite[Prop.~3.11]{groth:revisit} and \cite[Prop.~3.14]{groth:revisit}.
\end{proof}

\begin{defn}\label{defn:tcof}
Let \D be a pointed derivator and let $n\geq 1$. The \textbf{total cofiber} of an $n$-cube $X\in\D^{\cube{n}}$ is the cone of the comparison map~$t_!(X)_\infty\to t_!(X)_{\infty+1},$
\[
\tcof(X)=C\big(t_!(X)_\infty\to t_!(X)_{\infty+1}\big)\in\D.
\]
\end{defn}

\begin{rmk}
Let \D be a pointed derivator and let $n\geq 1$.
\begin{enumerate}
\item For $n=1$ the total cofiber is canonically isomorphic to $C$.
\item There is the dual notion of the \textbf{total fiber} of an $n$-cube.
\item The total cofiber and the total fiber, respectively, define morphisms of derivators
\[
\tcof\colon\D^{\cube{n}}\to\D\qquad\text{and}\qquad\tfib\colon\D^{\cube{n}}\to\D.
\]
\end{enumerate}
\end{rmk}

The duality principle for derivators allows us to focus on total cofibers. 

\begin{cor}\label{cor:cocart-tcof-zero}
Let \D be a pointed derivator, let $n\geq 1$, and let $X\in\D^{\cube{n}}$.
\begin{enumerate}
\item If $X$ is cocartesian, then the total cofiber $\tcof(X)$ vanishes.
\item If \D is stable, then $X$ is cocartesian if and only if the total cofiber $\tcof(X)$ vanishes.
\end{enumerate}
\end{cor}
\begin{proof}
By \autoref{prop:cocart-cocone} an $n$-cube $X\in\D^{\cube{n}}$ is cocartesian if and only if the canonical morphism $t_!(X)_\infty\to t_!(X)_{\infty+1}\in\D^{[1]}$ is an isomorphism. Since cones of isomorphisms are trivial \cite[Prop.~3.12]{groth:ptstab}, the first claim follows directly from \autoref{defn:tcof}. The second claim is also immediate since isomorphisms in stable derivators can be characterized by the vanishing of the cone.
\end{proof}

By \autoref{cor:cotr-filtration} there are variants of this for cotruncated $n$-cubes. A systematic approach to a symmetric characterization for stable derivators is obtained through an identification of total cofibers as iterated cones (\autoref{thm:total-cof}), and we come back to this in \S\ref{sec:sse-cotruncated}. We now build towards this result, including some additional tools along the way.

\begin{con}\label{con:cof-comp}
Let \D be a pointed derivator, $n\in\lN$, and $1\leq i\leq n$. For the $n$-tuple $e_i=(\delta_{ij})_{j=1}^n=(0,\ldots,0,1,0,\ldots 0)\in\lN^n$ we denote by
\[
\cof^{e_i}\colon\D^{\cube{n}}\to\D^{\cube{n}}
\]
the morphism which forms the cofiber in the $i$-th coordinate. This is the cofiber morphism with parameters in $\cube{n-1}$. More generally, for $\underline m=(m_1,\ldots,m_n)\in\lN^n$ we make the definition
\begin{equation}\label{eq:cof-comp}
\cof^{\underline m}=(\cof^{e_n})^{m_n}\circ\ldots\circ(\cof^{e_1})^{m_1}\colon\D^{\cube{n}}\to\D^{\cube{n}}.
\end{equation}
\end{con}

\begin{prop}\label{prop:cof-comp}
For every pointed derivator \D and $n\in\lN$ the formula \eqref{eq:cof-comp} extends to a pseudo-action 
\[
\cof^\bullet\colon\lN^n\to\End(\D^{\cube{n}})\colon\underline{m}\mapsto \cof^{\underline m}.
\]
The derivator \D is stable if and only if this pseudo-action extends to a pseudo-action
\[
\cof^\bullet\colon\lZ^n\to\Aut(\D^{\cube{n}}).
\]
\end{prop}
\begin{proof}
The cofiber morphism $\cof\colon\D^{[1]}\to\D^{[1]}$ is a left adjoint morphism of pointed derivators, and hence preserves cofibers. Passing to parameterized versions, this shows that for $1\leq i\neq j\leq n$ there are canonical isomorphisms
\[
\cof^{e_i}\circ \cof^{e_j}\toiso\cof^{e_j}\circ\cof^{e_i},
\]
and these induce the desired pseudo-action of $\lN^n$. A pointed derivator \D is stable if and only if $\cof\colon\D^{[1]}\to\D^{[1]}$ is an equivalence, and this property allows us to obtain the extended pseudo-action of $\lZ^n$.
\end{proof}

\begin{notn}\label{notn:cof-comp}
Let \D be a pointed derivator and $n\geq 0$.
\begin{enumerate}
\item For every $k\in\lN$ we denote by $\underline{k}$ the $n$-tuple $(k,\ldots,k)\in\lN^n$ and the corresponding iterated cofiber construction by $\cof^{\underline k}$.
\item In order to emphasize notationally the dimension of the cubes under consideration, we also write
\[
\cof^{\underline{1}_n}=\cof^{\underline{1}}\colon\D^{\cube{n}}\to\D^{\cube{n}}.
\]
\end{enumerate}
\end{notn}

\begin{cor}\label{cor:cof-S}
For every pointed derivator \D and $n\geq 1$ there are canonical isomorphisms
\[
(\cof^{\underline{1}})^3\cong\cof^{\underline 3}\cong \Sigma^n\colon\D^{\cube{n}}\to\D^{\cube{n}}.
\]
\end{cor}
\begin{proof}
The first canonical isomorphism is immediate from the pseudo-action of $\lN^n$ (\autoref{prop:cof-comp}). The remaining one follows from \cite[Lem.~5.13]{gps:mayer} by passing to parametrized Kan extensions.
\end{proof}

Instead of encoding the cofiber morphisms in \autoref{con:cof-comp}, we can alternatively keep track of the cofiber objects only, thereby obtaining the following interesting variant.

\begin{notn}\label{notn:cubical-maps}
For every $n\geq 1$ and $\underline{m}=(m_1,\ldots,m_n)\in\{0,1\}^n$ we use the shorthand notation
\[
|\underline m|=m_1+\ldots+m_n.
\]
Associated to every such $\underline{m}$ there are induced cubical structure maps
\[
0=0_{\underline{m}},1=1_{\underline{m}}\colon\cube{n-|\underline{m}|}\to\cube{n}
\]
defined as follows. The functors $0_{\underline{m}},1_{\underline{m}}$ add $|\underline{m}|$ copies of $0,1$, respectively, at every $i$-th coordinates such that $m_i=1$, while they are the identity on the remaining coordinates. Moreover, we denote by $\underline{m}^\vee\in\{0,1\}^n$ the $n$-tuple whose coordinates are given by $1-m_i$. (Considered as objects in $\cube{n}$, we simply pass to the complement.) Finally, using this notation, the functor $\underline{m}\colon\bbone\to\cube{n}$ classifying $\underline{m}$ factors as
\begin{equation}\label{eq:classifying-fun-in-cubes}
\underline{m}=0_{\underline{m}^\vee}\circ\infty\colon\bbone\to\cube{|\underline{m}|}\to\cube{n}.
\end{equation}
\end{notn}

\begin{con}\label{con:C-comp}
Let \D be a pointed derivator and $n\in\lN$. For every $n$-tuple $\underline m=(m_1,\ldots,m_n)\in\{0,1\}^n$ we define the corresponding iterated cone construction as
\begin{equation}\label{eq:C-comp}
C^{\underline m}=(1_{\underline m})^\ast\circ \cof^{\underline m}\colon\D^{\cube{n}}\to\D^{\cube{n-|\underline m|}}.
\end{equation}
In the case of $\underline m=\underline{1}_n$ we simplify notation and write
\[
C^n=C^{\underline{1}_n}=\infty^\ast\circ\cof^{\underline{1}_n}\colon\D^{\cube{n}}\to\D.
\]
In particular, associated to the $n$-tuple $e_i=(\delta_{ij})_{j=1}^n=(0,\ldots,0,1,0,\ldots 0)\in\{0,1\}^n$ for every $1\leq i\leq n$ we obtain the morphism
\[
C^{e_i}\colon\D^{\cube{n}}\to\D^{\cube{n-1}}
\]
which applies the cone $C$ in the $i$-th direction, while dealing the remaining copies of $[1]$ as parameters.
\end{con}

Also in this iterated construction the order does not matter up to canonical isomorphisms.

\begin{cor}\label{cor:C-comp}
For every pointed derivator \D, $n\in\lN$, and $1\leq i<j\leq n$ there is a canonical isomorphism
\[
C^{e_i}\circ C^{e_j}\cong C^{e_{j-1}}\circ C^{e_i}\colon\D^{\cube{n}}\to\D^{\cube{n-2}}.
\]
\end{cor}
\begin{proof}
This is immediate from \autoref{prop:cof-comp}.
\end{proof}

The classical 3-by-3-lemma in triangulated categories is closely related to this isomorphism as we discuss in the following remark.

\begin{rmk}
Let \D be a pointed derivator and let $\boxbar\in\cCat$ be the poset $[1]\times[2]$. Two iterations of the formation of cofiber squares yield a fully faithful morphism of derivators $\D^{[1]}\to\D^{\boxbar}$ which sends a morphism in \D to its cofiber sequence. Applying this construction in both coordinates independently, we obtain a morphism of derivators
\begin{equation}\label{eq:cof-seq-hyper}
\D^{\cube{2}}\to\D^{\boxbar\times\boxbar}.
\end{equation}
The coherent diagrams in the image of this morphism give rise to underlying incoherent diagrams of the form of the diagrams in the 3-by-3-lemma in triangulated categories (or its unstable analogues). To justify this claim, let us consider $X\in\D^{\cube{2}}$ which looks like
\[
\xymatrix{
x\ar[r]^-f\ar[d]_-g&y\ar[d]^-{g'}\\
z\ar[r]_-{f'}&w.
}
\]
We simplify matters a bit and begin by studying the morphism of pointed derivators
\begin{equation}\label{eq:cof-sq-hyper}
\D^{\cube{2}}\to\D^{\cube{4}}
\end{equation}
which forms \emph{cofiber hypercubes}, i.e., which forms cofiber squares in both directions independently. (This morphism is obtained from \eqref{eq:cof-seq-hyper} by restriction to the first of the four hypercubes.) The image of $X$ under \eqref{eq:cof-sq-hyper} is displayed in \autoref{fig:cof-hypercube}, in which the seven different zero objects are distinguished notationally by additional indices. The large squares in the back and the front are respectively the cofiber squares of $f$ and $f'$, while the upper two diagonal squares are the ones of $g$ and $g'$. The small square in the front is the cofiber square for $Cg\to Cg'$ and the lower right diagonal square is the one of $Cf\to Cf'$. In these two squares, following the notation established in \autoref{con:C-comp}, we denote by $C^2X$ the corresponding iterated cone objects
\[
C(Cf\to Cf')\cong C(Cg\to Cg'),
\]
which are isomorphic by \autoref{cor:C-comp}. Passing to a suitable underlying diagram of shape $[2]\times[2]$, we obtain the corresponding part of the diagram
\[
\xymatrix{
x\ar[r]\ar[d]&y\ar[r]\ar[d]&Cf\ar[d]\ar[r]&\Sigma x\ar[d]\\
z\ar[r]\ar[d]&w\ar[r]\ar[d]&Cf'\ar[d]\ar[r]&\Sigma z\ar[d]\\
Cg\ar[r]\ar[d]&Cg'\ar[r]\ar[d]&C^2X\ar[d]\ar[r]\ar@{}[dr]|{(-1)}&\Sigma Cg\ar[d]\\
\Sigma x\ar[r]&\Sigma y\ar[r]&\Sigma Cf\ar[r]&\Sigma^2x,
}
\]
as we are used to from the 3-by-3-lemma. If we instead apply \eqref{eq:cof-seq-hyper} to $X$, then similar arguments yield the corresponding incoherent diagram of shape $[3]\times[3]$. This diagram consists of eight incoherent cofiber sequences and it commutes with the exception of the lower right square which anticommutes \cite[\S4]{gps:additivity}. (An additional incarnation of this sign shows up in tensor triangulated categories and their refinements (see \cite{margolis:spectra,hps:axiomatic,may:additivity} and the derivatorish perspective on these axioms in \cite{gps:additivity}).)

\begin{figure}
    \centering
    \begin{tikzpicture}[->,xscale=1.6]
      \node (x) at (0.4,7.2) {$x$};
      \node (y) at (6.4,7.2) {$y$};
      \node (z) at (1,6) {$z$};
      \node (w) at (7,6) {$w$};
      \node (0g) at (2.2,5) {$0_g$};
      \node (0g') at (4.2,5) {$0_{g'}$};
      \node (Cg) at (2.8,3.8) {$Cg$};
      \node (Cg') at (4.8,3.8) {$Cg'$};
      \node (00) at (2.2,3) {$0$};
      \node (01) at (4.2,3) {$0_1$};
      \node (02) at (2.8,1.8) {$0_2$};
      \node (CCX) at (4.8,1.8) {$C^2X$};
      \node (0f) at (0.4,1.2) {$0_f$};
      \node (Cf) at (6.4,1.2) {$Cf$};
      \node (0f') at (1,0) {$0_{f'}$};
      \node (Cf') at (7,0) {$Cf'$};
      \draw (x) -- (y); \draw (y) -- (w);
      \draw (x) -- (z);
      \draw (0g) -- (0g'); \draw (Cg) -- (Cg');
      \draw (0g) -- (Cg); \draw (0g') -- (Cg');
      \draw (00) -- (01); \draw (02) -- (CCX);
      \draw (00) -- (02); \draw (01) -- (CCX);
      \draw (0f) -- (Cf); \draw (0f') -- (Cf');
      \draw (0f) -- (0f'); \draw (Cf) -- (Cf');
      \draw (x) -- (0g); \draw (y) -- (0g'); 
      \draw (0f) -- (00); \draw (Cf) -- (01); 
      \draw (x) -- (0f); \draw (y) -- (Cf);
      \draw (w) -- (Cf');
      \draw (0g) -- (00); \draw (0g') -- (01);
      \draw[white,line width=5pt,-] (z) -- (w);     \draw (z) -- (w);  
      \draw[white,line width=5pt,-] (z) -- (Cg);    \draw (z) -- (Cg); 
      \draw[white,line width=5pt,-] (w) -- (Cg');    \draw (w) -- (Cg'); 
      \draw[white,line width=5pt,-] (Cf') -- (CCX);   \draw (Cf') -- (CCX);
      \draw[white,line width=5pt,-] (0f') -- (02);   \draw (0f') -- (02);
      \draw[white,line width=5pt,-] (z) -- (0f');    \draw (z) -- (0f'); 
      \draw[white,line width=5pt,-] (Cg) -- (02);   \draw (Cg) -- (02);
      \draw[white,line width=5pt,-] (Cg') -- (CCX);   \draw (Cg') -- (CCX);
    \end{tikzpicture}
    \caption{The cofiber hypercube of $X$}
    \label{fig:cof-hypercube}
  \end{figure}
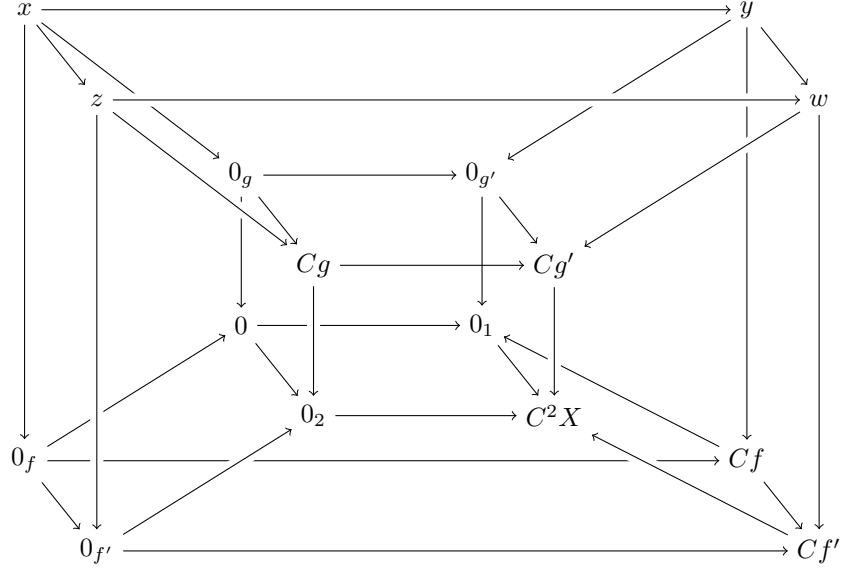
\end{rmk}

We collect a few basic facts on these iterated cone constructions. The formation of cofibers and restrictions in unrelated variables commute with each other. The following lemma describes what happens if both operations are applied to the same variables.

\begin{lem}\label{lem:iter-cones}
Let \D be a pointed derivator and $n\geq 1$.
\begin{enumerate}
\item There are canonical isomorphisms
\[ 
0^\ast\circ\cof\cong 1^\ast\colon\D^{[1]}\to\D\qquad\text{and}\qquad 1^\ast\circ\cof\cong C\colon\D^{[1]}\to\D.
\] 
\item For every $\underline{m}\in\{0,1\}^n$ the following square commutes up to a natural isomorphism
\[
\xymatrix{
\D^{\cube{n}}\ar[r]^-{\cof^{\underline{1}}}\ar[d]_-{(1_{\underline{m}^\vee})^\ast}\ar@{}[dr]|{\cong}&\D^{\cube{n}}\ar[d]^-{(0_{\underline{m}^\vee})^\ast}\\
\D^{\cube{|\underline{m}|}}\ar[r]_-{\cof^{\underline{1}}}&\D^{\cube{|\underline{m}|}}.
}
\]
\item For every $\underline{m}\in\{0,1\}^n$ the following square commutes up to a natural isomorphism
\[
\xymatrix{
\D^{\cube{n}}\ar[r]^-{\cof^{\underline{1}}}\ar[d]_-{(1_{\underline{m}^\vee})^\ast}\ar@{}[rd]|{\cong}&\D^{\cube{n}}\ar[d]^-{\underline{m}^\ast}\\
\D^{\cube{|\underline{m}|}}\ar[r]_-{C^{|\underline{m}|}}&\D.
}
\]
\end{enumerate}
\end{lem}
\begin{proof}
The first statement follows from construction of $C$ and $\cof$, and this implies the second one by a finite induction in combination with \autoref{prop:cof-comp}. Finally, for the third statement we recall from \eqref{eq:classifying-fun-in-cubes} that the classifying functor of $\underline{m}$ factors as $\underline{m}=0_{\underline{m}^\vee}\circ\infty,$ and this establishes the first equality in
\begin{align*}
\underline{m}^*\circ\mathsf{cof}^{\underline{1}} &= \infty^*\circ (0_{\underline{m}^{\vee}})^*\circ\mathsf{cof}^{\underline{1}}\\
&\cong \infty^* \circ\mathsf{cof}^{\underline{1}}\circ (1_{\underline{m}^{\vee}})^*\\
&= C^{|\underline{m}|}\circ (1_{\underline{m}^{\vee}})^*.
\end{align*}
The isomorphism follows from the second statement and the second equality holds by construction (\autoref{con:C-comp}).
\end{proof}

As an application of this result, we deduce the following stability of iterated cones under ``small perturbations''.

\begin{cor}\label{cor:cof-stability}
Let \D be a pointed derivator, let $n\geq 1$, and let $0\leq k\leq n$. If $f\colon X\to Y$ is a morphism in $\D^{\cube{n}}$ such that $\iota_{\geq k}^\ast f$ is an isomorphism, then also the restriction $\iota_{\leq n-k}^\ast(\mathsf{cof}^{\underline{1}}f)$ is an isomorphism.
\end{cor}
\begin{proof}
Let $\underline{m} \in \lbrace 0,1 \rbrace^n$ with classifying functor $\underline{m}\colon\bbone\rightarrow\cube{n}.$ By axiom (Der2) it suffices to show that $\underline{m}^*\mathsf{cof}^{\underline{1}}(f)$ is an isomorphism whenever $\vert\underline{m}\vert\leq n-k$.  \autoref{lem:iter-cones} implies that this morphism is isomorphic to $(C^{|\underline{m}|}\circ(1^{\underline{m}^{\vee}})^*)f$, and we show that $(1^{\underline{m}^{\vee}})^*f$ is invertible. In fact, since $|\underline{m}^\vee|=n-|\underline{m}|\geq k$, the inclusion $1^{\underline{m}^{\vee}}$ factors through $\iota_{\geq k}$, and we conclude by assumption on $f$. 
\end{proof}

\begin{rmk}\label{rmk:cof-stability}
\autoref{cor:cof-stability} applies, in particular, to the adjunction morphisms
\[
\eta\colon\id_{\D^{\cube{n}}} \rightarrow (\iota_{\geq k})_*(\iota_{\geq k})^*\qquad\text{and}\qquad \epsilon\colon (\iota_{\geq k})_!(\iota_{\geq k})^* \rightarrow id_{\D^{\cube{n}}}
\]
which are invertible when restricted to $\cube{n}_{\geq k}$. In fact, this follows from the fully faithfulness of the respective Kan extensions and the triangular identities.
\end{rmk}

As a preparation for the applications in \S\ref{sec:sse-cotruncated}, we now collect an alternative description of total cofibers.

\begin{prop}\label{prop:tcof-adjoint}
For every pointed derivator \D and $n\geq 1$ there are adjunctions
\[
(\tcof,\infty_!)\colon\D^{\cube{n}}\rightleftarrows\D\qquad\text{and}\qquad
(\emptyset_\ast,\tfib)\colon\D\rightleftarrows\D^{\cube{n}}.
\]
\end{prop}
\begin{proof}
Let $j\colon[1]\to (\cube{n})^\rhd$ be the functor classifying the morphism $\infty\to\infty+1$. By \autoref{defn:tcof} the total cofiber morphism is given by the composition
\begin{equation}\label{eq:tcof-defn}
\tcof\colon\D^{\cube{n}}\stackrel{t_!}{\to}\D^{(\cube{n})^\rhd}\stackrel{j^\ast}{\to}\D^{[1]}\stackrel{\cof}{\to}\D^{[1]}\stackrel{1^\ast}{\to}\D.
\end{equation}
Each of these morphisms has a right adjoint and these are respectively given by 
\[
\D\stackrel{1_\ast}{\to}\D^{[1]}\stackrel{\fib}{\to}\D^{[1]}\stackrel{j_\ast}{\to}\D^{(\cube{n})^\rhd}\stackrel{t^\ast}{\to}\D^{\cube{n}}.
\]
It remains to calculate the effect of this chain of right adjoints. For $x\in\D$ we have
\[
(\fib\circ 1_\ast)(x)\cong\fib(x\toiso x)\cong (0\to x).
\]
In order to calculate $j_\ast$ we invoke the pointwise formula (axiom (Der4)). For every object $a\in\cube{n}$ the slice category $(a/j)$ contains $(a,a\to\infty)$ as initial object, and by homotopy initiality we hence obtain
\[
j_\ast(0\to x)_a\cong 0,\quad a\neq\infty+1\in(\cube{n})^\rhd.
\]
As an upshot, the chain of right adjoints sends $x\in\D$ to a coherent $n$-cube supported at $\infty$ only where it takes $x$ as value. But since $\infty\colon\bbone\to\cube{n}$ is a cosieve, also the morphism $\infty_!$ is left extension by zero, yielding the adjunction $(\tcof,\infty_!)$.
\end{proof}

\begin{rmk}\label{rmk:review-exc-im}
Let \D be a pointed derivator and $n\geq 1$. The adjunction
\[
(\tcof,\infty_!)\colon\D^{\cube{n}}\rightleftarrows\D
\]
exhibits $\tcof$ as an \emph{exceptional inverse image morphism}
\begin{equation}\label{eq:tcof-exc}
\tcof\cong\infty^?\colon\D^{\cube{n}}\to\D.
\end{equation}
We recall from \cite[\S3.1]{groth:ptstab} that these are additional left adjoints to left Kan extensions along cosieves which exist for pointed derivators.
\end{rmk}

\begin{thm}\label{thm:total-cof}
For every pointed derivator \D and $n\geq 1$ there is a canonical isomorphism
\[
\tcof\cong C^n\colon\D^{\cube{n}}\to\D.
\]
\end{thm}
\begin{proof}
The cosieve $\infty\colon\bbone\to\cube{n}$ factors as a composition of $n$ cosieves
\[
\infty\colon\bbone\stackrel{1}{\to}[1]\stackrel{1}{\to}\square\stackrel{1}{\to}\ldots\stackrel{1}{\to}\cube{n},
\]
where we invoke the shorthand notation
\[
1=\id_{\cube{k}}\times 1\colon\cube{k}\to\cube{k+1},\quad 0\leq k\leq n-1. 
\]
For each of these cosieves the corresponding exceptional inverse image morphism is given by
\begin{equation}\label{eq:partial-cone-exc}
1^?\cong C^{e_{k+1}}\colon\D^{\cube{k+1}}\to\D^{\cube{k}},
\end{equation}
as is immediate from \cite[Prop.~3.22]{groth:ptstab} by dealing $\cube{k}$ in $\cube{k}\times[1]$ as parameters. By \eqref{eq:tcof-exc} and two applications of the uniqueness of left adjoints (first to left Kan extensions and then to exceptional inverse images) we obtain the first two canonical isomorphisms in 
\begin{align*}
\tcof&\cong \infty^?\\
&\cong 1^?\circ\ldots\circ1^?\\
&\cong C^{e_1}\circ\ldots\circ C^{e_n}\\
&= C^n.
\end{align*}
The remaining isomorphism is simply \eqref{eq:partial-cone-exc}.
\end{proof}

\begin{rmk}\label{rmk:SFCn}
For every pointed derivator \D there is a canonical comparison map
\begin{equation}\label{eq:SFC}
\Sigma\circ F\to C\colon\D^{[1]}\to\D,
\end{equation}
and \D is stable if and only if this comparison map is invertible. If $n\geq 1$ and $\underline{m}\in\{0,1\}^n$, then an application of \eqref{eq:SFC} in $|\underline m|$ coordinates independently yields a canonical comparison map
\[
\Sigma^{|\underline m|}\circ F^{\underline m}\to C^{\underline{m}}\colon\D^{\cube{n}}\to\D^{\cube{n-|\underline m|}},
\]
which is invertible in stable derivators. This applies, in particular, to the canonical comparison maps
\[
\Sigma^n\circ\tfib\to\tcof\colon\D^{\cube{n}}\to\D.
\]
\end{rmk}

\autoref{thm:total-cof} suggests that $\cof^{\underline{1}}$ is more elementary than $\tcof$ in that the latter can be obtained from the former by evaluation. To conclude this section we show that this is only seemingly the case by proving that $\cof^{\underline{1}}$ can be derived from a parametrized version of $\tcof$ (\autoref{prop:tcof-cof1}).

\begin{lem}\label{lem:restr-tcof}
Let \D be a pointed derivator, $n\geq 1$, and $\underline{m}\in\cube{n}$. There is a canonical natural isomorphism $\tcof\cong\tcof\circ (1_{\underline{m}^\vee})_!,$
\[
\xymatrix{
\D^{\cube{|\underline{m}|}}\ar[r]^-{(1_{\underline{m}^\vee})_!}\ar[dr]_-\tcof&\D^{\cube{n}}\ar[d]^-\tcof\\
&\D.
}
\]
\end{lem}
\begin{proof}
By \autoref{prop:tcof-adjoint} the composition $\tcof\circ (1_{\underline{m}^\vee})_!$ is left adjoint to the composition $(1_{\underline{m}^\vee})^\ast\circ \infty_!\cong \infty_!$. By uniqueness of adjoints the claim follows from a second application of \autoref{prop:tcof-adjoint}.
\end{proof}

As a consequence of \autoref{thm:total-cof} and \autoref{lem:restr-tcof}, the local description for the cofiber cube $\cof^{\underline{1}}$ provided by \autoref{lem:iter-cones} (iii) admits the reformulation
\[
\underline{m}^*\circ\cof^{\underline{1}}\cong\tcof\circ (1_{\underline{m}^\vee})_!,\qquad\underline{m}\in\cube{n}.
\]
To show that these natural isomorphisms are coherently compatible for $\underline{m}\in\cube{n}$ we consider the following construction.

\begin{con}
Let \D be a pointed derivator, let $n\geq 0$, and let $A$ be the full subcategory of $\cube{n}\times\cube{n}$ spanned by the objects
\[
\lbrace(\underline{m}_1,\underline{m}_2)\in\cube{n}\times\cube{n}\vert\;\underline{m}_2^{\vee}\subseteq\underline{m}_1\rbrace.
\]
Furthermore, let $i\colon A\rightarrow\cube{n}\times\cube{n}$ be the inclusion and $p=\pi\circ i$ be the composition with the projection to the \emph{second} component. We define
\[
E=i_!\circ p^*\colon\D^{\cube{n}}\rightarrow\D^{\cube{n}\times\cube{n}}
\]
and observe that $i_!$ is left extension by zero since $i$ is a cosieve. The local description
\begin{equation}\label{eq:counit-cube}
(\id\times\underline{m}_2)^\ast\circ E \cong (1_{\underline{m}_2^{\vee}})_!\circ(1_{\underline{m}_2^{\vee}})^*
\end{equation}
follows directly from the definition and a more careful investigation shows, that for $\underline{m}_2\in\cube{n}$ and $i_{\underline{m}_2}\colon [1]\rightarrow\cube{n}$ the map classifying $\underline{m}_2\rightarrow\infty$ the inverse image
\[
(\id\times i_{\underline{m}_2})^*\circ E\colon\D^{\cube{n}}\rightarrow\D^{\cube{n}\times [1]}
\]
gives a coherent model for the counit of the adjunction $((1_{\underline{m}_2^{\vee}})_!,(1_{\underline{m}_2^{\vee}})^*)$. This also explains the notation $E$, which is meant to be reminiscent of a capital epsilon.
\end{con}

In the following proposition, $\tcof^{\cube{n}}\colon\D^{\cube{n}\times\cube{n}}\rightarrow\D^{\cube{n}}$ denotes the total cofiber morphism for the shifted derivator $\D^{\cube{n}}$ with parameters in the second component in $\cube{n}\times\cube{n}$.

\begin{prop}\label{prop:tcof-cof1}
Let \D be a pointed derivator and $n\geq 1$. There is a natural isomorphism
\[
\tcof^{\cube{n}}\circ E\cong\cof^{\underline{1}}\colon\D^{\cube{n}}\to\D^{\cube{n}}.
\]
\end{prop}

\begin{proof}
The morphisms in the statement can be obtained from the $2n$-cube
\[
E'=(\cof^{\underline{1}}\times\id)\circ E\colon\D^{\cube{n}}\rightarrow\D^{\cube{n}\times\cube{n}}
\]
via postcomposition with restrictions
\begin{equation}\label{eq:tcof-cof1}
(\infty\times\id)^\ast\circ E'\cong\tcof^{\cube{n}}\circ E \qquad\text{ and }\qquad(\id\times\infty)^\ast\circ E'\cong\cof^{\underline{1}}.
\end{equation}
The first identification follows from \autoref{lem:iter-cones} and \autoref{thm:total-cof} whereas the second one is immediate. To relate these restrictions we consider the natural transformations $\alpha_1\colon\Delta\rightarrow(\infty\times\id)$ and $\alpha_2\colon\Delta\rightarrow(\id\times\infty)$ of functors $\cube{n}\rightarrow\cube{n}\times\cube{n}$. We show that for all $\underline{m}\in\cube{n}$ the natural transformations $\underline{m}^*\circ\alpha_1^*\circ E'$ and $\underline{m}^*\circ\alpha_2^*\circ E'$ are isomorphisms. For this, it is clearly sufficient to prove that the $\underline{m}^{\vee}$-cubes 
\[
Q_1=((1_{\underline{m}})^*\times\underline{m}^*)\circ E' \qquad\text{ and }\qquad Q_2=(\underline{m}^*\times(1_{\underline{m}})^*)\circ E'
\]
are constant, since the natural transformations under consideration describe the passage from the initial to the final vertex in $Q_1$ respectively $Q_2$.
\begin{enumerate}
\item In the case of $Q_1$ we make the identifications
\begin{align*}
&Q_1\\
=\;\; &((1_{\underline{m}})^*\times\underline{m}^*)\circ E'\\
=\;\; &((1_{\underline{m}})^*\times\underline{m}^*)\circ(\cof^{\underline{1}}\times\id)\circ E\\
\cong\;\; & (1_{\underline{m}})^*\circ\cof^{\underline{1}}\circ(\id\times\underline{m}^*)\circ E \\
\cong\;\; & (1_{\underline{m}})^*\circ\cof^{\underline{1}}\circ (1_{\underline{m}^{\vee}})_!\circ (1_{\underline{m}^{\vee}})^*\\
\cong\;\; & C^{\underline{m}}\circ\cof^{\underline{m}^{\vee}}\circ (1_{\underline{m}^{\vee}})_!\circ (1_{\underline{m}^{\vee}})^*\\
\cong\;\; & C^{\underline{m}}\circ(\pi^{\underline{m}^{\vee}})^*\circ (1_{\underline{m}^{\vee}})^*
\end{align*}
where $\pi^{\underline{m}^{\vee}}\colon\cube{n}\rightarrow\cube{\underline{m}}$ is the projection. In the second isomorphism  we use \eqref{eq:counit-cube}, the third isomorphism is \autoref{lem:iter-cones}(i) and the last step follows from the obvious relation 
\[
\cof\circ 1_!\cong\pi^*\colon\D\rightarrow\D^{[1]}.
\]
These identifications show that $Q_1$ is constant.
\item Similarly, in the case of $Q_2$ we calculate
\begin{align*}
&Q_2\\
=\;\; & (\underline{m}^* \times (1_{\underline{m}})^*) \circ E'\\
=\;\; & (\underline{m}^* \times (1_{\underline{m}})^*) \circ (\cof^{\underline{1}}\times\id)\circ E\\
\cong\;\; & (\id\times(1_{\underline{m}})^*)\circ((\underline{m}^*\circ\cof^{\underline{1}})\times\id)\circ E \\
\cong\;\; & (\id\times(1_{\underline{m}})^*)\circ((C^{\underline{m}}\circ(1_{\underline{m}^{\vee}})^*)\times\id)\circ E \\
\cong\;\; & (C^{\underline{m}}\times\id)\circ((1_{\underline{m}^{\vee}})^*\times(1_{\underline{m}})^*)\circ E\\
\cong\;\; & (C^{\underline{m}}\times\id)\circ(\pi_1^{\underline{m}^{\vee}})^*\circ(1_{\underline{m}^{\vee}})^*.
\end{align*}
Here $\pi_1^{\underline{m}^{\vee}}\colon\cube{\underline{m}}\times\cube{{\underline{m}^{\vee}}}\rightarrow\cube{\underline{m}}$ is the projection onto the first coordinate. Moreover, the second isomorphism is \autoref{lem:iter-cones}, the third one is the compatibility of Kan extensions and restrictions in unrelated variables, and the last one follows from the definition of $A$. Thus, these identifications show that also $Q_2$ is constant.
\end{enumerate}
Since isomorphisms in derivators are detected pointwise (axiom (Der2)), we can invoke the identifications \eqref{eq:tcof-cof1} in order to conclude that
\[
(\alpha_2\circ E')\circ(\alpha_1\circ E')^{-1}\colon\tcof^{\cube{n}}\circ E\toiso\cof^{\underline{1}}
\]
yields the desired natural isomorphism.
\end{proof}

\section{Strong stable equivalences for cotruncated $n$-cubes}
\label{sec:sse-cotruncated}

In this short section we apply our previous results on iterated cofibers to derivators of cotruncated $n$-cubes. For stable derivators, we characterize cotruncated $n$-cubes by the vanishing of iterated cone constructions, and this yields the strong stable equivalences of $\cube{n}_{\leq k}$ and $\cube{n}_{\geq n-k}$ (\autoref{thm:sse-chunks-special-case}). In \S\ref{sec:sse-chunks} we generalize this to arbitrary chunks of $n$-cubes.

\begin{cor}\label{cor:coc-iter-cones}
Let \D be a pointed derivator, let $n\geq 1$, and let $X\in\D^{\cube{n}}$.
\begin{enumerate}
\item If $X$ is cocartesian, then $C^n(X)$ vanishes.
\item If \D is stable, then $X$ is cocartesian if and only if $C^n(X)$ vanishes.
\end{enumerate}
\end{cor}
\begin{proof}
By \autoref{thm:total-cof} this is merely a reformulation of \autoref{cor:cocart-tcof-zero}.
\end{proof}

This allows us to collect in a uniform way the following obstructions against cotruncatedness.

\begin{prop}\label{prop:cotr-iter-cones}
Let \D be a pointed derivator, $n\geq 1$, $1\leq k\leq n$, and $X\in\D^{\cube{n}}$.
\begin{enumerate}
\item If $X\in\D^{\cotr{n}{k}}$, then $C^{\underline m}(X)$ vanishes for all $|\underline m|=k+1$.
\item If \D is stable, then $X\in\D^{\cotr{n}{k}}$ if and only if $C^{\underline m}(X)$ vanishes for all $|\underline m|=k+1$.
\end{enumerate}
\end{prop}
\begin{proof}
The $n$-cube $X$ is $k$-cotruncated if and only if every $(k+1)$-subcube of $X$ is cocartesian (\autoref{thm:cotr-filtration}). Such a $(k+1)$-subcube is obtained by choosing coordinates in the remaining $n-k-1$ directions. For a choice of such directions there are $2^{n-k-1}$ different subcubes of dimension $(k+1)$, which is to say that we simply consider $X$ as living in $(\D^{\cube{n-k-1}})^{\cube{k+1}}$. And by \cite[Cor.~2.6]{groth:ptstab} $X$ defines a cocartesian $(k+1)$-cube in $\D^{\cube{n-k-1}}$ if and only if the $(k+1)$-cubes $X_a, a\in \cube{n-k-1},$ are cocartesian. As an upshot, $X\in\D^{\cube{n}}$ is $k$-cotruncated if and only if all $(k+1)$-cubes of the form $X\in(\D^{\cube{n-k-1}})^{\cube{k+1}}$ are cocartesian. The result hence follows from \autoref{cor:coc-iter-cones}.
\end{proof}

\begin{cor}\label{cor:cotr=tr}
Let \D be a stable derivator, $n\geq 1$, and $1\leq k\leq n$. An $n$-cube $X\in\D^{\cube{n}}$ is $k$-cotruncated if and only if it is $(n-k)$-truncated. In particular, $X$ is strongly cocartesian if and only if $X$ is strongly cartesian, and $X$ is cocartesian if and only if $X$ is cartesian.
\end{cor}
\begin{proof}
An $n$-cube $X$ is $k$-cotruncated if and only if for every $\underline m\in\{0,1\}^n$ with $|\underline m|=k+1$ we have $C^{\underline m}(X)\cong 0$. Invoking \autoref{rmk:SFCn} this is the case if and only if $(\Sigma^{|\underline m|}\circ F^{\underline m})X\cong 0$ which is to say that $F^{\underline m}X\cong 0$ for all $\underline m$ with $|\underline m|=k+1$. By the dual of \autoref{prop:cotr-iter-cones} this is the case precisely when $X$ is $(n-k)$-truncated.
\end{proof}

The special case of strongly (co)cartesian $n$-cubes was already dealt with as \cite[Cor.~8.13]{gst:basic}. This result can be reformulated in terms of abstract representation theory (see \cite{gst:basic} and its sequels). Let us recall that two small categories $A,B$ are \emph{strongly stably equivalent} if for every stable derivator \D there is an equivalence of derivators $\D^A\simeq\D^B$ which is pseudo-natural with respect to exact morphisms of derivators (see \cite[Def.~5.1]{gst:basic} for more details). The following is essentially a reformulation of earlier results, but due to its importance to this cubical calculus we formulate the result as a theorem.

\begin{thm}\label{thm:sse-chunks-special-case}
For every $n\geq 1$ and $0\leq k\leq n$ the categories $\cube{n}_{\leq k}$ and $\cube{n}_{\geq n-k}$ are strongly stably equivalent.
\end{thm}
\begin{proof}
For every derivator \D there are equivalences of derivators
\[
\D^{\cube{n}_{\leq k}}\simeq\D^{\cotr{n}{k}}\qquad\text{and}\qquad\D^{\tr{n}{n-k}}\simeq\D^{\cube{n}_{\geq n-k}}
\]
given by left and right Kan extension, respectively. If \D is stable, then the derivators $\D^{\cotr{n}{k}}$ and $\D^{\tr{n}{n-k}}$ are equal (\autoref{cor:cotr=tr}), and the result follows from the compatibility of exact morphisms of stable derivators with homotopy finite Kan extensions.
\end{proof}

The first statement in \autoref{prop:cotr-iter-cones} admits the following generalization that cones decrease the degree of cotruncatedness by one.

\begin{prop}\label{prop:cone-lower-cotr}
Let \D be a pointed derivator, $n\geq 1$, and $1\leq k \leq n$. For every $1\leq j\leq n$ the partial cone $C^{e_j}\colon\D^{\cube{n}}\to\D^{\cube{n-1}}$ restricts to a morphism
\[
C^{e_j}\colon\D^{\cotr{n}{k}}\to\D^{\cotr{n-1}{k-1}}.
\]
\end{prop}
\begin{proof}
We denote by $1^{e_j}\colon\cube{n-1}\to\cube{n}$ the functor which includes a $1$ in the $j$-th spot. Associated to this functor, there is the commutative diagram
\[
\xymatrix{
\cube{n-1}_{\leq k-1}\ar[r]^-{1^{e_j}}\ar[d]_-{\iota}&\cube{n}_{\leq k}\ar[d]^-\iota\\
\cube{n-1}\ar[r]_-{1^{e_j}}&\cube{n}\ultwocell\omit{\id}
}
\]
in $\cCat$ (which we on purpose populate by the identity transformation). Note that this square is a pullback square in which the vertical functors are sieves and the horizontal functors are cosieves. Hence, \autoref{lem:opfibration} shows in two different ways that this square is homotopy exact and the resulting canonical mate
\[
\xymatrix{
\D^{\cube{n-1}_{\leq k-1}}\ar[d]_-{\iota_\ast}&\D^{\cube{n}_{\leq k}}\ar[l]_-{(1^{e_j})^\ast}\ar[d]^-{\iota_\ast}\dltwocell\omit{\cong}\\
\D^{\cube{n-1}}&\D^{\cube{n}}\ar[l]^-{(1^{e_j})^\ast}
}
\]
is invertible. Clearly, the vertical morphisms have second left adjoints. This is also the case for the horizontal morphisms (\autoref{rmk:review-exc-im}), hence by passing to second left adjoints everywhere we obtain a canonical isomorphism 
\begin{equation}\label{eq:cone-lower-cotr}
\vcenter{
\xymatrix{
\D^{\cube{n-1}_{\leq k-1}}\ar[d]_-{\iota_!}&\D^{\cube{n}_{\leq k}}\ar[l]_-{(1^{e_j})^?}\ar[d]^-{\iota_!}\dltwocell\omit{\cong}\\
\D^{\cube{n-1}}&\D^{\cube{n}}\ar[l]^-{(1^{e_j})^?},
}
}
\end{equation}
We invoke \eqref{eq:partial-cone-exc} to see that the bottom morphism in this diagram is isomorphic to $C^{e_j}\colon\D^{\cube{n}}\to\D^{\cube{n-1}}$, concluding the proof.
\end{proof}

\begin{rmk}
\begin{enumerate}
\item The natural isomorphism \eqref{eq:cone-lower-cotr} constructed in the proof actually offers a more precise formulation of \autoref{prop:cone-lower-cotr}.
\item The upper horizontal morphisms in \eqref{eq:cone-lower-cotr} will be examined further in the discussion of the level structure in \cite{bg:global}. 
\item By a similar argument, there are variants for $\iota_{l,k}\colon\cube{n}_{\leq k}\to\cube{n}_{\leq l}, 1\leq k\leq l\leq n$. In fact, in that case also
\[
\xymatrix{
\cube{n-1}_{\leq k-1}\ar[r]^-{1^{e_j}}\ar[d]_-{\iota_{l-1,k-1}}&\cube{n}_{\leq k}\ar[d]^-{\iota_{l,k}}\\
\cube{n-1}_{\leq l-1}\ar[r]_-{1^{e_j}}&\cube{n}_{\leq l}\ultwocell\omit{\id}
}
\]
is homotopy exact.
\end{enumerate}
\end{rmk}

The following corollary actually follows immediately from \autoref{prop:cotr-iter-cones}, but we postponed it until now in order to put it into context with \autoref{prop:cone-lower-cotr}.

\begin{cor}\label{cor:cotr-iter-cones-stable}
Let \D be a stable derivator, $n\geq 1$, $1\leq k\leq n$, and $X\in\D^{\cube{n}}$. The following are equivalent.
\begin{enumerate}
\item The $n$-cube $X$ is $k$-cotruncated.
\item The $(n-1)$-cube $C^{e_j}X$ is $(k-1)$-cotruncated for every $1\leq j \leq n$.
\item There exists a subset $N \subset \lbrace 1, \ldots, n \rbrace,\vert N \vert =n-k,$ such that for every $j \in N$ the $(n-1)$-cube $C^{e_j}X$ is $(k-1)$-cotruncated.
\end{enumerate}
\end{cor}
\begin{proof}
By \autoref{prop:cotr-iter-cones} it remains to show that (iii) implies (i), and to this end, again by \autoref{prop:cotr-iter-cones}, it is enough to verify that $C^{\underline m}X$ vanishes for all $\vert \underline m \vert=k+1$. For any such $\underline m$ there is an index $j\in N$ with $m_j=1$, and $C^{\underline m}X$ is by \autoref{cor:C-comp} isomorphic to a $k$-fold cone of $C^{e_j}X$. Invoking \autoref{prop:cotr-iter-cones} again we conclude that this iterated cone vanishes.
\end{proof}

As a variant of \autoref{prop:cotr-iter-cones} we obtain the following result.

\begin{cor}\label{cor:cotr-iter-cones-II}
Let \D be a pointed derivator, $n\geq 1$, $1\leq k\leq n$, and $X\in\D^{\cube{n}}$.
\begin{enumerate}
\item If $X\in\D^{\cotr{n}{k}}$, then $C^{\underline m}(X)\in\D^{\cube{n-k}}$ is constant for all $|\underline m|=k$.
\item If \D is stable, then $X\in\D^{\cotr{n}{k}}$ if and only if $C^{\underline m}(X)\in\D^{\cube{n-k}}$ is constant for all $|\underline m|=k$.
\end{enumerate}
\end{cor}
\begin{proof}
By pasting $k$ squares of the type \eqref{eq:cone-lower-cotr}, we obtain a canonical isomorphism
\[
\xymatrix{
\D^{\cube{n-k}_{\leq 0}}\ar[d]_-{\iota_!}&\D^{\cube{n}_{\leq k}}\ar[l]_-{(1^{\underline m})^?}\ar[d]^-{\iota_!}\dltwocell\omit{\cong}\\
\D^{\cube{n-k}}&\D^{\cube{n}}\ar[l]^-{C^{\underline m}},
}
\]
Since $\iota\colon\bbone=\cube{n-k}_{\leq 0}\to\cube{n-k}$ classifies the initial object, the corresponding left Kan extension morphism forms constant $(n-k)$-cubes.
\end{proof}

Of particular interest is the following special case of \autoref{cor:cotr-iter-cones-II} ($k=n-1$).

\begin{rmk}
Let \D be a pointed derivator, $n\geq 1$, and $X\in\D^{\cube{n}}$.
\begin{enumerate}
\item If $X$ is cocartesian, then $C^{\underline m}(X)\in\D^{[1]}$ is an isomorphism for all $|\underline m|=n-1$.
\item If \D is stable, then $X$ is cocartesian if and only if $C^{\underline m}(X)\in\D^{[1]}$ is an isomorphism for all $|\underline m|=n-1$ if and only if this is the case for some $\underline m$ with $|\underline m|=n-1$.
\end{enumerate}
For the last reformulation we invoke \autoref{cor:cotr-iter-cones-stable}. The first statement in the case of $n=2$ is hence a derivatorish version of the following classical fact from ordinary category theory. Given a pushout square
\[
\xymatrix{
x\ar[r]^-f\ar[d]_-{g}&y\ar[d]^-{g'}\\
x'\ar[r]_-{f'}&y'\pushoutcorner
}
\]
in a pointed and finitely cocomplete category, then the maps $\cok(f)\to\cok(f')$ and $\cok(g)\to\cok(g')$ are isomorphisms.
\end{rmk}

\section{Chunks of $n$-cubes}
\label{sec:sse-chunks}

In this section we begin our systematic study of abstract representations of various subposets of $n$-cubes. For this purpose, it is convenient to consider the cardinality filtration of the $n$-cube as summarized in \autoref{fig:mother}. We begin to analyze this diagram in quite some detail in the sections \S\S\ref{sec:sse-chunks}-\ref{sec:global} and we pursue this further in~\cite{bg:parasimplicial}. 

To begin with, in this section we show that for every $0\leq k\leq l\leq n$ the chunk $\cube{n}_{k\leq l}$ has the same abstract representation theory as $\cube{n}_{n-l\leq n-k}$. This generalizes the corresponding result for $\cube{n}_{\leq l}$ and $\cube{n}_{\geq n-l}$ (\autoref{thm:sse-chunks-special-case}) to arbitrary chunks. 

As a preparation, we collect some convenient facts in the following special case.

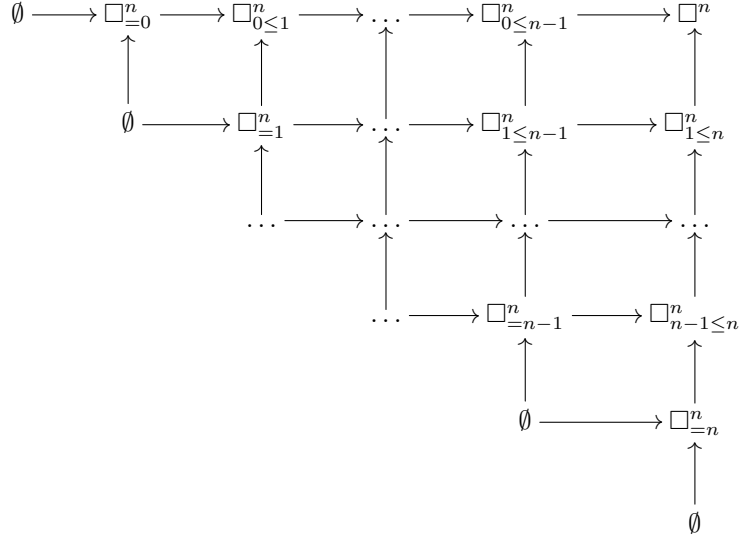
\begin{figure}
\begin{displaymath}
\xymatrix{
\emptyset\ar[r]&\cube{n}_{=0}\ar[r]&\cube{n}_{0\leq 1}\ar[r]&\ldots\ar[r]&\cube{n}_{0\leq n-1}\ar[r]&\cube{n}\\
&\emptyset\ar[r]\ar[u]&\cube{n}_{=1}\ar[r]\ar[u]&\ldots\ar[r]\ar[u]&\cube{n}_{1\leq n-1}\ar[r]\ar[u]&\cube{n}_{1\leq n}\ar[u]\\
&&\ldots\ar[u]\ar[r]&\ldots\ar[u]\ar[r]&\ldots\ar[u]\ar[r]&\ldots\ar[u]\\
&&&\ldots\ar[u]\ar[r]&\cube{n}_{=n-1}\ar[r]\ar[u]&\cube{n}_{n-1\leq n}\ar[u]\\
&&&&\emptyset\ar[r]\ar[u]&\cube{n}_{=n}\ar[u]\\
&&&&&\emptyset\ar[u]
}
\end{displaymath}
\caption{The cardinality filtration of the $n$-cube}
\label{fig:mother}
\end{figure}

\begin{notn}
Let \D be a stable derivator, $n\geq 0$ and $0\leq k\leq n$. By \autoref{cor:cotr=tr} an arbitrary $X\in\D^{\cube{n}}$ is $k$-cotruncated if and only if it is $(n-k)$-truncated. In this case we refer to $X$ as a \textbf{$k$-determined $n$-cube} and write
\[
\D^{\cube{n},k}=\D^{\cotr{n}{k}}=\D^{\tr{n}{n-k}}
\]
for the corresponding derivator. 
\end{notn}

\begin{lem}\label{lem:cof-fib-restrict}
Let \D be a stable derivator, let $n\geq 0$, and $0\leq k\leq n$.
\begin{enumerate}
\item The equivalence $(\cof^{\underline 1},\fib^{\underline 1})\colon\D^{\cube{n}}\simeq\D^{\cube{n}}$ restricts to an equivalence
\[
(\cof^{\underline{1}},\fib^{\underline 1})\colon\D^{\cube{n},k}\simeq\D^{\cube{n},\cube{n}_{\geq k+1}}.
\]
\item The equivalence $(\cof^{\underline 1},\fib^{\underline 1})\colon\D^{\cube{n}}\simeq\D^{\cube{n}}$ restricts to an equivalence
\[
(\cof^{\underline{1}},\fib^{\underline 1})\colon \D^{\cube{n},\cube{n}_{\leq n-(k+1)}}\simeq\D^{\cube{n},k}.
\]
\end{enumerate}
\end{lem}
\begin{proof}
The equivalence $\cof^{\underline 1}\colon\D^{\cube{n}}\toiso\D^{\cube{n}}$ induces by restriction an equivalence
\[
\cof^{\underline 1}\colon(\cof^{\underline 1})^{-1}(\D^{\cube{n},\cube{n}_{\geq k+1}})\toiso\D^{\cube{n},\cube{n}_{\geq k+1}}.
\]
For $X\in\D^{\cube{n}}$ we note that $\cof^{\underline 1}(X)$ vanishes on $\cube{n}_{\geq k+1}$ if and only if $C^{\underline m}X\cong 0$ for all $|\underline m|\geq k+1$ (see \autoref{con:cof-comp} and \autoref{con:C-comp}) if and only if $X$ is $k$-determined (\autoref{cor:cotr-iter-cones-stable}).
\end{proof}

\begin{cor}\label{cor:cof-fib-restrict}
Let \D be a stable derivator, let $n\geq 0$, and $0\leq k\leq n$. The equivalence $(\cof^{\underline 1},\fib^{\underline 1})\colon\D^{\cube{n}}\simeq\D^{\cube{n}}$ restricts to an equivalence
\[
(\cof^{\underline{1}},\fib^{\underline 1})\colon\D^{\cube{n},\cube{n}_{\geq k}}\simeq\D^{\cube{n},\cube{n}_{\leq n-k}}.
\]
\end{cor}
\begin{proof}
By restriction there is an equivalence of derivators
\[
\cof^{\underline 1}\colon (\cof^{\underline 1})^{-1}(\D^{\cube{n},\cube{n}_{\leq n-k}})\toiso\D^{\cube{n},\cube{n}_{\leq n-k}}
\]
and it suffices to identify this preimage. For $X\in\D^{\cube{n}}$ we note that $\cof^{\underline 1}(X)$ lies in $\D^{\cube{n},\cube{n}_{\leq n-k}}$ if and and only if
\[
\Omega^n \cof^{\underline 1}(X)\cong \fib^{\underline 2}(X)\in\D^{\cube{n},\cube{n}_{\leq n-k}}.
\]
Invoking \autoref{lem:cof-fib-restrict} twice, this is the case if and only if $\fib^{\underline 1}(X)\in\D^{\cube{n},k-1}$ which in turn is equivalent to $X\in\D^{\cube{n},\cube{n}_{\geq k}}$.
\end{proof}

Next, we analyze to some extent the generic rectangles in \autoref{fig:mother}.

\begin{lem}\label{lem:chunk}
For every $n\geq 0$ and $0\leq k\leq k'\leq l'\leq l\leq n$ the fully faithful inclusion
$\cube{n}_{k'\leq l'}\to\cube{n}_{k\leq l}$ factors as a cosieve followed by sieve 
\[
i_{l,l'}\circ j_{k,k'}\colon\cube{n}_{k'\leq l'}\to\cube{n}_{k\leq l'}\to\cube{n}_{k\leq l}
\]
and as a sieve followed by a cosieve
\[
j_{k,k'}\circ i_{l,l'}\colon\cube{n}_{k'\leq l'}\to\cube{n}_{k'\leq l}\to\cube{n}_{k\leq l}.
\]
The commutative square in $\cCat$
\begin{equation}\label{eq:chunk-chunk}
\vcenter{
\xymatrix{
\cube{n}_{k'\leq l'}\ar[d]_-{j_{k,k'}}\ar[r]^-{i_{l,l'}}\pullbackcorner\drtwocell\omit{\id}&\cube{n}_{k'\leq l}\ar[d]^-{j_{k,k'}}\\
\cube{n}_{k\leq l'}\ar[r]_-{i_{l,l'}}&\cube{n}_{k\leq l}
}
}
\end{equation}
is a pullback square and is homotopy exact. 
\end{lem}
\begin{proof}
The factorizations are immediate as is the fact that the square is a pullback. To conclude the proof it suffices to invoke \autoref{lem:opfibration}.
\end{proof}

This shows that in every derivator there are canonical isomorphisms
\[
(j_{k,k'})_!(i_{l,l'})^\ast\toiso(i_{l,l'})^\ast(j_{k,k'})_!.
\]
More interestingly, there is the following closely related observation.

\begin{lem}
For every $n\geq 0$ and $0\leq k\leq k'\leq l'\leq l\leq n$ and every derivator \D there are canonical isomorphisms
\begin{equation}\label{eq:zero-chunk-chunk}
(j_{k,k'})_!(i_{l,l'})_\ast\toiso(i_{l,l'})_\ast(j_{k,k'})_!\colon\D^{\cube{n}_{k'\leq l'}}\to\D^{\cube{n}_{k\leq l}}.
\end{equation}
\end{lem}
\begin{proof}
Left Kan extensions along cosieves are left extensions by initial objects and dually for right Kan extensions along sieves. Let $\D^{\cube{n}_{k\leq l},\exx}\subseteq\D^{\cube{n}_{k\leq l}}$ be spanned by all $X$ such that
\[
X|_{\cube{n}_{l'+1\leq l}}=\ast\qquad\text{and}\qquad X|_{\cube{n}_{k\leq k'-1}}=\emptyset.
\]
Then the above two compositions restrict to equivalences
\[
(j_{k,k'})_!(i_{l,l'})_\ast\colon\D^{\cube{n}_{k'\leq l'}}\toiso\D^{\cube{n}_{k\leq l},\exx},\quad (i_{l,l'})_\ast(j_{k,k'})_!\colon\D^{\cube{n}_{k'\leq l'}}\toiso\D^{\cube{n}_{k\leq l},\exx},
\]
and this exhibits both as inverse equivalences to the restriction morphism.
\end{proof}

We are mostly interested in this result for pointed derivators.

\begin{notn}
For every $n\geq 0$ and $0\leq k\leq k'\leq l'\leq l\leq n$ and every pointed derivator \D we denote by 
\begin{equation}\label{eq:z-def}
z=z_{k'\leq l'}^{k\leq l}\colon\D^{\cube{n}_{k'\leq l'}}\to\D^{\cube{n}_{k\leq l}}
\end{equation}
any of the naturally isomorphic morphisms \eqref{eq:zero-chunk-chunk}, which extends a given $[k',l']$-chunk by zeros on both sides. (The letter $z$ is meant to remind us of the word `zero'.) As a special case we have the equivalence
\begin{equation}\label{eq:z}
z=z_{k\leq l}=z_{k\leq l}^{0\leq n}\colon\D^{\cube{n}_{k\leq l}}\toiso\D^{\cube{n},(\cube{n}_{\leq k-1}\cup\cube{n}_{\geq l+1})}.
\end{equation}

As a variant of this, we can consider only left Kan extensions or only right Kan extensions along the functors in \eqref{eq:chunk-chunk}. By uniqueness of adjoints we hence obtain canonical isomorphisms
\[
(j_{k,k'})_!(i_{l,l'})_!\toiso(i_{l,l'})_!(j_{k,k'})_!\colon\D^{\cube{n}_{k'\leq l'}}\to\D^{\cube{n}_{k\leq l}}.
\]
For later reference, we denote any of these naturally isomorphic compositions by
\begin{equation}\label{eq:l-def}
l=l_{k'\leq l'}^{k\leq l}\colon\D^{\cube{n}_{k'\leq l'}}\to \D^{\cube{n}_{k\leq l}}.
\end{equation}
Similarly, in the case of right Kan extensions there are canonical isomorphisms
\[
(j_{k,k'})_\ast(i_{l,l'})_\ast\toiso(i_{l,l'})_\ast(j_{k,k'})_\ast\colon\D^{\cube{n}_{k'\leq l'}}\to\D^{\cube{n}_{k\leq l}},
\]
and we refer to these composition by the shorthand notation
\begin{equation}\label{eq:r-def}
r=r_{k'\leq l'}^{k\leq l}\colon\D^{\cube{n}_{k'\leq l'}}\to \D^{\cube{n}_{k\leq l}}.
\end{equation}
\end{notn}

With this preparation we now show that $\cube{n}_{k\leq l}$ and $\cube{n}_{n-l\leq n-k}$ are strongly stably equivalent.

\begin{con}\label{con:sse-chunks}
Let \D be a stable derivator, $n\geq 0$, and $0\leq k\leq l\leq n$. Up to equivalence, the codomain of the equivalence \eqref{eq:z} can be written differently. In fact, it follows from \autoref{lem:cof-fib-restrict} and \autoref{cor:cof-fib-restrict} that $\cof^{\underline{1}}\colon\D^{\cube{n}}\toiso\D^{\cube{n}}$ restricts to the following three outer equivalences in 
\[
\xymatrix{
&\D^{\cube{n},(\cube{n}_{0\leq k-1}\cup \cube{n}_{l+1\leq n})}\ar@/^1pc/[dr]^-{\cof^{\underline 1}}_-\sim&\\
\D^{(\cube{n},l),\cube{n}_{\geq n-k+1}}\ar@/^1pc/[ru]^-{\cof^{\underline 1}}_-\sim&\D^{\cube{n}_{k\leq l}}\ar[u]^-\sim_-z&\D^{(\cube{n},n-k),\cube{n}_{\leq n-l-1}}.\ar@/^2pc/[ll]^-{\cof^{\underline 1}}_-\sim
}
\]
By symmetry there is a similar diagram if we start with $\D^{\cube{n}_{n-l\leq n-k}}$ instead, and these two diagrams constitute the upper and the lower part in \autoref{fig:chunks}, respectively. The two additional equivalences pointing to the northeast and the southwest are restricted left Kan extension morphisms, while the two equivalences pointing to the northwest and the southeast are restrictions of right Kan extension morphisms. 
\end{con}

\begin{warn}
Note that we do not claim that \autoref{fig:chunks} commutes, and we come back to a more careful discussion of \autoref{fig:chunks} in \S\ref{sec:SSE-Serre}. 
\end{warn}

\begin{figure}
\begin{displaymath}
\xymatrix{
&\D^{\cube{n},(\cube{n}_{0\leq k-1}\cup \cube{n}_{l+1\leq n})}\ar@/^1pc/[dr]&\\
\D^{(\cube{n},l),\cube{n}_{\geq n-k+1}}\ar@/^1pc/[ru]&\D^{\cube{n}_{k\leq l}}\ar[u]\ar[ddl]\ar[rdd]&\D^{(\cube{n},n-k),\cube{n}_{\leq n-l-1}}\ar@/^2pc/[ll]\\
&&\\
\D^{(\cube{n},l),\cube{n}_{\leq k-1}}\ar@/^2pc/[rr]&\D^{\cube{n}_{n-l\leq n-k}}\ar[d]\ar[luu]\ar[uur]&\D^{(\cube{n},n-k),\cube{n}_{\geq l+1}}\ar@/^1pc/[ld]\\
&\D^{\cube{n},(\cube{n}_{0\leq n-l-1}\cup\cube{n}_{n-k+1\leq n})}\ar@/^1pc/[lu]&
}
\end{displaymath}
\caption{Strong stable equivalences between chunks}
\label{fig:chunks}
\end{figure}
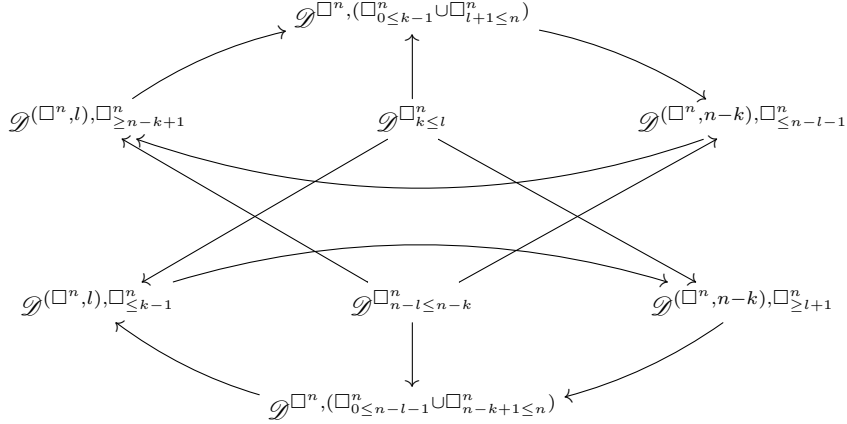

\begin{thm}\label{thm:sse-chunks}
For every $n\geq 0$ and every $0\leq k\leq l\leq n$ the chunks $\cube{n}_{k\leq l}$ and $\cube{n}_{n-l\leq n-k}$ are strongly stably equivalent.
\end{thm}
\begin{proof}
It suffices to note that the graph of \autoref{fig:chunks} is path-connected. In fact, by the details in \autoref{con:sse-chunks}, all morphisms in that diagram are given by combinations of homotopy finite Kan extensions in stable derivators, and exact morphisms of stable derivators are pseudo-natural with respect to such Kan extensions \cite[\S9]{groth:revisit}.
\end{proof}

In the above proof we did not make specific a particular strong stable equivalence, instead we merely invoked the path-connectedness of \autoref{fig:chunks}. In the following two sections we improve on this and, as a preparation, we conclude this section by introducing the following notation for the morphisms in \autoref{fig:chunks}. 

\begin{rmk}\label{rmk:describe-fig-chunks}
Let \D be a stable derivator, $n\geq 0$, and $0\leq k\leq l\leq n$. The six morphisms in the two circles of \autoref{fig:chunks} are restrictions of $\cof^{\underline 1}$ to suitable equivalences. For the remaining six cases, we consider the morphism $i=i_{k\leq l}^{0\leq n}\colon(k,l)\to(0,n)$ in $\chunk{n}$ and its image $i^\vee=i_{n-l\leq n-k}^{0\leq n}$ under the chunk symmetry \eqref{eq:chunk-symm}. Then, using suggestive notation, we have
\begin{align*}
\searrow\colon\quad&i_\ast^\simeq=r^\simeq,\\
\swarrow\colon\quad&i_!^\simeq=l^\simeq,\\
\nwarrow\colon\quad&(i^\vee)_\ast^\simeq=(r^\vee)^\simeq,\\
\nearrow\colon\quad& (i^\vee)_!^\simeq=(l^\vee)^\simeq,\\
\uparrow\colon\quad&z^\simeq,\quad\text{and}\\
\downarrow\colon\quad&(z^\vee)^\simeq.
\end{align*}
\end{rmk}

\section{Strong stable equivalences and Serre equivalences}
\label{sec:SSE-Serre}

In this section we address the lack of commutativity of \autoref{fig:chunks}. Traveling along one of the circles gives rise to the auto-equivalence $\Sigma^n$ (\autoref{cor:cof-S}). There are two other types of elementary subcells in \autoref{fig:chunks}:
\begin{enumerate}
\item There are two triangles in the central part of the diagram.
\item There are two hexagons, one of them is concentrated in the left part and the other one in the right part of \autoref{fig:chunks}.
\end{enumerate}
These respective subcells give rise to Serre equivalences and to two different strong stable equivalences for chunks. Let us begin by studying the two triangles more carefully (see \autoref{rmk:describe-fig-chunks} for an explanation of the notation).

\begin{defn}\label{defn:Serre-mor}
Let $\D$ be a stable derivator, $n \geq 0$, and $0 \leq k \leq l \leq n$.
\begin{enumerate}
\item The \textbf{Serre morphism} $S:=S_{k,l}:=S_{k,l}^{(n)}\colon \D^{\cube{n}_{k\leq l}}\toiso\D^{\cube{n}_{k\leq l}}$ is given by
\[
(r^{\simeq})^{-1}\circ\cof^{\underline{1}}\circ l^{\simeq}\colon \D^{\cube{n}_{k\leq l}}\toiso\D^{\cube{n}_{k\leq l}}.
\]
\item The morphism $\tilde{S}:=\tilde{S}_{k,l}:=\tilde{S}_{k,l}^{(n)}\colon \D^{\cube{n}_{n-l \leq n-k}}\toiso\D^{\cube{n}_{n-l \leq n-k}}$ is given by
\[
((r^{\vee})^{\simeq})^{-1}\circ\cof^{\underline{1}}\circ (l^{\vee})^{\simeq}
\]
and called the \textbf{opposite Serre morphism}. 
\end{enumerate}
\end{defn}

Before we justify the terminology by some examples, we make the following remark.

\begin{rmk}\label{rmk:S-dual-S}
Let $n\geq 0$ and $0\leq k\leq l\leq n$. The category $\cube{n}$ is self-dual, and under this equivalence the chunk $\cube{n}_{k\leq l}$ is identified with the dual chunk
\[
(\cube{n}_{k\leq l})^\vee=\cube{n}_{n-l\leq n-k}.
\]
This duality corresponds in \autoref{fig:chunks} to a rotation by $\pi$, and we obtain a corresponding identification
\[
\tilde{S}_{k,l}=S_{n-l,n-k}.
\]
\end{rmk}

We now relate these Serre morphisms to other instances of Serre morphisms in abstract representation theory. A more detailed discussion of these Serre morphisms and global versions of them will appear in \cite{bg:global}. In the cases where a chunk $\cube{n}_{k\leq l}$ is (strongly stably equivalent to) a product of $A_m$-quivers we can compare the Serre morphisms of \autoref{defn:Serre-mor} with the Serre morphisms discussed in \cite{gst:Dynkin-A}. This shows, in particular, that these Serre equivalences are not necessarily identities.

\begin{eg}\label{eg:products-A2}
For $n\geq 0$ the full cube $\cube{n}$ is the $n$th power of the $A_2$-quiver $[1]$. For $A_2$-quivers Serre equivalences are given by the cofiber morphism $\mathsf{cof}$ (\cite[Lem.~5.16]{gst:Dynkin-A}), and since $S_{0,n}^{(n)}=\mathsf{cof}^{\underline{1}}$, we immediately see that both definitions are compatible.
\end{eg}

\begin{eg}\label{eg:A3}
The chunk $\cube{2}_{\leq 1}$ is an $A_3$-quiver, and in this example we show that the corresponding Serre morphisms from \autoref{defn:Serre-mor} agrees with the one from \cite[\S5]{gst:Dynkin-A} (and hence with the classical one from representation theory). To distinguish these notationally, we temporarily write $\hat{S}$ for these latter Serre morphisms.
\begin{enumerate}
\item We begin with a calculation of $S_{0,1}^{(2)}$, and to this end we consider the diagram
\[
\xymatrix{
x_{\emptyset} \ar[r]^{f} \ar[d]_{g} & x_1 \ar[r] \ar[d] & Cf \ar[d]^-\sim\\
x_2 \ar[r] \ar[d] & x_1 \cup_{x_{\emptyset}} x_2\pushoutcorner \ar[r] \ar[d] & Cf \ar[d]\\
Cg \ar[r]_-\sim & Cg \ar[r] & 0.
}
\]
By definition of $S_{0,1}^{(2)}$, we first complete the span $X=(x_2 \leftarrow x_{\emptyset} \rightarrow x_1)$ to a cocartesian square. We next extend all morphisms to cofiber sequences in order to obtain $\mathsf{cof}^{\underline{1}}l^{\simeq}(X)$ as the lower right square in the above diagram. By abuse of notation we wrote $Cf$ and $Cg$ twice, and this is justified by \autoref{cor:cotr-iter-cones-II}. To conclude the description of $S_{0,1}^{(2)}(X)$ we restrict this square to the span $(Cg \leftarrow x_1 \cup_{x_{\emptyset}} x_2 \rightarrow Cf)$. 
\item On the other hand the Serre morphism $\hat{S}$ for the $A_3$-quiver is up to a suspension constructed by iterated reflections at an admissible sequence of sinks. In our case of the source of valence two this amounts to two reflections at the outer vertices followed by a reflection at the central one. To describe this we consider the following diagram consisting of bicartesian squares:

\[
\xymatrix{
Fg\times_{x_{\emptyset}} Ff \ar[r] \ar[d] & Ff \ar[r] \ar[d] & 0\ar[d]\\
Fg \ar[r] \ar[d] & x_{\emptyset} \ar[r]^{f} \ar[d]^{g} & x_1\\
0 \ar[r] & x_2
}
\]
We obtain $\hat{S}(X)$ as the supension of the span $(Fg \leftarrow Fg\times_{x_{\emptyset}} Ff \rightarrow Ff)$. Up to a flip, this suspension is obtained by extending the previous diagram to the following mesh
\begin{equation}\label{eq:S-square}
\vcenter{
\xymatrix{
Fg\times_{x_{\emptyset}} Ff \ar[r] \ar[d] & Ff \ar[r] \ar[d] & 0 \ar[d]\\
Fg \ar[r] \ar[d] & x_{\emptyset} \ar[r]^{f} \ar[d]^{g} & x_1 \ar[d] \ar[r] & 0 \ar[d]\\
0 \ar[r] & x_2 \ar[r] \ar[d] & x_1 \cup_{x_{\emptyset}} x_2 \ar[r] \ar[d] & Cg \ar[d]\\
& 0 \ar[r] & Cf \ar[r] & \Sigma x_{\emptyset}
}
}
\end{equation}
in turn by adding bicartesian squares (see \cite[\S4]{gst:Dynkin-A}). 
\end{enumerate}
These calculations show that the two Serre morphisms are naturally isomorphic. In particular, recycling the fractionally Calabi--Yau property from \cite[Thm.~5.19]{gst:Dynkin-A} we conclude that $(S_{0,1}^{(2)})^4\cong\Sigma^2$. Moreover, \eqref{eq:S-square} shows that there is a natural isomorphism $(S_{0,1}^{(2)})^2\cong\sigma_{1,2}^*\Sigma$. 
\end{eg}

Dual arguments apply to the case of $\cube{n}_{\geq 1}$. 

\begin{eg}\label{eg:source3}
The chunk $\cube{3}_{\leq 1}$ is the trivalent source and hence strongly stably equivalent to the product of two $A_2$-quivers (\cite[Thm.~9.2]{gst:basic}). Using arguments similar to (but more involved than) the ones in \autoref{eg:A3}, one shows that the Serre morphisms are also compatible in this case. As a consequence of \cite[Thm.~5.19]{gst:Dynkin-A} we obtain the fractionally Calabi--Yau property $(S_{0,1}^{(3)})^3\cong\Sigma^2$. 
\end{eg}

\begin{rmk}
In contrast to the previous cases, it is not always true that suitable powers of the Serre morphisms $S^{(n)}_{k,l}$ and the suspensions are isomorphic. For instance, this is impossible for the chunk $\cube{3}_{1\leq 2}$ as we can already see at the level of underlying diagrams. In fact, let $X \in \D^{\cube{3}_{1\leq 2}}$ be  obtained from $x\in\D$ by applying $\iota_{1\leq 2}^*\circ \cof^{\underline{1}}\circ (\iota_{=2})_!\circ \pi_{\cube{3}_{=2}}^\ast$. The underlying diagram is given by:
\[
X_{\underline{m}}=\begin{cases}
x^2=x\oplus x, \text{ if } \vert \underline{m}\vert=1,\\
x, \text{ if } \vert \underline{m}\vert=2
\end{cases}
\]
Inductively, one can show that
\[
(S^{(3)}_{1,2})^nX_{\underline{m}}=\begin{cases}
\Sigma^nx^{2+2n}, \text{ if } \vert \underline{m}\vert=1,\\
\Sigma^{n}x^{1+2n}, \text{ if } \vert \underline{m}\vert=2.
\end{cases}
\]

We thank Jan {\v S}{\v t}ov{\'\i}{\v c}ek for pointing out to us, that there are also purely representation theoretic arguments to obtain the same conclusion. In fact, over a field~$k$ the poset $\cube{3}_{1\leq 2}$ is derived equivalent to the extended Dynkin quiver of type $\tilde{A}_5$, which has tame infinite representation type and these are not fractionally Calabi--Yau. And by passing to the derivator $\D_k$ of the field $k$, the abstract fractionally Calabi--Yau property would imply the classical one which is impossible.
\end{rmk}

Next, we turn to a discussion of the hexagons in \autoref{fig:chunks} which will lead to two different strong stable equivalences. We again refer to \autoref{rmk:describe-fig-chunks} for the notation.

\begin{con}\label{con:H-H-H-H}
Let \D be a stable derivator, $n\geq 0$, and $0\leq k\leq l\leq n$. The hexagon on the right in \autoref{fig:chunks} gives rise to the two equivalences
\[
H^{r1},H^{r2}\colon \D^{\cube{n}_{k\leq l}}\toiso\D^{\cube{n}_{n-l \leq n-k}}
\]
which are respectively given by
\[
H^{r1}=((l^{\vee})^{\simeq})^{-1}\circ\cof^{\underline{1}}\circ z^{\simeq} \qquad\text{ and }\qquad H^{r2}=((z^{\vee})^{\simeq})^{-1}\circ\cof^{\underline{1}}\circ r^{\simeq}.
\]
Similarly, the hexagon on the left yields the two equivalences
\[
H^{l1},H^{l2}\colon \D^{\cube{n}_{n-l\leq n-k}}\toiso\D^{\cube{n}_{k \leq l}}
\]
defined by the formulas
\[
H^{l1}=(l^{\simeq})^{-1}\circ\cof^{\underline{1}}\circ (z^{\vee})^{\simeq} \qquad\text{ and }\qquad H^{l2}=(z^{\simeq})^{-1}\circ\cof^{\underline{1}}\circ (r^{\vee})^{\simeq}.
\]

To distinguish these equivalences for the various chunks $\cube{n}_{k\leq l}$ notationally, we also write
\[
H^{r1}=H^{r1}_{k,l}=(H^{r1}_{k,l})^{(n)},
\]
and similarly in the remaining three cases. Similar to \autoref{rmk:S-dual-S}, we note that the self-duality of $\cube{n}$ yields the identifications
\[
H^{l1}_{k,l}=H^{r1}_{n-l,n-k}\qquad \text{ and } \qquad H^{l2}_{k,l}=H^{r2}_{n-l,n-k}.
\]
\end{con}

We observe that the compositions $H^{r1}$ and $H^{r2}$ are exactly the boundary of right hexagon in \autoref{fig:chunks}. Thus the following result is equivalent to the commutativity of this hexagon.

\begin{lem}\label{lem:hexagon}
Let \D be a stable derivator, $n\geq 0$, and $0\leq k\leq l\leq n$.  The two strong stable equivalences $H^{r1},H^{r2}\colon\D^{\cube{n}_{k\leq l}}\toiso\D^{\cube{n}_{n-l\leq n-k}}$ are canonically isomorphic.
\end{lem}

\begin{proof}
Since all morphisms in \autoref{fig:chunks} are equivalences, it is sufficient to show that the compositions
\begin{equation}\label{eq:hexagon}
\mathsf{cof}^{\underline{1}}\circ(r)^{\simeq}\circ((z)^{\simeq})^{-1}\qquad \text{ and } \qquad(z^{\vee})^{\simeq}\circ((l^{\vee})^{\simeq})^{-1}\circ\mathsf{cof}^{\underline{1}}
\end{equation}
are naturally isomorphic. We see immediately that there are natural isomorphisms
\[
(r)^{\simeq}\circ((z)^{\simeq})^{-1} \cong(\iota_{\geq k})_*(\iota_{\geq k})^* \quad\text{and}\quad (z^{\vee})^{\simeq}\circ((l^{\vee})^{\simeq})^{-1}\cong(\iota_{\leq n-k})_*(\iota_{\leq n-k})^*.
\]
Thus, for $X\in\D^{\cube{n},(\cube{n}_{0\leq k-1}\cup\cube{n}_{l+1\leq n})}$, the unit $\eta$ of the adjunction $((\iota_{\geq k})^*,(\iota_{\geq k})_*)$ can be identified with a morphism $X \xrightarrow{\eta_X}(r)^{\simeq}\circ((z)^{\simeq})^{-1}X$ such that $(\iota_{\geq k})^*(\eta_X)$ is an isomorphism. Then \autoref{cor:cof-stability} implies that $(\iota_{\leq n-k})^*\mathsf{cof}^{\underline{1}}(\eta_X)$
isomorphism, hence so is $(\iota_{\leq n-k})_*(\iota_{\leq n-k})^*\mathsf{cof}^{\underline{1}}(\eta_X)$. Since $(\iota_{\leq n-k})_*$ is right extension by zero, the corresponding unit yields $\id\toiso(\iota_{\leq n-k})_*(\iota_{\leq n-k})^*$ on $\D^{\cube{n},(\cube{n}_{0\leq n-l-1}\cup\cube{n}_{n-k+1\leq n})}$, and this yields the natural isomorphism of the compositions in \eqref{eq:hexagon}.
\end{proof}

\begin{rmk}\label{rmk:hexagon-left}
Let \D be a stable derivator, $n\geq 0$, and $0\leq k\leq l\leq n$.  By precisely the same argument one shows that the equivalences
$H^{l1}$ and $H^{l2}$ constituting the hexagon on the left in \autoref{fig:chunks} are canonically isomorphic. 
\end{rmk}

\begin{defn}\label{defn:Phi-Psi}
Let \D be a stable derivator, $n\geq 0$, and $0\leq k\leq l\leq n$. We write
\[
\Phi_{k,l}=\Phi^{(n)}_{k,l}\colon\D^{\cube{n}_{k\leq l}}\toiso\D^{\cube{n}_{n-l\leq n-k}}
\]
for any of the naturally isomorphic equivalences $H^{r1}\cong H^{r2}$, and refer to it as the \textbf{right strong stable equivalence}. Similarly, the \textbf{left strong stable equivalence}
\[
\Psi_{k,l}=\Psi^{(n)}_{k,l}\colon\D^{\cube{n}_{n-l\leq n-k}}\toiso\D^{\cube{n}_{k\leq l}}
\]
is defined to be any of the naturally isomorphic $H^{l1}\cong H^{l2}$.
\end{defn}

\begin{cor}\label{cor:modif-local}
Let $\D$ be a stable derivator, $n \geq 0$, and $0 \leq k\leq l\leq n$. Then there is a natural isomorphism
\[
\Sigma^n\cong \Psi_{k,l}\circ\tilde{S}_{k,l}\circ\Phi_{k,l}\colon\D^{\cube{n}_{k\leq l}}\toiso\D^{\cube{n}_{k\leq l}}
\]
and
\[
\Sigma^n\cong \Phi_{k,l}\circ S_{k,l}\circ\Psi_{k,l}\colon\D^{\cube{n}_{n-l\leq n-k}}\toiso\D^{\cube{n}_{n-l\leq n-k}}.
\]
\end{cor}
\begin{proof}
This follows by unraveling definitions and \autoref{cor:cof-S}.
\end{proof}

\begin{cor}\label{cor:dodecagon}
Let $\D$ be a stable derivator, $n \geq 0$, and $0 \leq k\leq l\leq n$. Then there are natural isomorphisms
\[
S_{k,l}\cong \Psi_{k,l}\circ\tilde{S}_{k,l}\circ \Psi^{-1}_{k,l}\cong \Phi_{k,l}^{-1}\circ\tilde{S}_{k,l}\circ \Phi_{k,l}.
\]
\end{cor}

\begin{proof}
In the proof we drop indices from notation and use $H^{l1}$ as a model for $\Psi$. We show that $\id\cong H^{l1}\circ\tilde{S}\circ (H^{l1})^{-1}\circ S^{-1}$, and begin by observing that the right hand side is given by  
\[
(l^{\simeq})^{-1}\circ\cof^{\underline{1}}\circ (z^{\vee})^{\simeq}\circ ((r^{\vee})^{\simeq})^{-1}\circ\cof^{\underline{1}}\circ (l^{\vee})^{\simeq} \circ ((z^{\vee})^{\simeq})^{-1}\circ\fib^{\underline{1}}\circ l^{\simeq}\circ (l^{\simeq})^{-1}\circ\fib^{\underline{1}}\circ r^{\simeq}.
\]
The two morphisms $l^{\simeq}$ and $(l^{\simeq})^{-1}$ at the end of this expression cancel and we are left with the first line in
\begin{align*}
&(l^{\simeq})^{-1}\circ\cof^{\underline{1}}\circ (z^{\vee})^{\simeq}\circ ((r^{\vee})^{\simeq})^{-1}\circ\cof^{\underline{1}}\circ (l^{\vee})^{\simeq} \circ ((z^{\vee})^{\simeq})^{-1}\circ\fib^{\underline{1}}\circ\fib^{\underline{1}}\circ r^{\simeq}\\
\cong\;\;&\Omega^n\circ(l^{\simeq})^{-1}\circ\cof^{\underline{1}}\circ (z^{\vee})^{\simeq}\circ ((r^{\vee})^{\simeq})^{-1}\circ\cof^{\underline{1}}\circ (l^{\vee})^{\simeq} \circ ((z^{\vee})^{\simeq})^{-1}\circ\cof^{\underline{1}}\circ r^{\simeq}\\
\cong\;\;&\Omega^n\circ(l^{\simeq})^{-1}\circ\cof^{\underline{1}}\circ (z^{\vee})^{\simeq}\circ ((r^{\vee})^{\simeq})^{-1}\circ\cof^{\underline{1}}\circ (l^{\vee})^{\simeq} \circ ((l^{\vee})^{\simeq})^{-1}\circ\cof^{\underline{1}}\circ z^{\simeq}\\
\cong\;\;&\Omega^n\circ(l^{\simeq})^{-1}\circ\cof^{\underline{1}}\circ (z^{\vee})^{\simeq}\circ ((r^{\vee})^{\simeq})^{-1}\circ\cof^{\underline{1}}\circ\cof^{\underline{1}}\circ z^{\simeq}\\
\cong\;\;&(l^{\simeq})^{-1}\circ\cof^{\underline{1}}\circ (z^{\vee})^{\simeq}\circ ((r^{\vee})^{\simeq})^{-1}\circ\fib^{\underline{1}}\circ z^{\simeq}\\
\cong\;\;& \id.
\end{align*}
Since all morphisms under consideration are exact, we invoke the natural isomorphism $\fib^{\underline{3}}\cong\Omega^n$ (\autoref{cor:cof-S}) in order to obtain the first isomorphism. The second isomorphism is given by the commutativity of the right hexagon (\autoref{lem:hexagon}). The third and the forth isomorphisms are again a cancellation and an application of \autoref{cor:cof-S}, while the remaining one is the commutativity of the hexagon on the left (\autoref{rmk:hexagon-left}). The second natural isomorphism follows from the chain
\begin{align*}
\Psi_{k,l}\circ\tilde{S}_{k,l}\circ \Psi^{-1}_{k,l}&\cong \Psi_{k,l}\circ\tilde{S}_{k,l}\circ \Phi_{k,l}\circ \Phi^{-1}_{k,l}\circ\Psi^{-1}_{k,l}\\
& \cong \Sigma^n\circ \Phi^{-1}_{k,l}\circ\Psi^{-1}_{k,l}\\
& \cong \Phi^{-1}_{k,l}\circ\Psi^{-1}_{k,l}\circ\Sigma^n\\
& \cong \Phi^{-1}_{k,l}\circ\Psi^{-1}_{k,l}\circ\Psi_{k,l}\circ\tilde{S}_{k,l}\circ \Phi_{k,l}\\
& \cong \Phi^{-1}_{k,l}\circ\tilde{S}_{k,l}\circ \Phi_{k,l}.
\end{align*}
Here, the second and fourth natural isomorphisms are by \autoref{cor:modif-local}, the third one is by exactness and the remaining two are trivial.
\end{proof}

\begin{rmk}
Let \D be a stable derivator, $n\geq 0$, and $0\leq k\leq l\leq n$. If $F\colon\D^{\cube{n}_{k\leq l}}\rightarrow\D^{\cube{n}_{k\leq l}}$ is a composition of morphisms appearing in \autoref{fig:chunks}, then there are $p,q \in \mathbb{Z}$ such that $F\cong(\Sigma^n)^p\circ S^q$. The key step to this is \autoref{cor:dodecagon}.
\end{rmk}

In the following section we make precise in which sense these strong stable equivalences and Serre morphisms are compatible for the various chunks of $\cube{n}$.

\section{Global Serre duality}
\label{sec:global}

In this section we show that for every fixed dimension $n$, the Serre equivalences and the strong stable equivalences introduced in the previous section depend pseudo-naturally on the chunks $\cube{n}_{k\leq l}$. To this end we begin by defining the category $\chunk{n}$ of chunks in $\cube{n}$.

\begin{defn}
The \textbf{category} $\chunk{n}\in\cCat$ \textbf{of chunks} in $\cube{n},n\geq 0,$ is the following partially ordered set, considered as a category.
\begin{enumerate}
\item Objects are pairs $(k,l)$ of natural numbers such that $0\leq k\leq l\leq n$.
\item For pairs $(k,l), (k',l')$ we set $(k',l')\leq (k,l)$ if and only if $k\leq k'\leq l'\leq l'$. We denote the unique morphism by
\[
i=i^{k\leq l}_{k'\leq l'}\colon(k',l')\to(k,l).
\]
\end{enumerate}
\end{defn} 

\begin{rmk}
\begin{enumerate}
\item The category $\chunk{n}$ is the \emph{twisted morphism category} of the simplex $[n]=(0<\ldots <n)$. This perspective on $\chunk{n}$ is closely related to the paracyclic $S_\bullet$-construction, and we come back to this in more detail in \cite{bg:parasimplicial}.
\item The category $[n]$ is uniquely isomorphic to its opposite category by means of the assignment $[n]\toiso[n]\op\colon k\mapsto n-k$. Considering this isomorphism as an identification, the functoriality of the twisted morphism category again yields the  \textbf{chunk symmetry}
\begin{equation}\label{eq:chunk-symm}
(-)^{\vee}\colon\chunk{n}\toiso\chunk{n}\colon (k,l)\mapsto(n-l,n-k),
\end{equation}
which already appeared in \autoref{rmk:S-dual-S}.
\end{enumerate}
\end{rmk}

Let us recall from the introduction that $\cDER_{\mathrm{St},\exx}$ denotes the $2$-category of stable derivators, exact morphisms, and natural transformations. 

\begin{prop}\label{prop:Z-R-L}
Let \D be a stable derivator and $n\geq 0$. There are three pseudo-functors 
\[
Z=Z^{(n)}(\D),\;L=L^{(n)}(\D),\; R=R^{(n)}(\D)\colon\chunk{n}\to\cDER_{\mathrm{St},\exx}
\]
defined as follows.
\begin{enumerate}
\item On objects $(k,l)\in\chunk{n}$ they are defined uniformly by 
\[
Z(k,l)=R(k,l)=L(k,l)=\D^{\cube{n}_{k \leq l}}.
\]
\item Invoking \eqref{eq:z-def}, \eqref{eq:l-def}, and \eqref{eq:r-def}, respectively, the images of a morphism $i^{k\leq l}_{k'\leq l'}$ in $\chunk{n}$ are defined by 
\[
Z(i_{k'\leq l'}^{k\leq l})=z_{k'\leq l'}^{k\leq l},\quad
L(i_{k'\leq l'}^{k\leq l})=l_{k'\leq l'}^{k\leq l},\quad\text{and}\quad
R(i_{k'\leq l'}^{k\leq l})=r_{k'\leq l'}^{k\leq l}.
\]
\end{enumerate}
\end{prop}
\begin{proof}
Stable derivators are closed under exponentials, and compositions of Kan extensions between stable derivators are exact morphisms, showing that the assignments indeed take values in $\cDER_{\mathrm{St},\exx}$. The pseudo-functoriality follows from uniqueness of adjoints and from the compatibility of canonical mates with pasting.
\end{proof}

\begin{notn}
Let \D be a stable derivator and $n\geq 0$. We denote by
\[
Z^\vee=Z^\vee(\D), R^\vee=R^\vee(\D), L^\vee=L^\vee(\D)\colon\chunk{n}\to\cDER_{\mathrm{St},\exx}
\]
the pseudo-functors which are obtained from the corresponding pseudo-functors in \autoref{prop:Z-R-L} by precomposition with the chunk symmetry \eqref{eq:chunk-symm},
\[
Z^\vee(\D)=Z(\D)\circ(-)^\vee,\quad R^\vee(\D)=R(\D)\circ(-)^\vee,\quad L^\vee(\D)=L(\D)\circ(-)^\vee.
\]
\end{notn}

We now show that the equivalences from \S\ref{sec:SSE-Serre} assemble into pseudonatural equivalences and we begin with the Serre morphisms.

\begin{thm}\label{thm:Serre}
Let \D be a stable derivator and $n\geq 0$. The local Serre equivalences $S_{k,l}$ for $(k,l)\in\chunk{n}$ assemble to a pseudo-natural equivalence
\[
S=S^{(n)}\colon L^{(n)}(\D)\toiso R^{(n)}(\D)\colon \chunk{n}\to\cDER_{\mathrm{St},\exx}.
\]
\end{thm}
\begin{proof}
We already know that the Serre morphisms are equivalences, and it remains to establish the pseudo-naturality. For this purpose, let $i\colon(k',l')\to(k,l)$ be a morphism in $\chunk{n}$. Unraveling definitions, we have to show that \autoref{fig:Serre-chunks} commutes up to coherent natural isomorphisms, thereby yielding natural isomorphisms
\begin{equation}\label{eq:Serre-conjugate}
S_{k,l}\circ i_!\cong i_\ast\circ S_{k',l'}.
\end{equation}
In this diagram the horizontal morphisms in the middle are the natural inclusions which exist by the assumption that $0\leq k\leq k'\leq l'\leq l\leq n$. Hence, the square in the middle can be chosen to commute on the nose. The upper square and the lower square, respectively, commute up to canonical natural isomorphism by uniqueness of adjoints. Since the pseudo-functoriality constraints of $L(\D)$ and $R(\D)$ are also induced by uniqueness of adjoints (\autoref{prop:Z-R-L}), this establishes the pseudo-naturality of the local Serre equivalences. 
\end{proof}

\begin{figure}
\begin{displaymath}
\xymatrix{
\D^{\cube{n}_{k' \leq l'}} \ar[d]_{(Li)^\simeq} \ar[r]^{Li} & \D^{\cube{n}_{k \leq l}} \ar[d]^{(Li)^\simeq}\\
\D^{(\cube{n},l'),\cube{n}_{\leq k'-1}} \ar[d]_{\cof^{\underline{1}}} \ar[r]^-\subseteq & \D^{(\cube{n},l),\cube{n}_{\leq k-1}} \ar[d]^{\cof^{\underline{1}}}\\
\D^{(\cube{n},n-k'),\cube{n}_{\geq l'+1}} \ar[r]_-\subseteq & \D^{(\cube{n},n-k),\cube{n}_{\geq l+1}} \\
\D^{\cube{n}_{k' \leq l'}}\ar[u]^{(Ri)^\simeq} \ar[r]_{Ri} & \D^{\cube{n}_{k \leq l}}\ar[u]_{(Ri)^\simeq}
}
\end{displaymath}
\caption{Global Serre duality $S^{(n)}$ for $\chunk{n}$}
\label{fig:Serre-chunks}
\end{figure}

\begin{rmk}
Let \D be a stable derivator and $n\geq 0$. Dually, also the opposite local Serre morphisms $\tilde{S}_{k,l}$ for $(k,l)\in\chunk{n}$ assemble to a pseudo-natural equivalence
\[
\tilde{S}\colon L^{\vee}(\D)\toiso R^{\vee}(\D)\colon \chunk{n}\to\cDER_{\mathrm{St},\exx},
\]
and \autoref{cor:dodecagon} immediately implies that $S^{\vee}=\tilde{S}$.
\end{rmk}

The formalism of these compatible Serre equivalences is fairly rich, and we now discuss some immediate consequences. We intend to come back to this more systematically in \cite{bg:global}. By \eqref{eq:Serre-conjugate} left and right Kan extensions along morphisms in $\chunk{n}$ correspond to each other under conjugation by Serre equivalences. This immediately yields the existence of infinite chains of adjunctions.

\begin{notn}
Let $f$ be a morphism in a 2-category, such that all iterated adjoints of $f$ exist. Then $f[n]$ denotes the $n$th iterated right adjoint of $f$, if $n\in\mathbb{Z}$ is positive, or the $-n$th iterated left adjoint of $f$ if $n\in\mathbb{Z}$ is negative.
\end{notn}

\begin{cor}\label{cor:infty-adjoints}
Let \D be a stable derivator, let $n\geq 0$, and let $i\colon(k',l')\to(k,l)$ be a morphism in $\chunk{n}$. The restriction morphism
 $i^*\colon\D^{\cube{n}_{k\leq l}} \rightarrow \D^{\cube{n}_{k'\leq l'}}$ generates an infinite chain of adjunctions which for $n\in\lZ$ is given by
\begin{enumerate}
\item $i^*[2n]=S_{k',l'}^{n}i^*S_{k,l}^{-n},$
\item $i^*[2n+1]=S_{k,l}^{n}i_*S_{k',l'}^{-n}\cong S_{k,l}^{n+1}i_!S_{k',l'}^{-n-1}.$
\end{enumerate}
\end{cor}
\begin{proof}
This follows from \autoref{thm:Serre} and the corresponding equation \eqref{eq:Serre-conjugate}.
\end{proof}

One can see immediately from the definition that the Serre equivalences
\[
S_{0,0}^{(n)}\colon\D^{\cube{n}_{=0}}\toiso\D^{\cube{n}_{=0}}\qquad\text{and}\qquad S_{n,n}^{(n)}\colon\D^{\cube{n}_{=n}}\toiso\D^{\cube{n}_{=n}}
\]
 are naturally isomorphic to the identity morphisms. In fact, we show that this holds for $S_{k,k}^{(n)}$ for all $0\leq k\leq n$.

\begin{prop}\label{prop:discrete-Serre}
Let $\D$ be a stable derivator, $n \geq 0$, and $0\leq k \leq n$. There are natural isomorphisms
\[
S_{k,k}^{(n)}\cong\id\colon\D^{\cube{n}_{=k}}\toiso\D^{\cube{n}_{=k}}.
\]
\end{prop}
\begin{proof}
Since $\cube{n}_{=k}$ is a discrete category, axiom (Der1) exhibits $\D^{\cube{n}_{=k}}$ as a corresponding power of \D,
\begin{equation}\label{eq:disjoint}
\D^{\cube{n}_{=k}}\toiso\prod_{\underline{m}\in \cube{n}_{=k}} \D.
\end{equation}
Hence, it suffices to show that for every $\underline{m}\in\cube{n}_{=k}$ there is a natural isomorphism $\underline{m}^\ast S_{k,k}\cong \underline{m}^\ast\colon\D^{\cube{n}_{=k}}\to\D$. For every such $\underline{m}$ with complement $\underline{m}^\vee$ we make the calculation
\begin{align*}
&\underline{m}^*\circ S_{k,k}\\
=\;\;&\underline{m}^*\circ ((Ri)^{\simeq})^{-1}\circ\mathsf{cof}^{\underline{1}}\circ(Li)^{\simeq}\\
=\;\; &\underline{m}^*\circ \mathsf{cof}^{\underline{1}}\circ(Li)^{\simeq}\\
\cong\;\;& C^{|\underline{m}|}\circ (1_{\underline{m}^\vee})^\ast\circ (Li)^\simeq.
\end{align*}
The first equation is by \autoref{defn:Serre-mor}, the second equation holds since $((Ri)^\simeq)^{-1}$ is a restriction morphism, and the isomorphism comes from \autoref{lem:iter-cones}. Since Kan extensions and restrictions in unrelated variables commute, we can continue the calculation by the first isomorphism in 
\begin{align*}
&C^{|\underline{m}|}\circ (1_{\underline{m}^\vee})^\ast\circ (Li)^\simeq\\
\cong\;\;&(1_{\underline{m}^\vee})^\ast\circ C^{\underline{m}}\circ  (Li)^\simeq\\
\cong\;\;&(0_{\underline{m}^\vee})^\ast\circ C^{\underline{m}}\circ  (Li)^\simeq\\
\cong\;\;& C^{|\underline{m}|}\circ (0_{\underline{m}^\vee})^\ast\circ (Li)^\simeq.
\end{align*}
The morphism $(Li)^\simeq\colon\D^{\cube{n}_{=k}}\toiso\D^{(\cube{n},k),\cube{n}_{\leq k-1}}$ takes, in particular, values in $k$-determined $n$-cubes. Hence, all $k$-fold partial cones of $n$-cubes in the image of $(Li)^\simeq$ are constant (\autoref{cor:cotr-iter-cones-II}), and an application of this to $C^{\underline{m}}$ yields the second isomorphism in the above chain. The third isomorphism holds by the same arguments as the first one. Since $n$-cubes in the image of $(Li)^\simeq$ vanish on $\cube{n}_{\leq k-1}$, the same is true for $k$-cubes in the image of $(0_{\underline{m}^\vee})^\ast\circ (Li)^\simeq$. But $\infty\colon\bbone\to\cube{k}$ is a cosieve, and we conclude that the image of $(0_{\underline{m}^\vee})^\ast\circ (Li)^\simeq$ lies in the essential image of $\infty_!$. Thus, the counit $\infty_!\infty^\ast\to\id$ is invertible on this image, thereby inducing the first isomorphism in
\begin{align*}
&C^{|\underline{m}|}\circ (0_{\underline{m}^\vee})^\ast\circ (Li)^\simeq\\
\cong\;\;&C^{|\underline{m}|}\circ \infty_!\circ \infty^\ast\circ (0_{\underline{m}^\vee})^\ast\circ (Li)^\simeq\\
=\;\;&C^{|\underline{m}|}\circ \infty_!\circ \underline{m}^\ast\circ (Li)^\simeq\\
\cong\;\;& C^{|\underline{m}|}\circ \infty_!\circ \underline{m}^\ast\\
=\;\;&\infty^\ast\circ \cof^{\underline{1}_k}\circ\infty_!\circ\underline{m}^\ast
\end{align*}
The second isomorphism is by fully faithfulness of $i_!$. Since the $k$-cubes in the image of $\infty_!\circ\underline{m}^\ast$ vanish on $\cube{k}_{\leq k-1}$, \autoref{lem:cof-fib-restrict} implies that the image of $\cof^{\underline{1}_k}\circ\infty_!\circ\underline{m}^\ast$ consists of $0$-determined $k$-cubes, which is to say constant $k$-cubes. This gives the first isomorphism in
\begin{align*}
&\infty^\ast\circ \cof^{\underline{1}_k}\circ\infty_!\circ\underline{m}^\ast\\
\cong\;\;&0^\ast\circ \cof^{\underline{1}_k}\circ\infty_!\circ\underline{m}^\ast\\
\cong\;\;&\infty^\ast\circ\infty_!\circ\underline{m}^\ast\\
\cong\;\;& \underline{m}^\ast,
\end{align*}
the second one follows from \autoref{lem:iter-cones} and the third one since $\infty_!$ is fully faithful. To conclude the proof it suffices to put these chains of natural isomorphisms together.
\end{proof}

As a more specific application, we can now use \autoref{cor:infty-adjoints} to prove a relation between limits and colimits of chunks.

\begin{cor}\label{cor:colim-lim}
Let \D be a stable derivator, $n\geq 0$, and $0\leq k\leq l\leq n$. There are natural isomorphisms
\[
\colim_{\cube{n}_{k\leq l}}\toiso\mathrm{lim}_{\cube{n}_{k\leq l}}\circ S_{k,l}^{(n)}\colon\D^{\cube{n}_{k\leq l}}\to\D.
\]
\end{cor}
\begin{proof}
Let $i\colon\cube{n}_{k \leq l }\rightarrow \cube{n}$ be the inclusion. By uniqueness of left adjoints and homotopy finality of final objects, there is a canonical isomorphism
\[
\colim_{\cube{n}_{k\leq l}}\cong\infty^*\circ Li.
\]
We note that $\infty^*$ is the right adjoint of $Lj$ for the inclusion $j\colon\cube{n}_{=n}\rightarrow\cube{n}$. Applying \autoref{thm:Serre} to $j$ yields the relation
\[
S_{0,n}\circ Lj \cong Rj\circ S_{n,n}.
\]
By passing to right adjoints we obtain the corresponding invertible total mate 
\[
Lj[1]\circ S_{0,n}^{-1}\cong S_{n,n}^{-1}\circ Rj[1].
\]
Hence precomposition with $S_{0,n}$ and postcomposition with $S_{n,n}$ yields the relation
\[
S_{n,n}\circ Lj[1] \cong Rj[1]\circ S_{0,n},
\]
which is to say that the square on the right in the diagram
\[
\xymatrix{
\D^{\cube{n}_{k\leq l}} \ar[r]^{Li} \ar[d]^{S_{k,l}} & \D^{\cube{n}} \ar[r]^{Lj[1]} \ar[d]^{S_{0,n}} & \D \ar[d]^{S_{n,n}}\\
\D^{\cube{n}_{k\leq l}} \ar[r]^{Ri} & \D^{\cube{n}} \ar[r]^{Rj[1]} & \D
}
\]
commutes up to an invertible natural transformation. The same is true for the left square by an application of \autoref{thm:Serre} to $i$.
We observe that $\lim_{\cube{n}_{k\leq l}}$ is the double right adjoint of $\colim_{\cube{n}_{k\leq l}}$. By an application of \autoref{cor:infty-adjoints} to both subsquares of the diagram, the bottom horizontal morphisms are double right adjoint to the respective top horizontal morphism. By the above, the top row is isomorphic to $\colim_{\cube{n}_{k\leq l}}$, and the pseudo-functoriality of taking adjoints implies that the bottom row is isomorphic to $\lim_{\cube{n}_{k\leq l}}$. In order to conclude the proof we invoke $S_{n,n}\cong\id$ (\autoref{prop:discrete-Serre}).
\end{proof}

\begin{rmk}
Let us recall from \cite{gs:stable} that the defining feature of stability is that the distinction between homotopy finite limits and colimits is blurred. One way to make this precise is as follows: a derivator is stable if and only if homotopy finite colimits are \emph{weighted} limits (and there are variants using homotopy finite Kan extensions). For chunks of $n$-cubes the Serre equivalences can be used to translate these weighted limits back into actual limits. This way the blurring is made even more specific.
\end{rmk}

We illustrate \autoref{cor:colim-lim} in the case of $\cube{3}_{1\leq 2}$.

\begin{eg}
Let \D be a stable derivator and let $X\in\D^{\cube{3}_{1\leq 2}}$ be an object with underlying diagram
\[
\xymatrix{
& X_3 \ar[rr] \ar[dd] & & X_{13}\\
& & X_1 \ar[ur] \ar[dd]\\
& X_{23}\\
X_2 \ar[ur] \ar[rr] & & X_{12}.
}
\]
We can calculate the colimit of $X$ by first left extending by zero to a punctured cube and then applying the inductive formula for colimits of punctured cubes (\autoref{cor:pun-cube}). As an upshot we obtain a cocartesian square
\[
\xymatrix{
X_1\oplus X_2\ar[r]\ar[d]&X_{12}\ar[d]\\
X_{13}\cup_{X_3}X_{23}\ar[r]&\colim_{\cube{3}_{1\leq 2}}X.\pushoutcorner
}
\]
In order to describe this colimit as a limit, we calculate $\lim_{\cube{3}_{1\leq 2}}S_{1,2}^{(3)}X$. For this purpose, we add cofibers in all three coordinates to the $3$-cube $(Li)^{\simeq}(X)$ to obtain a diagram of $[2]^3$-shape (\autoref{fig:eg-lim-colim}).
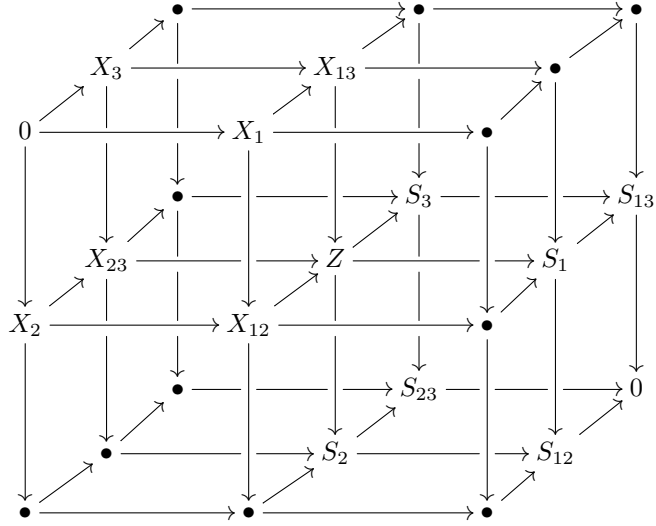
\begin{figure}
\begin{displaymath}
\xymatrix@C=1em@R=1em{
& & \bullet \ar[rrr] \ar'[d]'[dd][ddd] & & & \bullet \ar[rrr] \ar'[d]'[dd][ddd] & & & \bullet \ar[ddd]\\
& X_3 \ar[ur] \ar[rrr] \ar'[d][ddd] & & & X_{13} \ar[ur] \ar[rrr] \ar'[d][ddd] & & & \bullet \ar[ur] \ar[ddd]\\
0 \ar[ur] \ar[rrr] \ar[ddd] & & & X_1 \ar[ur] \ar[rrr] \ar[ddd] & & & \bullet \ar[ur] \ar[ddd]\\
& & \bullet \ar'[r]'[rr][rrr] \ar'[d]'[dd][ddd] & & & S_3 \ar'[r]'[rr][rrr] \ar'[d]'[dd][ddd] & & & S_{13} \ar[ddd]\\
& X_{23} \ar[ur] \ar'[rr][rrr] \ar'[d][ddd] & & & Z \ar[ur] \ar'[rr][rrr] \ar'[d][ddd] & & & S_1 \ar[ur] \ar[ddd]\\
X_2 \ar[ur] \ar[rrr] \ar[ddd] & & & X_{12} \ar[ur] \ar[rrr] \ar[ddd] & & & \bullet \ar[ur] \ar[ddd]\\
& & \bullet \ar'[r]'[rr][rrr] & & & S_{23} \ar'[r]'[rr][rrr] & & & 0\\
& \bullet \ar[ur] \ar'[rr][rrr] & & & S_2 \ar[ur] \ar'[rr][rrr] & & & S_{12} \ar[ur]\\
\bullet \ar[ur] \ar[rrr] & & & \bullet \ar[ur] \ar[rrr] & & & \bullet \ar[ur]
}
\end{displaymath}
\caption{The $[2]^3$-diagram induced by $X$}
\label{fig:eg-lim-colim}
\end{figure}
Hence $(Li)^{\simeq}(X)$ is the front upper left cube in \autoref{fig:eg-lim-colim}, whereas $\cof^{\underline{1}}(Li)^{\simeq}(X)$ is the back lower right cube. Moreover, we see that $\colim_{\cube{3}_{1\leq 2}}X$ is the final vertex of $(Li)^{\simeq}(X)$, hence the central vertex $Z$ of \autoref{fig:eg-lim-colim}. We observe that $S:=S^{(3)}_{1,2}(X)$ is by definition the restriction $i^*\cof^{\underline{1}}(Li)^{\simeq}(X)$ as indicated in \autoref{fig:eg-lim-colim}. Furthermore, we know that the back lower right cube is $2$-determined and that it vanishes at the final vertex, since the cube lies in the image of the morphism
\[
\cof^{\underline{1}}\colon\D^{(\cube{3},2),\cube{3}_{=0}}\toiso\D^{(\cube{3},2),\cube{3}_{=3}}
\]
Consequently, this cube is cartesian and it exhibits the central vertex $Z$ of \autoref{fig:eg-lim-colim} also as the limit $\lim_{\cube{3}_{1\leq 2}}\circ S^{(3)}_{1,2}(X)$. More explicitly, by the dual of \autoref{cor:pun-cube} there is a cartesian square
\[
\xymatrix{
\colim_{\cube{3}_{1\leq 2}}X\ar[r]\ar[d]\pullbackcorner&S_1\times_{S_{12}}S_2\ar[d]\\
S_3\ar[r]&S_{13}\times S_{23}.
}
\]
\end{eg}

We now turn to the pseudo-naturality of the strong stable equivalences from \autoref{defn:Phi-Psi}.

\begin{thm}\label{thm:sse-chunk-global}
Let \D be a stable derivator and $n\geq 0$. 
\begin{enumerate}
\item The right strong stable equivalences $\Phi_{k,l}=\Phi^{(n)}_{k,l}$ for $(k,l)\in\chunk{n}$ assemble to a pseudo-natural equivalence
\[
\Phi=\Phi^{(n)}\colon R(\D)\toiso Z^\vee(\D)\colon\chunk{n}\to\cDER_{\mathrm{St},\exx}.
\]
\item The left strong stable equivalences $\Psi_{k,l}=\Psi^{(n)}_{k,l}$ for $(k,l)\in\chunk{n}$ assemble to a pseudo-natural equivalence
\[
\Psi=\Psi^{(n)}\colon Z^\vee(\D)\toiso L(\D)\colon\chunk{n}\to\cDER_{\mathrm{St},\exx}.
\]
\end{enumerate}
\end{thm}

\begin{proof}
We prove only (i), the proof of (ii) is very similar, and for definitiveness we choose $H_{k,l}^{r2}$ from \autoref{con:H-H-H-H} as our local models for the components $\Phi^{(n)}_{k,l}$ of the desired pseudo-natural equivalence $\Phi^{(n)}$. For every $0\leq k\leq k'\leq l'\leq l\leq n$ we consider the morphism $j=i_{k'\leq l'}^{k\leq l}$ in $\chunk{n}$. Moreover, let us use the shorthand notation
\[
i'=i_{k'\leq l'}^{0\leq n}\colon(k',l')\to(0,n)\qquad\text{and}\qquad i=i_{k\leq l}^{0\leq n}\colon(k,l)\to(0,n),
\]
so that the relation $i'=i\circ j$ holds in $\chunk{n}$. Unraveling definitions, we have to show that \autoref{fig:Phi-chunks} commutes up to coherent natural isomorphisms. The upper two squares are dealt with as in the case of Serre equivalences, and in the case of the bottom square it suffices to invoke the pseudo-functoriality constraints of $Z^\vee$.
\end{proof}

\begin{figure}
\begin{displaymath}
\xymatrix{
\D^{\cube{n}_{k' \leq l'}} \ar[d]_{(Ri')^\simeq} \ar[r]^{Rj} & \D^{\cube{n}_{k \leq l}} \ar[d]^{(Ri)^\simeq}\\
\D^{(\cube{n},n-k'),\cube{n}_{\geq l'+1}} \ar[d]_{\cof^{\underline{1}}} \ar[r]^-\subseteq & \D^{(\cube{n},n-k),\cube{n}_{\geq l+1}} \ar[d]^{\cof^{\underline{1}}}\\
\D^{\cube{n},\cube{n}_{0\leq n-l'-1}\cup\cube{n}_{n-k'+1\leq n}} \ar[r]_-\subseteq & \D^{\cube{n},\cube{n}_{0\leq n-l-1}\cup\cube{n}_{n-k+1\leq n}} \\
\D^{\cube{n}_{n-l' \leq n-k'}}\ar[u]^{Z((i')^\vee)^\simeq} \ar[r]_{Z(j^\vee)} & \D^{\cube{n}_{n-l \leq n-k}}\ar[u]_{Z(i^\vee)^\simeq}
}
\end{displaymath}
\caption{Global standard equivalence $\Phi^{(n)}$ for $\chunk{n}$}
\label{fig:Phi-chunks}
\end{figure}

\begin{rmk}\label{rmk:duality}
Let \D be a stable derivator and $0\leq n$. For every chunk $\cube{n}_{k\leq l}$ we can invoke the remaining strong stable equivalences $H^{l2}_{k,l}$ (resp. $H^{r1}_{k,l}$).
Arguments dual to the ones in the proof of \autoref{thm:sse-chunk-global} show that these assemble to a pseudo-natural equivalence
\[
\tilde{\Psi}=\tilde{\Psi}^{(n)}\colon R^{\vee}(\D)\toiso Z(\D)\colon\chunk{n}\to\cDER_{\mathrm{St},\exx},
\]
resp.
\[
\tilde{\Phi}=\tilde{\Phi}^{(n)}\colon Z(\D)\toiso L^{\vee}(\D)\colon\chunk{n}\to\cDER_{\mathrm{St},\exx}.
\]
We invoke \autoref{cor:dodecagon} again to see that
\[
\tilde{\Psi}=\Phi^{\vee}\qquad \text{ and } \qquad\tilde{\Phi}=\Psi^{\vee}.
\]
\end{rmk}

\begin{cor}\label{cor:modif}
There are invertible modifications:
\[
\Phi\circ S \circ \Psi \toiso \Sigma^n\colon id_{Z^{\vee}(\D)} \rightarrow id_{Z^{\vee}(\D)} \text{ and } \tilde{\Psi}\circ\tilde{S} \circ\tilde{\Phi} \toiso \Sigma^n\colon id_{Z(\D)} \rightarrow id_{Z(\D)}
\]
\end{cor}
\begin{proof}
This is an immediate consequence from the corresponding local statements (\autoref{cor:modif-local}).
\end{proof}

\begin{rmk}
Let \D be a stable derivator, $n\geq 0$, and $0\leq l\leq n$ (the case of $k=0$). We have already seen that $\D^{\cube{n}_{\leq l}}$ and $\D^{\cube{n}_{\geq n-l}}$ are both equivalent to the derivator $\D^{\cube{n},l}$ of $l$-determined $n$-cubes (\autoref{thm:sse-chunks-special-case}). Moreover, these equivalences are induced by left Kan extension along $\iota_{\leq l}$ and right Kan extension along $\iota_{\geq n-l}$. More explicitly, in this special case the two circles of \autoref{fig:chunks} collapse to a single one and the diagram simplifies to:
\[
\xymatrix@C=1.8em{
& & & \D^{\cube{n},l} \ar@/_2pc/[ldd]^{\cof^{\underline{1}}}\\
\D^{\cube{n}_{\leq l}} \ar@/^1pc/[rrru]^{(\iota_{\leq l})_!^{\simeq}} \ar@/_1pc/[rrd]^{(\iota_{\leq l})_*^{\simeq}} & & & & & & \D^{\cube{n}_{\geq n-l}} \ar@/_1pc/[lllu]_{(\iota_{\geq n-l})_*^{\simeq}} \ar@/^1pc/[lld]_{(\iota_{\geq n-l})_!^{\simeq}}\\
& & \D^{\cube{n},\cube{n}_{\geq l+1}} \ar@/_2pc/[rr]^{\cof^{\underline{1}}} & & \D^{\cube{n},\cube{n}_{\geq l+1}} \ar@/_2pc/[uul]^{\cof^{\underline{1}}}.
}
\]
Hence, in contrast to the situation of general chunks, here we obtain a natural choice of a strong stable equivalence $\D^{\cube{n}_{\leq l}}\toiso\D^{\cube{n}_{\geq n-l}}$, namely by composing the two previous equivalences. As a consequence we can explicitly describe the adjoint quintuple
\[
\iota_{k\leq l}[-2] \dashv \iota_{k\leq l}[-1] \dashv \iota_{k\leq l} \dashv \iota_{k\leq l}[1] \dashv \iota_{k\leq l}[2]
\]
which is associated to the natural inclusion $\iota_{k\leq l}: \D^{\cube{n},k}\rightarrow\D^{\cube{n},l}$. The adjoint triple $\iota_{k\leq l}[-2] \dashv \iota_{k\leq l}[-1] \dashv \iota_{k\leq l}$ is induced by $(\iota_{n-l,n-k})_! \dashv (\iota_{n-l,n-k})^* \dashv (\iota_{n-l,n-k})_*$ whereas $\iota_{k\leq l} \dashv \iota_{k\leq l}[1] \dashv \iota_{k\leq l}[2]$ is induced by $(\iota_{l,k})_! \dashv (\iota_{l,k})^* \dashv (\iota_{l,k})_*$. Therefore an application of \autoref{thm:Serre} to either
\[
\iota_{l,k}\colon\cube{n}_{\leq k}\rightarrow \cube{n}_{\leq l}\qquad\text{or}\qquad \iota_{n-l,n-k}:\cube{n}_{\geq n-l}\rightarrow\cube{n}_{\geq n-k}
\]
yields a description of the infinite chain of adjunctions generated by $\iota_{k\leq l}$. These are dual to each other but not symmetric by themselves. Alternatively, we can apply \autoref{thm:Serre} to both inclusions of chunks simultaneously. More explicitly, one can show that the diagram
\[
\xymatrix{
\D^{\cube{n},k} \ar[r]^{F_k} \ar[d]_{\iota_{k\leq l}[-2]} & \D^{\cube{n},k} \ar[d]^{\iota_{k\leq l}[2]}\\
\D^{\cube{n},l} \ar[r]^{F_l} & \D^{\cube{n},l}
}
\]
commutes, where $F_k\colon \D^{\cube{n},k}\rightarrow\D^{\cube{n},k}$ is the composition
\[
\xymatrix@C=1.8em{
\D^{\cube{n},k} \ar[r]_{((\iota_{\geq n-k})^{\simeq}_*)^{-1}} & \D^{\cube{n}_{\geq n-k}} \ar[r]_{(\iota_{\geq n-k})^{\simeq}_!} & \D^{\cube{n},\cube{n}_{\leq n-k-1}} \ar[r]_{\mathsf{fib}^{\underline{1}}} & \D^{\cube{n},\cube{n}_{\geq k+1}} \ar[r]_{((\iota_{\leq k})^{\simeq}_*)^{-1}} & \D^{\cube{n}_{\leq k}} \ar[r]_{(\iota_{\leq k})^{\simeq}_!} & \D^{\cube{n},k}.
}
\]
This yields a symmetric description of the infinite chain of adjunctions, which moreover is $4$-periodic, since the morphisms $F_k$ are strongly related to compositions of two iterations of Serre equivalences.
\end{rmk}

\begin{eg}
We now provide some explicit computations of iterated adjoints of morphisms of the form $\iota_{k\leq l}$ in specific cases where we have the fractionally Calabi--Yau property.
\begin{enumerate}
\item
By \autoref{cor:cof-S}, \autoref{eg:products-A2} and \autoref{eg:A3} we can identify
\[
(S^{(2)}_{0,2})^3=\Sigma^2 \qquad \text{ and } \qquad (S^{(2)}_{0,1})^4=\Sigma^2.
\]
Hence, we can compute
\[
\iota_{1\leq 2}^{(2)}[24]=(S^{(2)}_{0,2})^{12}\iota_{1\leq 2}^{(2)}(S^{(2)}_{0,1})^{-12}=\Sigma^8\iota_{1\leq 2}^{(2)}\Sigma^{-6}=\Sigma^2\iota_{1\leq 2}^{(2)}
\]
where the first equality is \autoref{cor:infty-adjoints}.
\item Similarly, invoking additionally \autoref{eg:source3}, we have the identifications
\[
(S^{(3)}_{0,3})^3=\Sigma^3 \qquad \text{ and } \qquad (S^{(3)}_{0,1})^3=\Sigma^2
\]
leading to the computation
\[
\iota_{1\leq 3}^{(3)}[6]=(S^{(3)}_{0,3})^{3}\iota_{1\leq 3}^{(3)}(S^{(3)}_{0,1})^{-3}=\Sigma^3\iota_{1\leq 3}^{(3)}\Sigma^{-2}=\Sigma\iota_{1\leq 3}^{(3)}.
\]
\end{enumerate}
\end{eg}

\section{Universal formulas}
\label{sec:formulas}

In this concluding section we show that the above calculus of $n$-cubes is compatible with morphisms of derivators. More generally, we show that morphisms of derivators enjoy a ``lax-oplax'' compatibility with canonical mates. Essentially by forming compositions of invertible such mates and their inverses, we define formulas relative to a chosen $2$-category $\mathscr{K}$ of derivators, thereby axiomatizing certain natural isomorphisms which exist in all derivators $\D\in\mathscr{K}$ and which are preserved by all morphisms in $\mathscr{K}$. As an important special case we obtain the class of exact formulas in stable derivators, and the cubical calculus provides examples of such formulas. As a related result we show that suitable formulas propagate from a monoidal derivator $\V$ to all $\V$-modules, thereby making precise a universality of formulas in monoidal derivators.

The key to the results in this section is the following, fairly elementary result on the compatibility of the calculus of mates. Given a natural transformation $\alpha$ living in a square of small categories,
\begin{equation}\label{eq:mates-vs-mor}
\vcenter{
\xymatrix{
C\ar[r]^-p\ar[d]_v \drtwocell\omit{\alpha} & A\ar[d]^u\\
D\ar[r]_-q & B,
}
}
\end{equation}
in every derivator \D there is the canonical mate
\[
\alpha_!\colon v_! p^\ast\to q^\ast u_!.
\]
Moreover, associated to every morphism $F\colon\D\to\E$ of derivators with pseudo-naturality constraints $\gamma$ there are the canonical mates
\[
v_!F\to Fv_!\qquad\text{and}\qquad u_!F\to Fu_!.
\]
These various canonical mates are compatible in the following precise sense.

\begin{prop}\label{prop:mates-vs-morphisms}
For every morphism of derivators $F\colon\D\to\E$ and every natural transformation \eqref{eq:mates-vs-mor} in $\cCat$ the following diagram commutes, 
\begin{equation}\label{eq:mates-vs-morphisms}
\vcenter{
\xymatrix{
Fv_! p^\ast\ar[r]^-{\alpha_!}&Fq^\ast u_!\\
v_!Fp^\ast\ar[u]&q^\ast F u_!\ar[u]_-\gamma^-\cong\\
v_!p^\ast F\ar[u]^-\gamma_-\cong\ar[r]_-{\alpha_!}&q^\ast u_!F.\ar[u]
}
}
\end{equation}
\end{prop}
\begin{proof}
Unraveling definitions we have to show that the clockwise and the counterclockwise boundary paths from $v_!p^\ast F$ to $Fq^\ast u_!$ in \autoref{fig:mates}
\begin{figure}
\begin{equation}
\vcenter{
\xymatrix{
Fv_! p^\ast\ar[r]^-\eta\ar@{}[rd]|{=}&Fv_!p^\ast u^\ast u_!\ar[r]^-{\alpha^\ast}\ar@{}[rd]|{=}&Fv_!v^\ast q^\ast u_!\ar[ddr]^-\varepsilon&\\
v_!Fp^\ast\ar[u]\ar[r]^-\eta\ar@{}[rd]|{=}&v_!Fp^\ast u^\ast u_!\ar[u]\ar[r]^-{\alpha^\ast}&v_!Fv^\ast q^\ast u_!\ar[u]&\\
v_!p^\ast F\ar[u]^-\gamma_-\cong\ar[r]^-\eta\ar[ddr]_-\eta&v_!p^\ast F u^\ast u_!\ar[u]^-\gamma_-\cong&v_!v^\ast F q^\ast u_!\ar[u]_-\gamma^-\cong\ar[r]_-\varepsilon\ar@{}[rd]|{=}&Fq^\ast u_!\\
&v_!p^\ast u^\ast F u_!\ar[u]^-\gamma_-\cong\ar[r]_-{\alpha^\ast}\ar@{}[rd]|{=}&v_!v^\ast q^\ast F u_!\ar[u]_-\gamma^-\cong\ar[r]_-\varepsilon\ar@{}[rd]|{=}&q^\ast F u_!\ar[u]_-\gamma^-\cong\\
&v_!p^\ast u^\ast u_! F\ar[u]\ar[r]_-{\alpha^\ast} &v_!v^\ast q^\ast u_! F\ar[u]\ar[r]_-\varepsilon&q^\ast u_!F\ar[u]
}
}
\end{equation}
\caption{Morphisms and canonical mates}
\label{fig:mates}
\end{figure}
coincide. To this end it suffices to show that this diagram commutes. In this diagram, the six rectangles decorated by an equality sign commute since they are naturality squares. The larger rectangle in the middle is commutative by one of the coherence properties of pseudo-natural transformations. Finally, the remaining two quadrilaterals also commute by \cite[Lem.~3.10]{groth:revisit}.
\end{proof}

The remainder of this section is essentially an exploration of immediate implications of this proposition. It is possible to formalize the results more systematically, but we prefer to present them in a way that allows us to focus more easily on the simple underlying ideas.

\begin{rmk}
Let $F\colon\D\to\E$ be a morphism of derivators and let \eqref{eq:mates-vs-mor} be a natural transformation in $\cCat$.
\begin{enumerate}
\item The duality principle for derivators yields a variant of the proposition for right Kan extensions.
\item The commutative diagram \eqref{eq:mates-vs-morphisms} makes precise a compatibility statement for certain canonical mates. In the special case that $F$ is cocontinuous and that the square \eqref{eq:mates-vs-mor} is homotopy exact the above diagram consists of natural isomorphisms only.
\end{enumerate}
\end{rmk}

The formalism of homotopy exact squares is the technical key tool in the theory of derivators. In fact, in many cases developing the calculus of (co)limits and Kan extensions in derivators essentially amounts to establishing additional classes of homotopy exact squares. 

\begin{rmk}
There are interesting \emph{relative} versions of the second point of the previous remark, and in the background of these variants there are the following two Galois connections.
\begin{enumerate}
\item Let \D be a derivator and let \eqref{eq:mates-vs-mor} be a natural transformation in $\cCat$. Following Maltsiniotis \cite{maltsiniotis:htpy-exact}, we say that \eqref{eq:mates-vs-mor} is \D-\emph{exact} if the canonical mates
\begin{equation}\label{eq:mates-exact}
\alpha_!\colon v_!p^\ast\to q^\ast u_!\qquad\text{and}\qquad\alpha_\ast\colon u^\ast q_\ast\to p_\ast v^\ast
\end{equation}
are invertible in \D. We recall that these mates are conjugate to each other so that one of them is an isomorphism if and only if the other is. This notion extends to a Galois connection between classes of derivators and classes of squares in $\cCat$ populated by (possibly non-invertible) natural transformations as follows.
\begin{enumerate}
\item Given a class $\mathcal{C}$ of derivators, we say that the square \eqref{eq:mates-vs-mor} is \textbf{$\mathcal{C}$-exact} if it is \D-exact for all $\D\in\cC$.
\item Conversely, given a class $\mathcal{S}$ of squares \eqref{eq:mates-vs-mor} in $\cCat$, we associate to it the class of all derivators \D in which the canonical mates \eqref{eq:mates-exact} are invertible.
\end{enumerate}
\item The second Galois connection is closely related to the discussion of \emph{saturation} or \emph{closure} of classes of colimits \cite{albert-kelly:closure,kelly-pare:closure}.
\begin{enumerate}
\item Given a class $\mathcal{M}$ of morphisms of derivators, we associate to it the class $u\colon A\to B$ of functors between small categories such that every morphism $F\in\mathcal{M}$ preserves left Kan extensions along $u$ (or, equivalently, such that every $F$ preserves colimits of shape $(u/b)$ for all $b\in B$).
\item Conversely, given a class $\mathcal{F}$ of functors in $\cCat$, we associate to it the class of \textbf{$\mathcal{F}$-cocontinuous} morphisms of derivators, i.e., of those morphisms which preserve left Kan extensions along all $u\in\mathcal{F}$.
\end{enumerate}
There is an obvious dual version of this Galois connection using right Kan extensions instead.
\end{enumerate}
\end{rmk}

\begin{con}
Let $\mathscr{K}\subseteq\cDER$ be a sub-$2$-category of the $2$-category of derivators. The following basic building blocks will be used to define formulas in $\mathscr{K}$.
\begin{enumerate}
\item Let \eqref{eq:mates-vs-mor} be $\mathscr{K}$-exact ($\D$-exact for all $\D\in\mathscr{K}$) and such that all morphisms in $\mathscr{K}$ preserve left Kan extensions along the vertical functors $u,v$ in \eqref{eq:mates-vs-mor}. In every derivator $\D\in\mathscr{K}$ the canonical mate
\begin{equation}\label{eq:con-formula-L}
\alpha_!\colon v_!p^\ast\toiso q^\ast u_!
\end{equation}
is invertible and these isomorphisms are compatible with morphisms in $\mathscr{K}$ in the sense of \eqref{eq:mates-vs-morphisms}.
\item Let \eqref{eq:mates-vs-mor} be $\mathscr{K}$-exact and such that all morphisms in $\mathscr{K}$ preserve right Kan extensions along the horizontal functors $p,q$ in \eqref{eq:mates-vs-mor}. In every derivator $\D\in\mathscr{K}$ the canonical mate
\begin{equation}\label{eq:con-formula-R}
\alpha_\ast\colon u^\ast q_\ast\toiso p_\ast v^\ast
\end{equation}
is invertible and compatible with morphisms in $\mathscr{K}$.
\end{enumerate}
\end{con}

\begin{defn}
Let $\mathscr{K}\subseteq\cDER$ be a sub-$2$-category of the $2$-category of derivators. A \textbf{formula} in $\mathscr{K}$ is a natural isomorphism which can be written as a finite composition of whiskerings along restriction morphisms of natural isomorphisms of the form \eqref{eq:con-formula-L}, \eqref{eq:con-formula-R} or their respective inverses. 
\end{defn}

\begin{eg}
A \textbf{cocontinuous formula} is a formula in the $2$-category $\cDER_{\mathrm{cc}}$ of derivators, cocontinuous morphisms, and all natural transformations. For instance the isomorphisms $\colim_{\Delta\op}sX\toiso\colim_AX$ provided by the Bousfield--Kan formulas (\autoref{thm:BK-formulas}) are cocontinuous formulas as are the decomposition isomorphisms from \autoref{thm:decompose}. More specific examples are the inductive formulas for colimits of punctured $n$-cubes (\autoref{cor:pun-cube}) and the various isomorphisms in \autoref{eg:colim-of-sources} to calculate colimits of sources of valence $n$ (these latter two are \emph{right exact formulas} for derivators in the obvious sense).
\end{eg}

The formalism of formulas gets richer when we mix left and right Kan extensions more seriously. This is illustrated by the following two examples, and it might be worth to study variants of these as well. We recall that right exact morphisms of pointed derivators preserve right extensions by zero \cite[Cor.~8.2]{groth:revisit}), and we hence obtain \autoref{eg:formulas-pt-rex}. 

\begin{eg}\label{eg:formulas-pt-rex}
Formulas in the $2$-category $\cDER_{\mathrm{Pt},\mathrm{rex}}$ of pointed derivators, right exact morphisms, and all natural transformations are \textbf{right exact formula for pointed derivators}. For every pointed derivator \D and $n\geq 0$,
\begin{enumerate}
\item the isomorphisms $\cof^3\cong\Sigma\colon\D^{[1]}\to\D^{[1]}$ (\cite[Lem.~5.13]{gps:mayer}),
\item the isomorphisms $\Sigma^{n-1}\cong\colim_{\cube{n}_{\leq n-1}}\circ\emptyset_\ast\colon\D\to\D$ (\autoref{eg:higher-susp}),
\item the coherence isomorphisms of the pseudo-actions $\cof^\bullet$ (\autoref{prop:cof-comp}) and their variants in \autoref{cor:C-comp}, and
\item the isomorphisms $\tcof\cong C^n\colon\D^{\cube{n}}\to\D$ (\autoref{thm:total-cof} and \autoref{prop:tcof-cof1})
\end{enumerate}
are instances of such formulas, and all morphisms in $\cDER_{\mathrm{Pt},\mathrm{rex}}$ are hence compatible with these isomorphisms.
\end{eg}

\begin{rmk}
The observation that $\cof^3\cong\Sigma$ is a right exact formula for pointed derivators is the essential ingredient in the proof that exact morphisms of strong, stable derivators induce exact functors between canonical triangulations (see the proof of \cite[Theorem~10.6]{groth:revisit}). A direct verification of this is already a bit cumbersome as witnessed by the proof in \emph{loc.~cit.}, and the formalization of formulas makes such verifications obsolete (see also \cite[Rmk.~10.8]{groth:revisit}).
\end{rmk}

We now turn to a particularly rich class of formulas.
 
\begin{egs}
An \textbf{exact formula for stable derivators} is a formula in the $2$-category $\cDER_{\mathrm{St},\exx}$ of stable derivators, exact morphisms, and natural transformations. Of course, right exact and left exact formulas for derivators or pointed derivators (as in \autoref{eg:formulas-pt-rex}) restrict to exact formulas for stable derivators. For every stable derivator \D and $n\geq 0$,
\begin{enumerate}
\item the isomorphism $\Sigma\circ F\toiso C\colon\D^{[1]}\to\D$ \eqref{eq:SFC},
\item the isomorphisms $\Sigma^n\circ\tfib\toiso\tcof\colon\D^{\cube{n}}\to\D$ (\autoref{rmk:SFCn}),
\item the pseudo-naturality constraints $S_{k,l}\circ i_!\cong i_\ast\circ S_{k',l'}$ of the global Serre equivalence (\autoref{thm:Serre}), and,
\item as a special case, the isomorphisms $\colim_{\cube{n}_{k\leq l}}\toiso\mathrm{lim}_{\cube{n}_{k\leq l}}\circ S_{k,l}^{(n)}\colon\D^{\cube{n}_{k\leq l}}\to\D$ from \autoref{cor:colim-lim}
\end{enumerate}
are exact formulas for stable derivators.
\end{egs}

There is an additional perspective on \autoref{prop:mates-vs-morphisms}. Let \D be a derivator and let the square \eqref{eq:mates-vs-mor} be \D-exact. Given a sufficiently cocontinuous morphism $F\colon\D\to\E$ we know that the canonical mate of \eqref{eq:mates-vs-mor} is invertible on all objects in the essential image of $F$. As an application of this, we can extend formulas from a monoidal derivator \V to \V-enriched derivator. This works more generally for \textbf{cocontinuous \V-modules}, i.e., for derivators \D which are endowed with an associative and unital action $\otimes\colon\V\times\D\to\D$ which preserves colimits in both variables independently.

\begin{prop}\label{prop:propagate}
Let \V be a monoidal derivator. If the square \eqref{eq:mates-vs-mor} is \V-exact, then it is \D-exact for all cocontinuous \V-modules \D (in particular, for all \V-enriched derivators \D).
\end{prop}
\begin{proof}
For every $A\in\cCat$ the action $\otimes\colon\V\times\D\to\D$ gives rise to the canceling tensor product 
\[
\otimes_{[A]}\colon\V^{A\times A\op}\times\D^A\to\D^A,
\]
which is obtained from the pointwise product by means of a coend (\cite[\S8]{gps:additivity}). Moreover, denoting by $\lI_A\in\V(A\times A\op)$ the identity profunctor, for $X\in\D(A)$ there is by \cite[Thm.~5.9]{gps:additivity} an isomorphism
\[
\lI_A\otimes_{[A]}X\cong X.
\]
Put differently, $X\in\D(A)$ lies in the essential image of the canceling tensor product morphism
\[
-\otimes_{[A]}X\colon\V^{A\times A\op}\to\D^A.
\]
Moreover, since $\otimes$ preserves colimits in both variables separately, the partial morphism $-\otimes_{[A]}X$ is cocontinuous. Since the square \eqref{eq:mates-vs-mor} is \V-exact, we can apply \autoref{prop:mates-vs-morphisms} to the morphisms $-\otimes_{[A]}X$ to conclude that the square is also \D-exact.
\end{proof}

\begin{rmk}
This result is closely related to the construction of universal tilting modules in abstract representation theory \cite[\S10]{gst:Dynkin-A}.
\end{rmk}

\begin{eg}
Every derivator is a cocontinuous module over the derivator $\cS$ of spaces \cite{cisinski:derived-kan}. As a consequence, \autoref{prop:propagate} implies that $\cS$-exact squares are homotopy exact and that cocontinuous formulas extend from the derivator of spaces to arbitrary derivators. 
\end{eg}

The fact that $\cS$-exact squares are homotopy exact is not new (see \cite[\S9]{gps:mayer} which relies heavily on \cite{heller:htpythies,cisinski:presheaves}).

\begin{eg}
Every pointed derivator is a cocontinuous module over the derivator $\mathcal{S}_\ast$ of pointed spaces \cite{cisinski:derived-kan}, and we conclude that cocontinuous, pointed formulas or right exact, pointed formulas propagate from the derivator of pointed spaces to arbitrary pointed derivators.
\end{eg}

\begin{eg}
Every stable derivator is enriched over the derivator $\cSp$ of spectra \cite{heller:htpythies,heller:stable,franke:adams,cisinski:derived-kan,tabuada:universal-invariants,cisinski-tabuada:non-connective,cisinski-tabuada:non-commutative}. In particular, an exact formula holds in all stable derivators as soon it is true in spectra.
\end{eg}

\begin{rmk}
There are interesting additional variants to be studied further, such as stable derivators that admit an action of the derivator $\D_\lZ$ of the integers or of the derivator $\D_k$ of a field~$k$. These are analogues in derivator land of \emph{algebraic} or \emph{$k$-linear algebraic} triangulated categories \cite{keller:dg-categories,schwede:alg-versus-top}, and in those cases formulas can be propagated from $\D_\lZ$ or $\D_k$ to the corresponding classes of stable derivators. Of course, by \autoref{prop:propagate} for every square \eqref{eq:mates-vs-mor} there are the implications
\begin{align*}
& \eqref{eq:mates-vs-mor}\text{ is }\cSp\text{-exact}\\
\Rightarrow\quad & \eqref{eq:mates-vs-mor}\text{ is }\D_\lZ\text{-exact}\\
\Rightarrow\quad & \eqref{eq:mates-vs-mor}\text{ is }\D_k\text{-exact},
\end{align*}
and it would be interesting to study to which extent the converse implications fail.
\end{rmk}

We conclude this paper by the following closely related remark.

\begin{rmk}
By \cite[Cor.~4.34]{maltsiniotis:htpy-exact} a square \eqref{eq:mates-vs-mor} is $y_{\mathrm{Set}}$-exact if and only if it is $y_{\cC}$-exact for all bicomplete categories $\cC$. However, clearly, not every $y_{\mathrm{Set}}$-exact square is homotopy exact. For instance, the decomposition theorem (\autoref{thm:decompose}) admits a variant for represented derivators. In that case, one can consider more general left decompositions of categories which are defined by means of $y_{\mathrm{Set}}$-exact squares, which is to say by \emph{final} functors instead of \emph{homotopy final} functors. As a specific case, in represented derivators the Bousfield--Kan formulas (\autoref{thm:BK-formulas}) can be simplified further to the well-known coequalizers
\[
\xymatrix{
\displaystyle\coprod_{a_0\to a_1}X_{a_0}\ar@<1mm>[r] \ar@<-1mm>[r] &
    \displaystyle\coprod_a^{\vphantom{a_n}} X_a
}
\]
calculating colimits of diagrams of shape~$A$, while this does not work in arbitrary derivators. The reason for this is that while the inclusion $\Delta\op_{\mathrm{inj}}\to\Delta\op$ of the wide subcategory of monomorphisms is homotopy final, the inclusion $\Delta\op_{\mathrm{inj},\leq 1}\to\Delta\op_{\mathrm{inj}}$ of the full subcategory spanned by $[0],[1]$ (hence the Kronecker quiver) is final but not homotopy final.

Similarly, coends in derivators are defined by means of twisted morphism categories \cite[\S5]{gps:additivity}. The two-sided bar construction yields a cocontinuous formula to calculate coends by simplicial resolutions \cite[Appendix~A]{gps:additivity}. For represented derivators these formulas can be simplified further by means of subdivision categories \cite[\S~IX.5]{maclane}, but this does not extend to arbitrary derivators.
\end{rmk}

\appendix

\bibliographystyle{alpha}
\bibliography{cubical}

\end{document}